\documentclass[11pt]{article}
\usepackage{amsmath,amssymb,amsthm,fancyhdr,graphicx}

\title{Two generator groups acting on the complex hyperbolic plane.}
\date{}
\author{Pierre WILL\\
Institut Fourier\\
100 rue des Maths\\
38402 Saint Martin d'H\`eres\\
FRANCE}

\newcommand{\la}{\langle}
\newcommand{\ra}{\rangle}
\newcommand{\HdC}{ \mbox{{\bf H}}^{2}_{\mathbb C}}
\newcommand{\HdR}{\bold H^{2}_{\mathbb R}}

\newcommand{\HuC}{\bold H^{1}_{\mathbb C}}
\newcommand{\HnC}{\bold H^{n}_{\mathbb C}}

\newcommand{\C}{\mathbb C}

\newcommand{\R}{\mathbb R}
\newcommand{\Q}{\mathbb Q}
\newcommand{\Z}{\mathbb Z}
\newcommand{\A}{\mathbb A}

\newcommand{\n}{\noindent}

\newcommand{\PSL}{\mbox{\rm PSL(2,$\R$)}}

\newcommand{\X}{\mathbf{X}}
\newcommand{\tr}{\mbox{\rm tr}}

\newcommand{\bp}{{\bf p}}
\newcommand{\bm}{{\bf m}}

\newcommand{\bn}{{\bf n}}
\newcommand{\bq}{{\bf q}}
\newcommand{\bz}{{\bf z}}
\newcommand{\bv}{{\bf v}}

\newcommand{\bA}{{\bf A}}
\newcommand{\bB}{{\bf B}}
\newcommand{\bC}{{\bf C}}
\newcommand{\bE}{{\bf E}}
\newcommand{\bP}{{\bf P}}
\renewcommand{\arg}{\mbox{arg}}
\renewcommand{\Re}{\mbox{\rm Re}\,}
\renewcommand{\Im}{\mbox{\rm Im}\,}
\renewcommand{\geq}{\geqslant}
\renewcommand{\leq}{\leqslant}

\newtheorem{theorem}{Theorem}[section]
\newtheorem*{theo*}{Theorem}
\newtheorem{lemma}{Lemma}[section]
\newtheorem{corollary}{Corollary}[section]
\newtheorem{proposition}{Proposition}[section]

\theoremstyle{remark}
\newtheorem{remark}{Remark}[section]

\theoremstyle{definition}
\newtheorem{definition}{Definition}[section]

\begin{document}

\title{Two-generator groups acting on the complex hyperbolic plane}

\author{Pierre Will\thanks{
Work partially supported by ANR Project SGT} \\
Institut Fourier, Universit\'e
Joseph Fourier\\
100 rue des maths, 38402 Saint Martin d'H\`eres, France\\
email:\,\tt{pierre.will@ujf-grenoble.fr}
}

\maketitle

\begin{abstract}
 This is an expository article about groups generated by two isometries of the complex hyperbolic plane.
\end{abstract}



\tableofcontents   

\section{Introduction}
Discrete groups isometries of the complex hyperbolic $n$-space ($n\leqslant 2$) are natural generalisations of Fuchsian 
groups from the context of the Poincar\'e disc to the one of the complex unit ball in $\C^n$. They are far from having been
studied as much as their cousins from the real hyperbolic space. The first works in that direction go back to the end of
the nineteeth century, with works of Picard for instance. Between that moment and the 1970's the subject has not been very 
active, in spite of works of Giraud around 1920 and E. Cartan in the 1930's. The subject was brought back into light in the late 
1970's by Mostow's interest to it and his article \cite{Most}, related to the question  of arithmeticity of lattices in 
symmetric spaces. During the 1980's Goldman and Millson adressed the question of the deformations of lattices from PU($n$,1) 
to PU($n+1$,1), and proved their local rigidity theorem (see \cite{GM}). In this article, we will restrict ourselves to the frame 
of the complex hyperbolic plane, that is when $n=2$.

One of the first problems one encounters is to be able to produce representative examples of discrete subgroups in PU(2,1). This 
question is related for instance to the construction of polyhedra, that arise as fundamental domains for discrete groups. The 
construction of a polyhedron is made difficult by the fact that no totally geodesic hypersurfaces exist in $\HdC$ (indeed, the 
complex hyperbolic space has non-constant negative curvature). Under the influence of Goldman and then Falbel, Parker and Schwartz 
among others, methods to overcome that difficulty have been developed since the early 1990's (see for instance \cite{GP}), and the 
collection of known examples of discrete subgroups of PU($2$,1) has expanded. However, a general theory for these groups is still 
not known.

The goal of this article is to present a few results and methods used in the study of subgroups of PU(2,1), through the scope of 
2-generator subgroups, that is representations of the rank 2 free group $F_2$ to PU(2,1). The reason for this choice is that most
of the known examples are in fact related to 2-generator subgroups, mostly because they are relatively easy to describe algebraically. 

The main general reference on complex hyperbolic geometry is Bill Goldman's book \cite{Go}. In the article \cite{CG}, complex 
hyperbolic space appears as one the the possible hyperbolic spaces over various fields. A lot of information can also be found in
Richard Schwartz's monograph \cite{Schbook}. The book to come \cite{Parbook} by John Parker will provide another introductory source 
soon. Other expository articles on various aspects of complex hyperbolic geometry have been published. Among them, \cite{PP2} 
discusses the question of the complex hyperbolic quasi-Fuchsian groups, \cite{Parklattices} in concerned with lattices, and \cite{S2} 
discusses triangle groups. I have tried to do something different from these articles, but they have been a source of inspiration.
All the results presented here have already been published elsewhere, and I have tried to give as precise references as I could. 
Two exceptions : Theorem \ref{theo-identity} and Theorem  \ref{pairexists} are, to my knowledge, new.

This article is organised as follows. In section \ref{presentation}, I present a few basic facts on the complex hyperbolic 
space. In section \ref{invariants}, projective invariants are exposed, like cross-ratios or triple ratios, and I expose a few 
results connecting these invariants to eigenvalues of matrices in SU(2,1). Section \ref{section-conjug} is devoted to the 
classification of pairs of matrices in SU(2,1) by traces. In section \ref{section-constraints}, I describe some constraints 
(or their absence) on conjugacy classes. Section \ref{section-discrete} is concerned with the question of discreteness of a 
subgroup of PU(2,1). The last two sections \ref{section-triangle-groups} and \ref{section-example-modular} are devoted to examples : triangle 
groups and representations of the modular group, and related questions.\\


\section{The complex hyperbolic space and its isometries\label{presentation}}
\subsection{Projective models for the complex hyperbolic plane}
Projective models for the complex hyperbolic plane are obtained by projecting to $\C P^2$ 
the negative cone of a Hermitian form of signature (2,1) on $\C^3$.
\begin{definition}
Let $H$ be a Hermitian form of signature (2,1) on $\C^3$. The
projective model of $\HdC$ associated to $H$ is $P(V^-)$, where
$V^-=\lbrace v\in\C^3, H(v,v)<0\rbrace$, equipped with the distance
function $d$ given by
\begin{equation}\label{distance}
\cosh^2\left(\dfrac{d(m,n)}{2}\right)=\dfrac{H(\bm,\bn)H(\bn,\bm)}{H(\bm,\bm)H(\bn,\bn)}
\end{equation}
\end{definition}
Among the most frequently used such models are the ball and the Siegel model,
which are respectively obtained from the Hermitian forms $H_1$ and
$H_2$ given in the canonical basis of $\C^3$ by the matrices
\begin{equation}\label{J1J2}
J_1=\begin{bmatrix}
1 & 0 & 0\\
0 & 1 & 0\\
0 & 0 & -1
\end{bmatrix}
\mbox{ and }
J_2=\begin{bmatrix}
0 & 0 & 1\\
0 & 1 & 0\\
1 & 0 & 0
\end{bmatrix}.
\end{equation}
In the case of $J_1$, the complex hyperbolic plane corresponds to the
unit ball of $\C^2$ seen as the affine chart $\{z_3=1\}$ of $\C P^2$. In
the same affine chart, $J_2$ gives an identification between $\HdC$ and
the Siegel domain of $\C^2$ defined by $\lbrace (z_1,z_2),
2\Re(z_2)+|z_1|^2<0\rbrace$. The Siegel model of $\HdC$ is also often referred to as the 
\textit{paraboloid model}. The vector $\begin{bmatrix}1 & 0 & 0\end{bmatrix}^T$ projects onto the only point in its 
closure which is not contained in this affine chart. We will denote this vector by $\bq_\infty$, and refer to 
the corresponding point as $q_\infty$. The two matrices $J_1$ and $J_2$ are conjugate in GL(3,$\C$) by the (order 2) matrix
\begin{equation}\label{Cayley} \bC=\dfrac{1}{\sqrt{2}}\begin{bmatrix}
1 & 0 & 1\\
0 & \sqrt{2} & 0\\
1 & 0 & -1
\end{bmatrix}.\end{equation}
The linear transformation given by $\bC$ descends to $\C P^2$ as the
\textit{Cayley transform}, which exchanges the ball and Siegel models
of $\HdC$. Of course, picking another Hermitian leads to another system of
coordinates on the complex hyperbolic plane, which is projectively equivalent to these. 
It is often useful to do so to adapt coordinates to a specific problem 
(see for instance \cite{Most,P3,ParkPau}). 

\n In the Siegel model, any point of $\HdC$ admits a unique standard lift given by
\begin{equation}
 m_{z,t,u}=\begin{bmatrix}
            -|z|^2-u+it\\\sqrt{2}z\\1
           \end{bmatrix}\mbox{, where }z\in\C,t\in\R\mbox{ and }u>0.
\end{equation}
The triple $(z,t,u)$ is called the horospherical coordinates of a point: horospheres based at $q_\infty$ are the level sets 
of $u$. It is an easy exercise to verify that they are preserved by parabolic isometries fixing $q_\infty$ using the matrices 
in section \ref{section-isom-types}. In these coordinates, the boundary of $\HdC$ corresponds to the null-locus of $u$. 
A boundary point is thus given by a pair $[z,t]\in\C\times\R$. The action of the unipotent parabolic maps (see Section 
\ref{section-isom-types} below) coincides with the Heisenberg multiplication :
\begin{equation}\label{heislaw}
[z,t]\cdot[w,s]=[z+w,t+s+2\Im(z\bar w)].
\end{equation}
The boundary of $\HdC$ can thus be thought of as the one point compactification of the 3 dimensional Heisenberg group.
Detailed information on these two models of the complex hyperbolic plane can be found in Chapters 3 and 4 of \cite{Go}.
\subsection{Totally geodesic subspaces.\label{section-total-geod}}
An important feature of complex hyperbolic plane is that it does not contain any (real) codimension-1 totally geodesic 
subspace. This is crucial when one wants to construct fundamental domains for a subgroups of PU(2,1), as it makes the 
construction of polyhedra very difficult. It is a consequence of the fact that the real sectional curvature of $\HdC$ is 
non constant (it is pinched between $-1$ and $-1/4$ in the normalisation we have chosen here). The maximal totally geodesic 
subspaces come in two types:
\begin{enumerate}
 \item Complex lines are intersections of projective lines of $\C P^2$ with $\HdC$ (when non-empty). The duality associated with the 
Hermitian form provides a correspondence between the set of complex lines of $\HdC$ and the outside $\C P^2\setminus\HdC$. Indeed, any 
complex line $L$ is the intersection with $\HdC$ of the projection to $\C P^2$ of the kernel of a linear form $z\longmapsto\la z,\bn\ra$, 
where $\bn$ is a positive type vector. Such a vector $\bn$ is called \textit{polar} to $L$ and is unique up to scaling.
\item Real planes are totally geodesic totally real subspaces of $\HdC$. Using a Hermitian form with real coefficients as 
$H_1$ or $H_2$, real planes are images under PU(2,1) of the set of real points of the considered projective model of complex 
hyperbolic space. In particular, they are the fixed point sets of \textit{real reflections}. The standard example of a real 
reflection is complex conjugation $(z_1,z_2)\longmapsto (\overline{z_1},\overline{z_2})$, which is an isometry.
\end{enumerate}
The two kinds of totally geodesic subspaces realise the extrema of the sectional curvature: it is  $-1/4$ for real planes and 
$-1$ for complex lines. 
\subsection{Isometry types and conjugacy classes in PU(2,1).\label{section-isom-types}}
As usual, there are three mutually exclusive isometry types: loxodromic, 
elliptic and parabolic, depending on the location of fixed points.  We refer the reader to Chapter 6 of \cite{Go} for 
basic definitions, but we would like to emphasize a few facts. 
\paragraph{Elliptics}
Elliptic isometries have (at least) one fixed point inside $\HdC$. There are
two kinds of elliptic isometries.
\begin{definition}
An elliptic isometry $f$ is called \textit{regular} if any of its
lifts to SU(2,1) has three pairwise distinct eigenvalues. Any other
elliptic isometry is called a \textit{complex reflection}.
\end{definition}
\noindent In the ball model coordinates, any elliptic isometry $E$ is conjugate to one given by the diagonal matrix
\begin{equation}
 {\bf E}=\begin{bmatrix}
  e^{i\alpha} & 0 & 0 \\
  0 & e^{i\beta} & 0 \\
  0 & 0 & e^{i\gamma}
 \end{bmatrix}.
\end{equation}
Projectively, the associated isometry is given by
\begin{equation}(z_1,z_2)\longmapsto\left(e^{i(\alpha-\gamma)}z_1,e^{i(\beta-\gamma)}z_2\right).\label{regular-elliptic}\end{equation}
We thus see that $E$ has two stable complex lines, which are in the normalised case of \eqref{regular-elliptic} the complex axes of 
coordinatesof the ball. Its conjugacy class is determined by the (unordered) pair $\lbrace\alpha-\gamma,\beta-\gamma \rbrace$. 
These two angles correspond to the rotation angles of the restriction of the elliptic to its stable complex lines.
Whenever two of the eigenvalues are equal, the action of $E$ on the unit ball is given (up to conjugacy) by one of the 
following maps
\begin{align}
(z_1,z_2) & \longmapsto  (z_1,e^{i\theta}z_2)\label{droite}\\
(z_1,z_2) & \longmapsto  (e^{i\theta}z_1,e^{i\theta}z_2)\label{point}.
\end{align}
\n  We refer to the first case as a \textit{complex reflection about a line} (in \eqref{droite} the first axis of coordinates is fixed),
 and to the second as  a \textit{complex reflection about a point} (in \eqref{point} the point $(0,0)$ is the unique fixed point). 
These reflections may not have order 2, and not even finite order. 
\paragraph{Parabolics}
Parabolic isometries fall into two types: unipotent parabolics and
screw-parabolics. Unipotent parabolics are those admitting a unipotent lift to
SU(2,1). Unipotent matrices in SU(2,1) can be either 2 or 3 step unipotent. There are in turn 
three conjugacy classes of unipotent parabolics in PU(2,1) represented in the Siegel model by the 
following matrices:
\begin{equation}
\begin{bmatrix} 1 & 0 & \pm i\\ 0 & 1 & 0\\0 & 0 & 1\end{bmatrix}\mbox{ and }\begin{bmatrix}1 & -\sqrt{2} & -1\\ 0 & 1 & \sqrt{2}\\0 & 0 & 1\end{bmatrix}.
\end{equation}
Unipotent parabolics fixing $q_\infty$ are \textit{Heisenberg translations} : they act on the Heisenberg group 
as the left multiplication by $[z,t]$ (compare with \eqref{heislaw}). 
\begin{equation}\label{liftHeistrans}
 T_{[z,t]}=\begin{bmatrix}
             1 & -\bar z\sqrt{2} & -|z|^2+it\\
             0 & 1 & z\sqrt{2}\\
             0 & 0 & 1 
            \end{bmatrix}.
\end{equation}
It is easily checked using \eqref{liftHeistrans} that $T_{[z,t]}T_{[w,s]}=T_{[z,t]\cdot[w,s]}$.
There is also a 1-parameter family of screw parabolic conjugacy classes, represented by the matrices
\begin{equation}\label{screw}
\begin{bmatrix} e^{-i\alpha/3} & 0 & ie^{-i\alpha/3}\\ 0 & e^{2i\alpha/3} & 0\\0 & 0 & e^{-i\alpha/3}\end{bmatrix}.
\end{equation}
The action on the boundary of the screw-parabolic element given by \eqref{screw} is in Heisenberg coordinates $[z,t]\longmapsto[e^{i\alpha}z,t+1]$. 
This explains the terminology.
\paragraph{Loxodromics.} Any loxodromic isometry is conjugate to the one given in the Siegel model by the matrix
\begin{equation}\label{class-conj-loxo}
 \bA_\lambda=\begin{bmatrix}
              \lambda & 0 & 0\\
               0 & \overline{\lambda}/\lambda & 0\\
               0 & 0 & 1/\overline{\lambda}
             \end{bmatrix},\mbox{ where }|\lambda|>1
\end{equation}
In view of \eqref{class-conj-loxo}, a loxodromic conjugacy class in PU(2,1) is determined by a complex number of modulus greater 
than $1$, defined up to multiplication by a cube root of $1$. This means that the set of loxodromic conjugacy classes can be seen
as the cylinder $\mathcal{C}_{\small {\rm lox}}=\{|z|>1\}/\Z_3$. The translation length $\ell$ of a loxodromic isometry is given 
in terms of eigenvalues by $e^{\ell/2}=|\lambda|$ (see Proposition 3.10 of \cite{Parktraces}).

\paragraph{Trace in SU(2,1) and isometry type.}
In SL(2,$\C$), the isometry type of an element can be decided by its trace. It is almost the case in SU(2,1), as shown 
by the next proposition (see Chapter 6 of \cite{Go} for more details).
\begin{proposition}\label{deltofunc}
Let $A\in $ PU(2,1) be a holomorphic isometry of the complex hyperbolic space, and ${\bf A}$ be a lift of it to SU(2,1). 
Let $f$ be the function on $\C$ defined by $f(z)=|z|^4-8\Re(z^3)+18|z|^2-27$. 
\begin{enumerate}
 \item If $f(\tr{\bf A})>0$ then $A$ is loxodromic.
 \item If $f(\tr{\bf A})<0$ then $A$ is regular elliptic.
 \item If $f(\tr{\bf A})=0$ then $A$ is either parabolic or a complex reflection.
\end{enumerate}
\end{proposition}
The zero-locus of the function $f$ is depicted on figure \ref{delto}.
Proposition \ref{deltofunc} is straightforward once one notes that the function $f$ is the resultant of the 
characteristic polynomial of a generic element ${\bf A}$ of SU(2,1), which is given by $\chi_A(X)=X^3-z X^2+\overline{z} X-1$, 
where $z=\tr{\bf A}$. 
\begin{figure}\label{delto}
\begin{center}
\scalebox{0.8}{\includegraphics{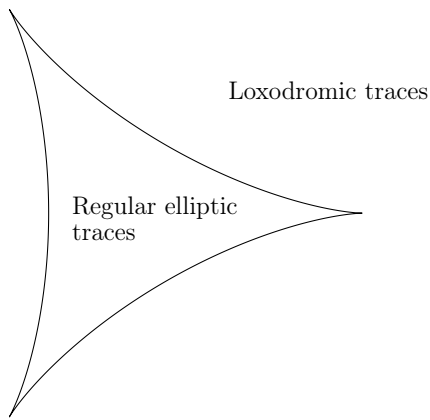}}
\end{center}
\caption{The zero locus of $f$.}
\end{figure}
The conjugacy class of an element in PU(2,1) is not determined in general by the trace of one of its lifts. 
The situation is as follows. All complex numbers here are considered up to multiplication by a cube root of 1.
\begin{enumerate}
 \item Each complex number outside the deltoid curve is the trace of a unique loxodromic conjugacy class in SU(2,1). 
 \item Each complex number on the deltoid curve, but not at a cusp, corresponds to three different SU(2,1)-conjugacy classes. 
One of these classes is parabolic and corresponds to non-semi simple elements in SU(2,1), and two are complex reflections 
about a point or about a line. In these cases the spectrum is of the type $\{e^{2i\theta},e^{-i\theta},e^{-i\theta}\}$ for 
some $\theta\in\R$.
 \item A non-trivial element in SU(2,1) with trace $3$ is unipotent. This gives the three unipotent conjugacy classes given above.
 \item Any complex number inside the deltoid curve corresponds to three regular elliptic conjugacy classes. Here the spectrum is of 
the form $\{e^{i\alpha},e^{i\beta},e^{-i(\alpha+\beta)}\}$. The three different conjugacy classes correspond to the possible relative
locations of the corresponding eigenvectors: one of them is inside the negative cone of the Hermitian form, and the other two are 
outside. This leaves three possibilities for angle pairs.
\end{enumerate}
\section{Projective invariants for configurations of points in $\HdC$.\label{invariants}}
\subsection{Triple ratio and cross ratio.\label{section-triple-cross}}
A lot of information on projective invariants for configurations of points or complex lines can be found in Chapter 7 of \cite{Go}.
We present here a few cross-ratio type invariants that we will need later on.
\subsubsection{Triples of points}
\begin{definition}
Let $\tau=(p_1,p_2,p_3)$ be an (ordered) triple of pairwise distinct
points in the closure of $\HdC$. We denote by $\bp_i$  a lift of $p_i$ to $\C^3$.
\begin{enumerate}
\item The triple ratio of $\tau$ is defined by
\begin{equation}\label{tripleratio}
{\bf T}(p_1,p_2,p_3)=\dfrac{\la\bp_1,\bp_2\ra\la\bp_2,\bp_3\ra\la\bp_3,\bp_1\ra}{\la\bp_1,\bp_3\ra\la\bp_3,\bp_2\ra\la\bp_2,\bp_1\ra}.
\end{equation}
\item The angular invariant of $\tau$ is the quantity
\begin{equation}\label{angular}
\alpha(\tau)=\arg\left(-\la\bp_1,\bp_2\ra\la\bp_2,\bp_3\ra\la\bp_3,\bp_1\ra\right).
\end{equation}
\end{enumerate}
\end{definition}
Both invariants are independant of the choices of lifts made. The angular invariant is linked to the  triple ratio by
\begin{equation}\label{angulartriple}
e^{2i\alpha(\tau)}={\bf T}(\tau).
\end{equation}
The angular invariant $\alpha$ measures the complex area of a simplex built on the triangle $(p_1,p_2,p_3)$. Indeed, it satisfies 
\begin{equation}
\int_{\Delta(p_1,p_2,p_3)} \omega =2\alpha(p_1,p_2,p_3),
\end{equation}
where $\omega$ is the K\"ahler form on $\HdC$. This is proved in Chapter 7 of \cite{Go}. The connection between the angular invariant 
of a triangle and the integral of the K\"ahler form is related to the definition of the Toledo invariant of a representation of the 
fundamental group of a surface in $\HnC$ (see \cite{Tol} and Section 7.1.4. of \cite{Go}).

In the case of ideal triangles where the three points are all on the boundary of $\HdC$,
the angular invariant is usually called \textit{Cartan's invariant} and denoted by $\A(p_1,p_2,p_3)$. We will call any 
such triple of points an \textit{ideal triangle}.  The main properties of the 
Cartan invariant are summarised in the next proposition (see Chapter 7 of \cite{Go} for proofs).
\begin{proposition}\label{prop-prop-Cartan}
 The Cartan invariant enjoys the following properties.
\begin{enumerate}
\item For any ideal triangle $\tau$, $\A(\tau)\in[-\pi/2,\pi/2]$.
\item Two ideal triangles are in the same PU(2,1)-orbit if and only if
  they have the same Cartan invariant.
\item An ideal triangle $\tau$ has zero  Cartan invariant if and only if it lies in a real plane.
\item An ideal triangle $\tau$ has extremal Cartan invariant ($|\A(\tau)|=\pi/2$) if and only if it lies in a complex line.
\end{enumerate}
\end{proposition}
\subsubsection{Quadruples of points\label{section-quadruples}}
In this section, we classify ideal tetrahedra up to PU(2,1). The main invariant is the complex cross-ratio, that was defined by 
Koranyi and Reimann in \cite{KR}.
\begin{definition}
Let $\mathcal{Q}=\left(p_1,p_2,p_3,p_4\right)$ be an ordered quadruple
of pairwise distinct points in the closure of $\HdC$. The
\textit{complex cross ratio} of $\mathcal{Q}$ is the quantity
\begin{equation}\label{crossratio}
\X(p_1,p_2,p_3,p_4)=\dfrac{\la\bp_3,\bp_1\ra\la\bp_4,\bp_2\ra}{\la\bp_3,\bp_2\ra\la\bp_4,\bp_1\ra},
\end{equation}
where $\bp_i$ is a lift of $p_i$ to $\C^3$.
\end{definition}
Clearly, the complex cross-ratio is PU(2,1)-invariant. A rough dimension count shows that the expected dimension of the set 
of PU(2,1)-orbits of ideal tetrahedra is four. Indeed, the set of ideal tetrahedra is $(S^3)^4$, and PU(2,1) has real dimension 8. 
In particular there is no hope to classify these orbits with a  a single cross-ratio. Various choices of invariants are possible to 
classify these orbits. In \cite{FalJDG,FalPla,PP,PP3,Wi3}, the choice made is to use three cross-ratios linked by 2 (real) relations. 
In \cite{GuH1}, Gusevski and Cunha have used two cross-ratios and one Cartan invariant that are connected by one (real) relation. 
Their choice is in a sense better, as it allows to avoid hypotheses of genericity. The choices of cross-ratios made in 
\cite{FalJDG,FalPla,PP,PP3,Wi3} do not detect (degenerate) ideal tetrahedra that are contained in a complex line. For our concern, 
we will use the same convention as in \cite{PP}, and keep in mind the slight ambiguity pointed out in \cite{GuH1}. In particular, we 
will only consider ideal tetrahedra that are not contained in a complex line, and we will calll these \textit{non-degenerate}. For 
a given ideal tetrahedron $(p_1,p_2,p_3,p_4)$, we denote by $\X_i$ the three cross-ratios given by
\begin{equation}\label{choice-cross}
 \X_1=\X(p_1,p_2,p_3,p_4), \X_2=\X(p_1,p_3,p_2,p_4) \mbox{ and }\X_3=\X(p_2,p_3,p_1,p_4).
\end{equation}
We will refer to $(\X_1,\X_2,\X_3)$ as the \textit{cross-ratio triple} of the tetrahedron $(p_1,p_2,p_3,p_4)$. Using the Siegel 
model one can normalise any ideal tetrahedron so that it is given by the following lifts.
\begin{equation}\label{normaltetra}
\bp_1=\begin{bmatrix}0 \\ 0 \\1\end{bmatrix},\, \bp_2= \begin{bmatrix}1 \\ 0 \\0\end{bmatrix},\, 
\bp_3= \begin{bmatrix}z_1 \\ z_2 \\1\end{bmatrix}\mbox{ and } \bp_4= \begin{bmatrix}1 \\ w_2 \\w_3\end{bmatrix}.
\end{equation}
As $p_3$ and $p_4$ belong to $\partial\HdC$ the following relations are satisfied
\begin{equation}\label{relbord}
z_1+\overline{z_1}+|z_2|^2=w_3+\overline{w_3}+|w_2|^2=0.
\end{equation}
In this case, the cross-ratio triple is as follows.
\begin{align}
 \X_1&= z_1w_3\label{X1}\\
 \X_2&= 1+\overline{z_2}w_2+\overline{z_1}w_3\label{X2}\\
 \X_3&=\dfrac{1+\overline{z_2}w_2+\overline{z_1}w_3}{\overline{z_1}w_3}.\label{X3}
\end{align}
The two real relations mentioned above are as follows.
\begin{proposition}\label{prop-relations-crossratio}
Let $(p_1,p_2,p_3,p_4)$ be an ideal tetrahedron. Then the three cross-ratios $\X_1$, $\X_2$ and $\X_3$ satisfy the relations
\begin{equation}\label{X-eq-1}
 |\X_2|=|\X_1\X_3|
\end{equation}
and
\begin{equation}\label{X-eq-2}
2|\X_1|^2\Re(\X_3) = |\X_1|^2+|\X_2|^2+1-2\Re(\X_1+\X_2)\\ 
\end{equation}
\end{proposition}
\begin{proof}
Relation \eqref{X-eq-1} is straightforward from the following dentity connecting the 
cross-ratio and the Cartan ratio.
\begin{equation}\label{Xtriple}
\X(p_1,p_2,p_3,p_4)\X(p_2,p_3,p_1,p_4)=\X(p_1,p_3,p_2,p_4)e^{2i\A(p_1,p_2,p_3)}.
\end{equation}
Relation \eqref{X-eq-2} is obtained by writing that for any
choice of lifts $(\bp_i)_{1=i}^4$ of the $p_i$'s, the determinant of the
Gram matrix $\left(\la\bp_i,\bp_j\ra\right)_{1\leq i,j\leq 4)}$ is equal to zero (see for instance chapter 7 of \cite{Go}).
\end{proof}
We see from relation \eqref{Xtriple} that the cross-ratio triple determines the Cartan invariant $\A(p_1,p_2,p_3)$ while it is not 
equal to $\pm \pi/2$. 
\begin{remark}\label{sym-cross-ratio}
If one fixes $\X_1$ and $\X_2$, there are two possible
complex conjugate values for $\X_3$, as $\Re(\X_3)$ and $|\X_3|$ are
given by \eqref{X-eq-1} and \eqref{X-eq-2}. Moreover two complex number $\X_1$ and $\X_2$ can
only be cross ratios for a quadruple of point if and only the
corresponding values of $\X_3$ satisfies $|\Re(\X_3)|\leq |\X_3|$. This condition is equivalent to the double inequality
\begin{equation}\label{compatcross}
\left(|\X_1|-|\X_2|\right)^2\leq 2
\Re\left(\X_1+\X_2\right)-1\leq\left(|\X_1|+|\X_2|\right)^2
\end{equation}
As mentioned by Parker and Platis in \cite{PP}, the change $(\X_1,\X_2,\X_3)\longmapsto(\X_1,\X_2,\overline{\X_3})$ corresponds 
to an involution on the set of ideal tetrahedra that is not induced by an isometry of $\HdC$. Indeed, an isometry would leave the 
three cross-ratios unchanged if it was holomorphic, or would conjugate them all if it was antiholomorphic. Morever, it can be checked
that this involution does not come from a permutation of the four points (all changes in the cross ratios induced by such permutations
are computed in \cite{Wi3}). 
\end{remark}
The cross-ratio triple is a complete system of invariants for the PU(2,1)-orbits of non degenerate ideal tetrahedron.
\begin{theorem}\label{theo-classtetra}
 \begin{enumerate}
\item Two non-degenerate ideal tetrahedra are in the same PU(2,1)-orbit of and only if the have the same cross ratio triple.
\item If $(\X_1,\X_2,\X_3)$ is a triple of complex numbers satisfying relations \eqref{X-eq-1} and \eqref{X-eq-2}, then there 
exists a non-degenerate ideal tetrahedron of which it is the cross-ratio triple if and only if $\X_1$ and $\X_2$ satisfy the 
compatibility relation \eqref{compatcross}.
\end{enumerate}
\end{theorem}
Proofs of the first part can be found in \cite{FalJDG,PP,Wi3}. In these, the fact that the tetrahedron is non degenerate is often 
implicitely used without being stated (this omission is made for instance in \cite{Wi3}). The second part can be found in \cite{PP}.

 We will also make use of the quadruple ratio ${\bf Q}$ which is defined by
$${\bf Q}(p_1,p_2,p_3,p_4)=\dfrac{\la\bp_1,\bp_2\ra\la\bp_2,\bp_3\ra\la\bp_3,\bp_4\ra\la\bp_4,\bp_1\ra}{\la\bp_1,\bp_4\ra\la\bp_4,\bp_3\ra\la\bp_3,\bp_2\ra\la\bp_2,\bp_1\ra}$$
A direct verification shows that the quadruple ratio $\bf Q$ satisfies the following relations
\begin{align}\label{relation-Q}
{\bf Q}(p_1,p_2,p_3,p_4) & =  {\bf T}(p_1,p_2,p_3)\cdot{\bf T}(p_1,p_3,p_4)\nonumber\\
 & = \dfrac{\X(p_1,p_2,p_3,p_4)}{\overline{\X(p_1,p_2,p_3,p_4}}. 
\end{align}
\subsection{Complex cross ratio and eigenvalues.\label{section-cross-ratio-eigenvalue}}
One often has to consider configurations of points that arise as fixed points of isometries. Taking lifts to $\C^3$ and SU(2,1), 
one obtains eigenvectors of matrices, and eigenvalues. It is interesting to relate the projective invariants of these 
configurations to the eigenvalues associated to the vectors lifting these fixed points. As an example, the following lemma 
can be found in \cite{ParkWi}. It provides a connection between the geometry of the fixed 
points of a pair of isometries and the associated eigenvalues. In this section, each time we will consider an isometry $A$, 
we will mean by $p_A$ a fixed point of $A$. When needed $\bA$ and $\bp_A$ will stand for lifts of $A$ to SU(2,1) and of $p_A$ to $\C^3$.
\begin{definition}
Let $A$ and $B$ be two elements of PU(2,1). We say that a 4-tuple of fixed points in $\overline{\HdC}$ $(p_A,p_B,p_{AB},p_{BA})$ 
is compatible if it satisfies $Bp_{AB}=p_{BA}$ and $Ap_{BA}=p_{AB}$.
\end{definition}
When $A$, $B$, $AB$ and $BA$ all have a unique fixed point the compatibility condition is empty. If for instance 
$AB$ and $BA$ are loxodromic we require here that $p_{AB}$ and $p_{BA}$ are either both repulsive or both attractive.
\begin{lemma}\label{lem-cross-ratio-eigenvalues}
Let $A$ and $B$ be two elements of PU(2,1), and let $(p_A,p_B, p_{AB},p_{BA})$ be a compatible 4-tuple of fixed points. 
Fix lifts of $A$ and $B$ given by $\bA$ and $\bB$ in SU(2,1). For $\bC\in\{\bA,\bB,\bA\bB,\bB\bA\}$ denote by $\lambda_\bC$ the 
eigenvalue of $\bC$ associated to $p_C$. Then
\begin{equation}\X(p_A,p_B,p_{AB},p_{BA})=\dfrac{1}{\overline{\lambda_A}\overline{\lambda_B}\lambda_{AB}}.
\label{Xcross}\end{equation}
\end{lemma}
Note that the right hand side do not depend of the choices made for the lifts of $A$ and $B$.
\begin{proof}
 For each of the four fixed points involved, we fix a lift to $\C^{3}$ and obtain four vectors $\bp_A$, $\bp_B$, 
$\bp_{AB}$ and $\bp_{BA}$. Because $A$ maps $p_{BA}$ to $p_{AB}$ and $B$ maps $p_{AB}$ to $p_{BA}$, there exist two complex numbers
$z$ and $w$ such that 
\begin{equation}\label{eqident}{\bf A}\bp_{BA}=z\bp_{AB}\mbox{ and }{\bf B}\bp_{AB}=w\bp_{BA}.\end{equation} 
The eigenvalue of $\bA\bB$ associated to $\bp_{AB}$ is $zw$. Then 
\begin{eqnarray*}
 \X(p_A,p_B,p_{AB},p_{BA})&=&\dfrac{\la\bp_{AB},\bp_{A}\ra\la\bp_{BA},\bp_{B}\ra}{\la\bp_{AB},\bp_{B}\ra\la\bp_{BA},\bp_{A}\ra}\nonumber\\
& = & \dfrac{\la\bp_{AB},\bp_{A}\ra\la\bp_{BA},\bp_{B}\ra}{\la{\bf B}\bp_{AB},{\bf B}\bp_{B}\ra\la{\bf A}\bp_{BA},{\bf A}\bp_{A}\ra}\nonumber\\
& = & \dfrac{1}{\overline{\lambda_A}\overline{\lambda_B}\lambda_{AB}}\mbox{, by using \eqref{eqident}.}
\end{eqnarray*}
\end{proof}
Identity \eqref{Xcross} is specially nice when $A$, $B$ and $C=(AB)^{-1}$ are parabolic. Indeed, in that case eigenvalues $\lambda_A$, 
$\lambda_B$ and $\lambda_C$ have unit modulus, and one obtains
\begin{equation}\label{Xcross-parab}
 \X(p_A,p_B,p_{AB},p_{BA})=\lambda_A\lambda_B\lambda_C.
\end{equation}
Viewing the group $\la A,B\ra$ as a representation to PU(2,1) of the 3 punctured sphere, we can relate \eqref{Xcross-parab} to the
Toledo invariant of this representation. Indeed, taking arguments on both sides, we obtain
\begin{equation}
 \A(p_A,p_B,p_{AB})+\A(p_A,p_{BA},p_B)=\arg(\lambda_A)+\arg(\lambda_B)+\arg(\lambda_C) \mod 2\pi\label{tol}
\end{equation}
The left hand side of \eqref{tol} is equal to the integral over the finite area 3 punctured sphere of the pull back of the 
K\"ahler form of $\HdC$ by an equivariant map from the Poincar\'e disc to $\HdC$, which is equal to the Toledo invariant 
(see \cite{BuIoWi,KoMa,Tol} for general definitions and \cite{GuP2} for calculations in this specific frame). 
In this very special case, this relation contains the same information as Lemma 8.2 of \cite{BuIoWi}, which connects in a much 
broader context the Toldedo invariant to the rotation numbers of images of the peripheral curves by a representation. 
Identity \eqref{Xcross} can be generalised to larger genus surfaces. To do so we will use a more symmetric identity given by 
Proposition \ref{prop-Q-eigenvals} below. It is a straightforward consequence of Lemma \ref{lem-cross-ratio-eigenvalues}, 
and the properties of the quadruple ratio \eqref{relation-Q}.
\begin{proposition}\label{prop-Q-eigenvals}
 Let $\pi_{0,3}\sim\la a,b,c\vert abc=1\ra$ be the fundamental group of the 3 punctured sphere. Let $\rho$ be a 
representation of $\pi_{0,3}$ in PU(2,1). Denoting $A=\rho(a)$, $B=\rho(b)$ and $C=\rho(c)$, let $(p_A,p_B,p_{C},p_{BCB^{-1}}$ 
be a compatible 4-tuple of fixed points for the pair $(A,B)$, and let $\lambda_A$, $\lambda_B$, $\lambda_C$ be the associated eigenvalues. Then
\begin{equation}
{\bf Q}(p_A,p_B,p_C,p_{BCB^{-1}})= \dfrac{\lambda_A\lambda_B\lambda_C}{\overline{\lambda_A\lambda_B\lambda_C}}
\end{equation}
\end{proposition}
Let $\Sigma_{g,p}$ be a oriented surface of genus $g$ with $p$ punctures, where $p\geqslant 0$, and denote by $\pi_{g,p}$ its 
fundamental group, given by 
$$\pi_{g,p}=\la a_1,b_1,\cdots,a_g,b_g,c_1\cdots c_p\vert \prod_{i=1}^g[a_i,b_i]\prod_{j=1}^pc_j=1\ra,$$ 
where the $c_j$'s are homotopy classes of loops enclosing the punctures. 
We make once and for all the choice that for a three punctured sphere the orientation of the peripheral loops 
is chosen so that the surfaces is on the right of each of these loops. Fix a pair of pants decomposition of 
$\Sigma_{g,p}=\bigcup_{i=1}^{2g-2+n} \mathcal{P}_i$, or equivalently a maximal collection of oriented simple 
closed curves on $\Sigma_{g,p}$. A representation $\rho$ of $\pi_{g,p}$ in PU(2,1) induces representations  $\rho_i$ of each 
of the fundamental groups of the $\mathcal{P}_i$'s which satisfies the following conditions.
\begin{enumerate}
 \item If two pairs of pants $\mathcal{P}_i$ and $\mathcal{P}_j$ for $i\neq j$ are glued along a common peripheral curve $\gamma$ then 
$\rho_i(\gamma)=\rho_j(\gamma)^{-1}$.
\item If two peripheral curves $\gamma$ and $\gamma'$ of a pant $\mathcal{P}_i$ are glued together to produce a handle in 
$\Sigma_{g,p}$, then $\rho_i(\gamma)$ is conjugate to $\rho_i(\gamma')^{-1}$
\end{enumerate}
These conditions follow from the convention we have taken for orintation, and correspond to the reconstruction of the group 
$\rho(\pi_{g,p})$ from the groups $\rho_i(\pi_1(\mathcal{P}_i))$ by amalgamated products and HNN extensions 
(see Remark \ref{CHFN} below). To each of the representations $\rho_i$ is associated a quadruple ratio ${\bf Q}_i$ as 
in Proposition \ref{prop-Q-eigenvals}.
\begin{theorem}\label{theo-identity}
Let $\Sigma_{g,p}$ be an oriented surface of genus $g$ with $p$ punctures. 
For any pair of pants decomposition of $\Sigma_{g,p}=\bigcup_{i=1}^{2g-2+n} \mathcal{P}_i$ and any representation $\rho$ of 
$\pi_1(\Sigma_g,p)$ the following identity holds.
\begin{equation}\label{identity}
 \prod_{i=1}^{2g-2+n}{\bf Q}_i=\prod_{j=1}^{p}\dfrac{\lambda_{c_j}}{\overline{\lambda_{c_j}}},
\end{equation}
where $\lambda_{c_j}$ is the eigenvalue associated to any fixed point of $\rho(c_j)$ in $\overline{\HdC}$.
\end{theorem}
If $A$ is loxodromic, then its eigenvalues associated to its fixed points in $\partial\HdC$ are $\lambda$ and $1/\bar\lambda$. 
If $A$ is a complex reflections, all its fixed points in $\overline{\HdC}$ are associated to the same eigenvalue. These are the only 
two cases where an isometry can have more than one fixed point in $\overline{\HdC}$. We see thus that the contribution of $\rho(c_j)$ 
to the right hand side product of \eqref{identity} does not depent on the chosen fixed point.
\begin{proof}
In view of Proposition \ref{prop-Q-eigenvals}, the product $\prod_{i=1}^{2g-2+n}{\bf Q}_i$ is equal to the product of all 
$\lambda/\bar\lambda$, where $\lambda$ runs along all eigenvalues of images of the simple curves in the pant decomposition 
under the representations $\rho_i$. Because of conditions 1 and 2 above, we see that each non peripheral curve contributes 
to this product by $1$. The result follows.
\end{proof}
\begin{remark}\label{CHFN}
The idea behind the sketch of proof above is the use of a complex hyperbolic analogue of the Fenchel-Nielsen coordinates 
for hyperbolic surfaces, which is a very natural way to pass from 2-generator groups to surface groups.
Such an analogue has been described by Parker and Platis in \cite{PP} (see also section 4.6 of the survey article 
\cite{Parktraces}). The first ingredient is to describe moduli for representations of 3-punctured spheres. Using 
a pant decomposition of of a surface $\Sigma$, one needs then to provide \textit{gluing parameters} in order to
combine together such representations to obtain a representation of the fundamental group of the whole surface.
The gluing parameters used by Parker and Platis are cross-ratios and eigenvalues, and are interpreted in as twist-bend 
parameters,  in a similar way as in \cite{Kourou,Tan} for the case of PSL(2,$\C$).
\end{remark}
\section{Classification of pairs in SU(2,1) by traces. \label{section-conjug}}
It is a classical fact from invariant theory that the ring of polynomials on the product of $p$ copies of SL($n$,$\C$) that
are invariant under the action of SL(n,$\C$) by diagonal conjugation is generated by the polynomials of the form 
$\tr X_{i_1}\cdots X_{i_k}$. Morever, this ring is finitely generated and in fact it suffices to consider words of 
length  $k\leqslant n^2$. We refer the reader to \cite{Probook} for general information on this topic. Our goal here is to 
expose an explicit result in the case where $p=2$ and $n=3$ and to specialise it to the real form SU(2,1) of SL(3,$\C$).
We first recall the main results concerning the case of SL(2,$\C$).
 All the material necessary to prove
the results we expose here on the SL(3,$\C$) case can be found in Chapter 10 of \cite{Fo}, which actually follows \cite{Wen}. 
The SL(3,$\C$)-trace equation for pairs of matrices given in Proposition \ref{S and P exist} below has been rediscovered by various 
authors, among which \cite{KH,Law,Wi3,Wi4}. A good survey on the question of traces in the specific case of SU(2,1) is 
\cite{Parktraces}, where all computations are made explicit. 
\subsection{Traces in SL(2,$\C$)}
It is classical to classify pairs of matrices in SL(2,$\C$) by traces. The basic identity is the following.
If $\bA$ and $\bB$ belong SL(2,$\C$) then the following trace identity holds
\begin{equation}\label{traceSL2}
\tr \bA\tr \bB = \tr \bA\bB + \tr \bA^{-1}\bB.
\end{equation}
Relation \eqref{traceSL2} is a direct consequence of the Cayley-Hamilton identity. The following result is central
in the study of the characters of representations of the free group of rank 2 in SL(2,$\C$). It goes back to Vogt \cite{Vo} 
and Fricke-Klein \cite{FrKl1,FrKl2}, and we refer to the survey article \cite{Gotraces} for a modern exposition oriented toward 
the description of the character varieties of small punctured surfaces.
Denote by $\mathcal{R}_2$=$\C[\mbox{SL(2,$\C$)$\times$SL(2,$\C$)}]^{{\mbox{\tiny{SL(2,$\C$)}}}}$ the ring of conjugacy 
invariant polynomials on SL(2,$\C$)$\times$SL(2,$\C$).
\begin{theorem}\label{theoFV}
\begin{enumerate}
\item Any element of $\mathcal{R}_2$ is a polynomial in $\tr \bA$, $\tr \bB$
  and $\tr \bA\bB$.
\item The map
\begin{align}
\Psi_2 : \mbox{SL(2,$\C$)$\times$ SL(2,$\C$)} &\longrightarrow \C^3\nonumber\\
(\bA,\bB) & \longmapsto  \left(\tr \bA,\tr \bB,\tr \bA\bB\right)\nonumber
\end{align}
is surjective.
\item Two irreducible pairs $(\bA,\bB)$ and $(\bA',\bB')$ of elements of
  SL(2,$\C$) are conjugate if and only if $\Psi_2(\bA,\bB)=\Psi_2(\bA',\bB')$.
\end{enumerate}
\end{theorem}
Among conjugacy invariant fonctions that appear naturally is the trace of the commutator. It is a simple exercise 
using \eqref{traceSL2} to check that 
\begin{equation}
\tr [\bA,\bB] =  Q(\tr \bA,\tr \bB,\tr \bA\bB)\label{tracecomSL2}
\end{equation}
where $Q\left(x,y,z\right)=x^2+y^2+z^2-xyz-2$.  This particular polynomial plays an important role in the study of the 
SL(2,$\C$)-character varietes for small surfaces, like the 1-punctured torus or the 4-holed sphere (see for instance 
\cite{Gotraces,Bow2}).
\subsection{The trace equation in SL(3,$\C$).}
Relations \eqref{traceSL2} and \eqref{tracecomSL2} can be generalised to SL($3$,$\C$) as follows.
\begin{proposition}\label{S and P exist}
There exist two polynomials $S$ and $P$ in $\mathbb{Z}[x_1,\dots, x_8]$ such
that for any pair of matrices $(\bA,\bB)\in$ SL(3,$\C$)$\times$SL(3,$\C$),
the two traces $\tr [\bA,\bB]$ and $\tr [\bA^{-1}\bB]$ are the roots of the
quadratic equation
\begin{equation}\label{trace equation}
X^2-s X + p=0,
\end{equation}
where $s=S({\bf \tau})$, $p=P({\bf \tau})$ and 
$${\bf \tau}=\left(\tr
\bA,\tr \bB,\tr \bA\bB,\tr \bA^{-1}\bB,\tr \bA^{-1},\tr \bB^{-1},\tr \bB^{-1}\bA^{-1},\tr
\bB\bA^{-1}\right).$$
\end{proposition}
\noindent We will often refer in the sequel to \eqref{trace equation} as the \textit{trace equation for SL(3,$(\C)$}, 
or more simply as the \textit{trace equation}. The proof of Proposition \ref{S and P exist} can be done in a very
similar spirit as the derivation of \eqref{tracecomSL2}, only more involved. All the material necessary to do this can be found
in chapter 10 of \cite{Fo}, which actually follows \cite{Wen}. All computations are made explicit in \cite{Law,Parktraces,Wi3} 
The basic idea is to make a repeated use of the Cayley-Hamilton identity. The explicit expressions for the 
polynomials $S$ and $P$ are as follows (a slightly simpler and more symmetric expression for \eqref{prodroots} is derived 
in \cite{Parktraces} after a change of variables).
\begin{equation}S=x_1x_5+x_2x_6+x_3x_7+x_4x_8-x_1x_2x_7-x_5x_6x_3-x_5x_2x_8-x_1x_6x_4+x_1x_2x_5x_6-3\label{sumroots} \end{equation}
and
\begin{align}
P & =  
 x_5^{2}\,x_6\,x_1^{2}\,x_2 + x_5\,x_6^{2}\,x_1\,x_2^{2}
 + x_4\,x_5^{2}\,x_2^{2} + x_5^{2}\,x_6^{2}\,x_7+ x_6^{2}\,x_8\,x_1^{2}
     + x_1^{2}\,x_2^{2}\,x_3 \nonumber\\
 &  - x_4\,x_5\,x_6\,x_1^{2} - x_4\,x_6^{2}\,x_1\,x_2 - x_5^{2}\,x_6\,x_1\,x_3
      - x_5^{2}\,x_8\,x_1\,x_2 \nonumber\\
 &  - x_5\,x_6^{2}\,x_2\,x_3 -x_5\,x_6\,x_8\,x_2^{2}
      - x_5\,x_7\,x_1^{2}\,x_2- x_6\,x_7\,x_1\,x_2^{2}\nonumber\\
 &  - x_5^{3}\,x_6\,x_2- x_5\,x_6^{3}\,x_1-x_5\,x_1\,x_2^{3}
     -x_6\,x_1^{3}\,x_2\nonumber\\
 &  - x_4\,x_5\,x_6\,x_7\,x_2- x_4\,x_5\,x_1\,x_2\,x_3-x_5\,x_6\,x_7\,x_8\,x_1
     -x_6\,x_8\,x_1\,x_2\,x_3\nonumber\\
 & 
+ x_4^{2}\,x_6\,x_7 
+ x_4^{2}\,x_1\,x_3 
+ x_4\,x_5^{2}\,x_6 
+ x_4\,x_5\,x_3^{2} 
+ x_4\,x_6^{2}\,x_3 \nonumber\\
 & 
+ x_4\,x_7^{2}\,x_2
+ x_4\,x_7\,x_1^{2} 
+ x_4\,x_1\,x_2^{2} 
+ x_5^{2}\,x_7\,x_2
+ x_5^{2}\,x_8\,x_3\nonumber\\
 & 
+ x_5\,x_6^{2}\,x_8
+ x_5\,x_7\,x_8^{2}
+ x_5\,x_2^{2}\,x_3
+ x_6^{2}\,x_7\,x_1
+ x_6\,x_8\,x_3^{2} \nonumber\\
 & 
+ x_6\,x_1^{2}\,x_3
+ x_7^{2}\,x_8\,x_1 
+ x_7\,x_8\,x_2^{2} 
+ x_8^{2}\,x_2\,x_3 
+ x_8\,x_1^{2}\,x_2\nonumber\\
&
- 2\,x_4^{2}\,x_5\,x_2
- 2\,x_5\,x_6\,x_7^{2}
- 2\,x_6\,x_8^{2}\,x_1
- 2\,x_1\,x_2\,x_3^{2}\nonumber\\
&
+ x_4\,x_5\,x_8\,x_1
+ x_4\,x_6\,x_8\,x_2
+ x_4\,x_7\,x_8\,x_3
+ x_5\,x_6\,x_1\,x_2
+ x_5\,x_7\,x_1\,x_3
+ x_6\,x_7\,x_2\,x_3\nonumber\\
& 
+ x_1^{3}+ x_2^{3} + x_3^{3}+ x_4^{3}+ x_5^{3}+ x_6^{3}+ x_7^{3}+x_8^{3}\nonumber\\
& 
- 3\,x_4\,x_5\,x_7- 3\,x_4\,x_2\,x_3- 3\,x_6\,x_7\,x_8- 3\,x_8\,x_1\,x_3\nonumber\\
& 
+ 3\,x_4\,x_6\,x_1+ 3\,x_5\,x_6\,x_3 + 3\,x_5\,x_8\,x_2+ 3\,x_7\,x_1\,x_2\nonumber\\
 & - 6\,x_4\,x_8 - 6\,x_5\,x_1 - 6\,x_6\,x_2 
 - 6\,x_7\,x_3 + 9.\label{prodroots}
\end{align}
\subsection{Classification of irreducible pairs in SU(2,1).\label{sectionclasspairs}}
The following theorem is due to Lawton \cite{Law}, and it generalizes Theorem \ref{theoFV} to the case of
SL(3,$\C$). We denote by $\mathcal{R}_3$ the ring of invariants $\C\mbox{[SL(3,$\C$)$\times$SL(3,$\C$)]}^{\mbox{\tiny{SL(3,$\C$)}}}$.
\begin{theorem}\label{theo-FVLawton}
\begin{enumerate}
\item Any element of $\mathcal{R}_3$ is a polynomial in the
  traces of the nine words $\bA$, $\bB$, $\bA\bB$, $\bA^{-1}\bB$, their inverses and $[\bA,\bB]$.
This polynomial is unique up to the ideal generated by the left hand side of
  \eqref{trace equation}.
\item The map $\Psi_3$ defined on  SL(3,$\C$)$\times$SL(3,$\C$) by
\begin{equation}\label{mapsl3c}
(\bA,\bB) \longmapsto  \left(\tr \bA,\tr \bB,\tr \bA\bB,\tr
\bA^{-1}\bB,\tr \bA^{-1},\tr \bB^{-1},\tr \bB^{-1}\bA^{-1},\tr \bB\bA^{-1} \right)\nonumber
\end{equation}
is a branched double cover of $\C^8$.
\item  Two irreducible pairs $(\bA,\bB)$ and $(\bA',\bB')$ of elements of
  SL(3,$\C$) are conjugate if and only if $\Psi_3(\bA,\bB)=\Psi_3(\bA',\bB')$ and $\tr [\bA,\bB]=\tr [\bA',\bB']$. 
\end{enumerate}
\end{theorem}
\noindent The relation $\overline{\bA}^TJ\bA=J$ defining SU(2,1) implies that any element $\bA\in$ SU(2,1) satisfies
\begin{equation}\tr \bA^{-1}=\overline{\tr \bA}\label{symSU21}\end{equation}
It is therefore possible to reduce the number of traces necessary to determine a pair $(\bA,\bB)$ up to conjugacy in SU(2,1).
Let  $\Psi_{2,1}$ be the mapping defined on SU(2,1)$\times$SU(2,1) by
\begin{eqnarray}\label{mappu21}
\Psi_{2,1}(\bA,\bB) = \left(\tr \bA,\tr \bB,\tr \bA\bB, \tr \bA^{-1}\bB,\tr [\bA,\bB]\right)
\end{eqnarray}
As a consequence of the first part of Theorem \ref{theo-FVLawton}, we see that for any word ${\tt w}$ in ${\tt a }$ and ${\tt b }$ there 
exists a polynomial $P_{\tt w}$ in the variables $z$ and $\overline{z}$ with $z\in\C^5$,  such that for any representation 
$\rho : F_2=\la{\tt a},{\tt b}\ra\longrightarrow $SU(2,1), 
$$\tr(\rho({\tt w}))=P_{\tt w}(\Psi_{2,1}(\rho({\tt a}),\rho({\tt b}))).$$
This polynomial is unique up to the relation given by the specialisation of the trace equation to SU(2,1). In the special case of 
triangle groups, Sandler \cite{Sa} and Prattousevitch \cite{Pra} have given explicit formulae allowing to compute traces of elements, 
that can be seen as a special case of the polynomials $P_{\tt w}$. In general though, no explicit or reccursive compuation of the 
polynomials $P_{\tt w}$ has been given to my knowledge. The map $\Psi_{2,1}$ classifies conjugacy classes of irreducible pairs in SU(2,1) :
\begin{proposition}\label{classSU21}
Two pairs irreducible pairs $(\bA,\bB)$ and $(\bA',\bB')$ of elements of
SU(2,1) are conjugate if and only if $\Psi_{2,1}(\bA,\bB)=\Psi_{2,1}(\bA',\bB')$.
\end{proposition}
\begin{proof}
 If $\Psi_{2,1}(\bA,\bB)=\Psi_{2,1}(\bA',\bB')$, then in view of \eqref{symSU21}  and Theorem \ref{theo-FVLawton}, there exists $g$ in SL(3,$\C$) such 
that $(\bA',\bB')=(g\bA g^{-1},g\bB g^{-1})$. The Hermitian form on $\C^3$ defined by $\la X,Y\ra_g=\la g^{-1} X,g^{-1} Y\ra$ is preserved by 
the group $\Gamma$ generated by $\bA'$ and $\bB'$. As $\Gamma$ is Zariski dense in SU(2,1), this implies that $\la\cdot,\cdot\ra_g$ is 
in fact SU(2,1)-invariant and thus $g$ belong to SU(2,1).
\end{proof}
\noindent We will also use the following map $\Phi_{2,1}$, which is the composition of $\Psi_{2,1}$ with the projection onto $\C^4$ 
given by the first four factors. It carries most of the information concerning traces.
\begin{eqnarray}\label{Phi21}
\Phi_{2,1} (\bA,\bB)= \left(\tr \bA,\tr \bB,\tr \bA\bB, \tr \bA^{-1}\bB\right)
\end{eqnarray}
Observe that the two polynomials $S$ and $P$ above are invariant under the change of variable
$(x_i)_{i=1}^{8}\longleftrightarrow(x_{i+4})_{i=1}^{8}$ (indices taken mod. 4). This is because
the two quantities $\tr[\bA,\bB]\tr[\bA^{-1},\bB^{-1}]$ and $\tr[\bA,\bB]+\tr[\bA^{-1},\bB^{-1}]$ are real, as can be 
checked using relation \eqref{symSU21}. Therefore the trace equation \eqref{trace equation} has a priori two real roots or two 
complex conjugates roots. For any $\bA\in$ SU(2,1), let us denote by $\bA^\tau$ the matrix $\overline{\bA^{-1}}=J\bA^TJ$, where $A^T$ is the 
transpose of $A$. The matrix $\bA^\tau$ belongs also to SU(2,1). It is a direct verification to see using \eqref{symSU21} that the 
pair $(\bA^\tau,\bB^\tau)$ satisfies
\begin{align}
& \tr \bA^\tau=\tr \bA,\, \tr \bB^\tau=\tr \bB,\, \tr \bA^\tau \bB^\tau=\tr \bA\bB,\, \nonumber\\
& \tr {(\bA^\tau)}^{-1}\bB^\tau=\tr \bA^{-1}\bB\mbox{ and } \tr [\bA^\tau,\bB^\tau]=\overline{\tr [\bA^\tau,\bB^\tau]}\label{tracestar}
\end{align}
This means that once the four traces of $\bA$, $\bB$, $\bA\bB$ and $\bA^{-1}\bB$ are fixed, the two possible values for the trace of 
$[\bA,\bB]$ are indeed represented by a pair of elements in SU(2,1) if and only if one of them is. As a consequence, the trace
equation \eqref{trace equation} has either one double real solution or two complex conjugate solutions.
This provides an explicit obstruction for a 4-tuple of complex numbers to be in the image of $\Phi_{2,1}$ : if 
\eqref{trace equation} has two distinct real solutions, they can not correspond to a pair of elements in SU(2,1). In other 
words, the polynomial $S^2-4P$ is negative on SU(2,1)$\times$SU(2,1). We will see in section \ref{section-loxoprod} that when 
$\bA$ and $\bB$ are loxodromic, this condition is in fact necessary and sufficient (see Theorem \ref{pairexists}). 
\subsection{When the trace equation has a real double root\label{subsection-double-root}}
In view of the previous section, it is natural to ask if one  characterise geometrically those pairs $(A,B)$ of elements of 
PU(2,1) such that $\tr [A,B]$ is real, that is the pairs of elements of PU(2,1) for which the trace equation has a double root? 
Note that even if the trace of an element in PU(2,1) is not well defined, the trace of a commutator is. Indeed, two lifts to SU(2,1) 
differ by multiplication  by a cube root of $1$, which is central in SU(2,1)  and thus does not affect the commutator. A sufficient 
condition for an element in SU(2,1) to have real trace is to have a real and positive eigenvalue associated to a fixed point in 
$\overline{\HdC}$. This follows from the fact that the spectrum of a matrix in SU(2,1) is stable under the transformation 
$z\longmapsto 1/\bar z$ (see chapter 6 of\cite{Go}). Even though the trace equation is only interesting for irreducible pairs, it is 
worth noting that pairs $(A,B)$ with a common fixed point in $\overline{\HdC}$ provide a first class of exemples where $\tr[A,B]$ is 
real. Indeed the common fixed point of $A$ and $B$ gives a fixed point of $[A,B]$ with eigenvalue equal to $1$. In \cite{PaW} the 
following result is proved. It provides a more interesting class of examples.
\begin{theorem}\label{theo-dec-PaW}
 Let $A$ and $B$ in PU(2,1) be two isometries with no common fixed point. The following assertions are equivalent.
\begin{enumerate}
 \item There exists three real symmetries $\sigma_i, i=1,2,3$ such that $A=\sigma_1\sigma_2$ and $B=\sigma_2\sigma_3$.
 \item The commutator $[A,B]$ has a fixed point $p$ in $\overline{\HdC}$ of which eigenvalue is real and positive 
(and thus it has real trace).
\end{enumerate}
\end{theorem}
Pairs satisfying the first property in Theorem \ref{theo-dec-PaW} are called $\R$-decomposable.
The main ingredients in \cite{PaW} are the following.
\begin{enumerate}
 \item Considering the four points given by the cycle associated to the fixed point $p$ of $[A,B]$ as follows
\begin{equation}\label{4-cycle}
p=p_1\overset{B^{-1}}\longrightarrow p_2\overset{A^{-1}}\longrightarrow p_3\overset{B}\longrightarrow p_4\overset{A}\longrightarrow p_1,
\end{equation}
one proves that $\lambda\cdot\X(p_2,p_4,p_1,p_3)$ is always positive, where $\lambda$ is the eigenvalue of $[A,B]$ associated to $p$. 
This is done by connecting cross ratio and eigenvalues by a relation in the spirit of Lemma \ref{lem-cross-ratio-eigenvalues}. In 
particular if $\lambda$ is positive, so is $\X$.
 \item A 4-tuple with real positive cross ratio has specific symmetries. More precisely $\X(p_2,p_4,p_1,p_3)$ 
is real and positive if and only if there exists a real symmetry $\sigma$ such that $\sigma(p_1)=p_3$ and $\sigma(p_2)=p_4$. 
A special case of this fact is mentioned in chapter 7 of \cite{Go}.
\end{enumerate}
The following fact follows also from \cite{PaW}.
\begin{proposition}
 If $[A,B]$ has a fixed point $p$ in $\overline{\HdC}$ with an associated real negative eigenvalue, then $p$ is on the boundary and
 the pair $(A,B)$  preserves a complex line.
\end{proposition}
Note that there are elements of SU(2,1) with a negative eigenvalue of negative type and non real trace : consider for instance an 
elliptic element with spectrum $\{e^{i\theta},-e^{-i\theta},-1\}$.
\begin{remark}\label{RdecCdec}
\begin{enumerate}
\item  The question of determining when a pair $(A,B)$ is $\R$-decomposable had been adressed in \cite{Wi2} under the 
assumption that $A$ and $B$ are loxodromic. The treatment of the question there was less natural 
than in \cite{PaW}. The result was obtained as a byproduct of the classification 
of pairs by traces given in section \ref{sectionclasspairs}. 
\item It is easy to see that $\tr [A,B]\in\R$ is a necessary condition for the pair $(A,B)$ to be $\R$ decomposable  
using lifts of real reflections. A lift of an antiholomorphic isometry $A$ is any matrix $M\in$U(2,1) such that for any $m\in\HdC$, 
$A(m)=P(M\overline{\bm})$, where $P$ denotes projectivisation, and $\bm$ is any lift of $m$ to 
$\C^3$. When $\sigma$ is a real reflection, any lift of $\sigma$ must satisfy 
$M\overline{M}=Id$ because $\sigma$ has order two. Now, if the first condition of theorem \ref{theo-dec-PaW} is satisfied, a 
direct computation shows that a lift of $[A,B]$ to SU(2,1) is given by $(M_1\overline{M_2}M_3\overline{M_1}M_2\overline{M_3})$, 
where $M_i$ is a lift of $\sigma_i$. The latter matrix is of the form $M\overline{M}$ and has therefore real trace.
\item Knowing that a pair $(A,B)$ is $\R$-decomposable can be very useful, as it shows that the group 
$\la A,B\ra$ has index two in a group generated by three real reflections. In particular, it provides an additional geometric data, 
given by the mirrors of the real reflections. As an example, in \cite{DFP} Deraux, Falbel and Paupert noticed that Mostow's Lattices
were generated by real reflections in this way, and used this remark to produce new fundamental domains for these groups.
\item There is a similar notion of $\C$-decomposablity : a pair $(A,B)$ is $\C$-decomposable whenever it can be written as in Theorem 
\ref{theo-dec-PaW}, but using complex symmetries instead of real ones. In this situation, we have $A=I_1I_2$,
$B=I_2I_3$, $AB=I_1I_3$ and $A^{-1}B=I_2I_1I_2I_3$. In particular, these four isometries are products of two complex 
reflections of order two (for $A^{-1}B$, note that $I_2I_1I_2$ is conjugate to $I_1$ and is thus a complex reflection). It is 
a simple exercise to prove that the product of two such complex reflections always have real trace, and therefore when $(A,B)$ is 
$\C$-decomposable $A$, $B$, $AB$ and $A^{-1}B$ are all real. In particular, $\C$-decomposable pairs provide fixed points for the 
involution on the trace variety given by $(z_1,z_2,z_3,z_4,z_5)\longmapsto (\overline{z_1},\overline{z_2},\overline{z_3},\overline{z_4},z_5)$ 
which is induced by $(A,B)\longmapsto(A^{-1},B^{-1})$.
\end{enumerate}
\end{remark}
\subsection{An example.}
We consider now the example of pairs of unipotents having uniporent product.
\begin{proposition}\label{uniplox}
 Let $A$ and $B$ in PU(2,1) be two unipotent parabolic elements with different fixed point and such that $AB$ 
is also unipotent. Then $A^{-1}B$ is loxodromic.
\end{proposition}
Note that pairs $(A,B)$ with different fixed points and such that $A$, $B$ and $AB$ all are unipotent exist. A simple example 
can be obtained by embedding a Fuchsian groups uniformising a 3-punctured sphere into the stabilizer of a real plane. 
They are described and classified in \cite{ParkWi}, and will be the object of the article to come \cite{ParkWi2}.
\begin{proof}
 As $A$ and $B$ are both unipotent parabolic, we can find lifts ${\bf A}$ and ${\bf B}$ to SU(2,1) with trace 3. The 
condition that $AB$ is unipotent gives $\tr {\bf AB}=3\omega$, where $\omega$ is a cube root of unity. Let us denote by $z$ the 
trace of ${\bf A}^{-1}{\bf B}$. Plugging $x_1=x_2=x_4=x_5=3$, $x_3=\overline{x_5}=3\omega$ and $x_4=\overline{x_8}=z$ in 
\eqref{sumroots} and \eqref{prodroots}, we see that the trace of $[\bf A,\bf B]$ is a solution of 
the quadratic
\begin{equation}\label{quadexample}T^2-(51+|w|^2)T+657+2\Re(w^3)+21|w|^2=0,\mbox{ where }w=z-9.\end{equation}
The discriminant of this equation is equal to $|w|^4-8\Re(w^3)+18|w|^2-27$. This function of $w$ is 
negative inside the deltoid curve described in proposition \ref{deltofunc} and positive outside. This means that \eqref{quadexample} 
has two complex conjugate roots or one real double root exactly when $z$ belongs to the translated by 9 of the closure of the 
interior of the deltoid. In particular, $z$ is a loxodromic trace.
\end{proof}
Because unipotent maps are quite easy to write in the Siegel model, Proposition \ref{uniplox} can  also be obtained using explicit 
matrices.

\section{Constraints on conjugacy classes\label{section-constraints}} 
In this section we expose certain obtructions for conjugacy classes of elements in 2 generator subgroups of PU(2,1) or SU(2,1). 
The main question we adress is the following. Let $\mathcal{C}_1$ and $\mathcal{C}_2$ be two conjugacy classes in PU(2,1). 
Describe the image of the map
\begin{align}\label{mu-prod}
 \pi : \mathcal{C}_1\times \mathcal{C}_2 & \longrightarrow [{\rm PU(2,1)}]\nonumber\\
(A,B) &\longmapsto [AB],
\end{align}
where $[{\rm PU(2,1)}]$ denotes the set of conjugacy classes in PU(2,1), and $[g]$ the conjugacy class of an element. If a 
conjugacy class $\mathcal{C}$ belongs to the image of $\pi$, this means that there exists a representation $\rho$ of the fundamental 
group of the 3-punctured sphere in $G$, where the peripheral loops are mapped to elements in the corresponding conjugacy classes. 
This problem has a long history, it is a very special case of the Deligne-Simpson problem (see \cite{Kost,Simp}). 

Similarly, knowing the image of the map $\Phi_{2,1}$ defined in \eqref{Phi21} would  provide even more precise such obstructions, 
but a complete description is not known. We will provide below an example proving that $\Phi_{2,1}$ is not onto. This is of course not 
at all surprising.

However, even when one knows for some reason that a certain pair $(A,B)$ with certain  given conjugacy classes should exist, 
finding an explicit expression of it is often not at all trivial. In particular, if one knows that a conjugacy class 
$\mathcal{C}$ is in the image of $\pi$, parametrising the fiber of $\pi$ above $\mathcal{C}$, or sometimes only finding a 
preimage of $\mathcal{C}$ by $\pi$ can be a non-trivial task. The same remark can be done concerning the fibers of the map 
$\Phi_{2,1}$.

\subsection{Pairs of loxodromics\label{section-loxoprod}}
\subsubsection{The product map in PU(2,1).}
We are now going to consider the map $\pi$ when $\mathcal{C}_1$ and $\mathcal{C}_2$ are loxodromic conjugacy classes. A loxodromic 
conjugacy class in PU(2,1) is determined by a complex number of modulus greater than 1, which is defined up to an order three 
rotation around the origin, corresponding to the three possible lifts to SU(2,1). In turn, the set of loxodromic conjugacy classes in PU(2,1) 
is identified to the cylinder $\{z\in\C, |z|>1\}/\Z_3$. Denoting this cylinder by $\mathcal{C}_{{\rm lox}}$, we see that the 
half-lines with fixed argument in $\C$ correspond to the vertical lines of $\mathcal{C}_{{\rm lox}}$. All eigenvalues we will 
consider in this section are consider up to this order three rotation.
\begin{theorem}\label{theo-FW}
When $\mathcal{C}_1$ and $\mathcal{C}_2$ are loxodromic conjugacy classes, the image of the map $\pi$ contains all 
loxodromic conjugacy classes. Morever, the fibre of $\pi$ above a loxodromic conjugacy class is compact modulo the diagonal action of PU(2,1) by conjugation on 
$\mathcal{C}_1\times\mathcal{C}_2$.
\end{theorem}
This fact has been proved in the frame of real, complex and quaternionic geometry in \cite{FW}. We now sum up the argument.

Note first the map analogous to $\pi$ but for hyperbolic conjugacy classes in PU(1,1) also contains all hyperbolic conjugacy classes 
in its image. It is a simple exercise in classical Poincar\'e disc geometry to prove this, for instance 
by decomposing hyperbolic maps into products of involutions. In the case of PU(2,1), denote by 
$\mu$ and $\nu$ the attractive eigenvalues any lifts $\bA$ and $\bB$ to SU(2,1).  When the pair $(A,B)$ is reducible, that is if $A$ 
and $B$ have a common fixed point in $\C P^2$, then it is a simple exercise to verify that $\lambda_{\bA\bB}$, the attractive 
eigenvalue  of the product  $\bA\bB$  has argument equal to $\arg(\mu\nu)$. Morever, $\lambda_{\bA_B}$ can take arbitrary modulus. 
Indeed if $A$ and $B$ preserve a common complex line, then the translation length of the product, $\ell_{AB}$, can take any 
real positive value (this follows from the remark about the PU(1,1) case above). These two quantities are related by 
$e^{\frac{\ell_{AB}}{2}}=|\lambda_{\bA\bB}|$ (see the paragraph on loxodromics in section \ref{section-isom-types}). This means that reducible 
configurations correspond to a vertical line in the cycliner $\mathcal{C}_{{\rm lox}}$. In particular, 
the complement of this line is connected. The two key facts are then the following.

First,  it is a general fact in Lie groups that the map $\pi$ has maximal rank at an irreducible pair (see for instance \cite{PopR}, 
or the last section of \cite{Go3} for similar facts in a different context). This imply that the 
restriction of $\pi$ to the set of irreducible pairs is an open map. Secondly, the map $\pi$ is proper. This can be seen as a 
consequence of the Bestvina-Paulin compacity Theorem (see \cite{Best}), as in \cite{FW}. In our special case though, it can be proved in 
a more elementery way, in the spirit of the next section. As a consequence of these two facts, the image $\pi$ must contain the 
whole complement of the reducible vertical line. We refer the reader to \cite{FW} for more details.

\subsubsection{The image of $\Phi_{2,1}$.}
We are now going to adress the question of the image of the map $\Phi_{2,1}$ in the case where $A$ and $B$ are loxodromic, and we will 
see that in this case, the obstruction observed at the end of section \ref{sectionclasspairs} is the only one. More precisely we  
prove the following.
\begin{theorem}\label{pairexists}
Let $z_A$, $z_B$, $z_{AB}$ and $z_{A^{-1}B}$ be four complex numbers. Assume that $z_A$ and $z_B$ satisfy
$$f(z_A)>0\mbox{ and }f(z_B)>0,$$ 
where $f$ is the resultant function defined in Proposition \ref{deltofunc} (this means that $z_A$ and $z_B$ are loxodromic traces). 
Denote by $z$ the 4-tuple $z=(z_A,z_B,z_{AB},z_{A^{-1}B})\in\C^4$ and by $Q$ the polynomial $S^2-4P$, where $S$ and $P$ are the two 
polynomials defined in  \eqref{sumroots} and \eqref{prodroots}. The following two conditions are equivalent.
\begin{enumerate}
 \item There exists a pair $(\bA,\bB)$ of loxodromic matrices in SU(2,1) such that $\Phi_{2,1}(\bA,\bB)=z$.
 \item The inequality  $Q(z,\overline{z})\leqslant 0$ holds.
\end{enumerate}
\end{theorem}
We first normalise pais of loxodromics and relate traces and cross-ratios, as in \cite{PP,Wi3}.
To any pair of loxodromic isometries is associated  the 4-tuple of fixed points $p_A$, $q_A$, $p_B$ and $q_B$. We take here the 
convention that $p_A$ (resp. $p_B$) is the attractive fixed point of $A$ (resp. $B$) and $q_A$ (resp. $q_B$) is the repulsive fixed point of $A$ 
(resp. $B$). In the Siegel model, we can conjugate by an element of PU(2,1) and assume that 
\begin{equation}\label{normalfixes}
\bp_A=\begin{bmatrix}1\\0\\0\end{bmatrix},\,
\bq_A=\begin{bmatrix}z_1\\z_2\\1\end{bmatrix},\,
\bp_B=\begin{bmatrix}0\\0\\1\end{bmatrix}\,
\mbox{ and }\bq_B=\begin{bmatrix}1\\w_2\\w_3\end{bmatrix},
\end{equation}
\n where $z_1+\overline{z_1}+|z_2|^2=w_3+\overline{w_3}+|w_2|^2=0$ because these four points belong to the boundary of $\HdC$ 
(compare to section \ref{section-quadruples}). Any pair of loxodromic isometries is then conjugate in PU(2,1) to a pair given by 
\begin{align}\label{normal-2loxo}
& \bA=\begin{bmatrix}
     \mu & \overline{z_2}g(\mu)& z_1g(\overline{\mu}^-1)+\overline{z_1}g(\mu)\\
     0 &\mu\overline{\mu}^{-1} & z_2g(\overline{\mu}^-1)\\
     0 & 0 & \overline{\mu}^{-1} 
     \end{bmatrix}\nonumber\\
\mbox{ and }&\nonumber\\
&\bB=\begin{bmatrix}
     \overline{\nu}^{-1}& 0 & 0 \\
     w_2g(\overline{\nu}^{-1}) &\overline{\nu}\nu^{-1}& 0\\
     \overline{w_3}g(\nu)+\overline{w_3}^{-1}g(\overline{\nu}^{-1}) &\overline{w_2}g(\nu) & \nu
    \end{bmatrix},
\end{align}
\noindent where  $\mu>1$, $\nu>1$ and $g(z)=z-\bar z/z$. Using these matrices, one obtains by a direct computation the 
following expressions for $\tr \bA\bB$ and $\tr \bA^{-1}\bB$, where $\X_1=\X(p_B,p_A,q_A,p_B)$ and $\X_2=\X(p_B,q_A,p_A,q_B)$ 
are the cross-ratios computed in section \ref{section-quadruples}.
\begin{align}
 \tr\bA\bB  &=g\left(\overline{\mu^{-1}}\right)g(\nu)\X_1+ g(\mu)g\left(\overline{\nu^{-1}}\right)\overline{\X_1}
+ g(\mu)g(\nu)\X_2+g(\bar\mu^{-1})g(\bar\nu^{-1})\overline{\X_2}\nonumber\\
& g(\mu)g(\nu)+g(\overline{\mu}^{-1})g(\overline{\nu}^{-1})+\mu\nu+
\dfrac{\overline{\mu\nu}}{\mu\nu}+\dfrac{1}{\overline{\mu\nu}}\label{traceAB-KR}\\
\tr\bA^{-1}\bB  &= g\left(\overline{\mu}\right)g(\nu)\X_1+ g(\mu^{-1})g\left(\overline{\nu^{-1}}\right)\overline{\X_1}
+ g(\mu^{-1})g(\nu)\X_2+g(\bar\mu)g(\bar\nu^{-1})\overline{\X_2}\nonumber\\
& +g(\mu^{-1})g(\nu)+g(\overline{\mu})g(\overline{\nu}^{-1})+\dfrac{\nu}{\mu}+ \dfrac{\mu\overline{\nu}}{\overline{\mu}\nu}+
\dfrac{\overline{\mu}}{\nu}.\label{traceAmB-KR}
\end{align}
Note that \eqref{traceAB-KR} is obtained from \eqref{traceAmB-KR} by changing $\mu$ to $\mu^{-1}$. 
Solving the system formed by these two relations, we express the cross-ratios $\X_1$ and $\X_2$ as functions 
of traces and eigenvalues, and we obtain
\begin{align}
\X_1 = &\dfrac{|\mu|^2|\nu|^2}{(\overline{\mu}^2-\mu)(\overline{\nu}^2-\nu)(|\mu|^2-1)(|\nu^2-1|)} \nonumber\\
&\times\left(\overline{\mu}\tr\bA\bB+\nu\overline{\tr\bA\bB}+\overline{\mu}\nu\tr\bA^{-1}\bB+\tr\bA^{-1}\bB+h(\mu,\nu)\right)\\
\X_2 = &\dfrac{|\mu|^2|\nu|^2}{(\mu^2-\overline{\mu})(\nu^2-\overline{\nu}^2)(|\mu|^2-1)(|\nu^2-1|)}\nonumber\\
&\times\left(\mu\nu\tr\bA\bB +\overline{\tr \bA\bB}+\nu\tr\bA^{-1}\bB+\mu\overline{\tr\bA^{-1}\bB}+h(\overline{\mu},\nu)\right),
\end{align}
\noindent where $h$ is a function of $\mu$ and $\nu$ (we do not make it explicit here, but the exact value can be found in \cite{PP}). 
The following identity is the crucial fact to prove Theorem \ref{pairexists}.
\begin{lemma}\label{Lemma-strike}
 Using the notation defined above, it holds 
\begin{align}\label{striking}
Q(z,\overline{z})=&\left(\dfrac{(|\mu|^2-1)(|\nu|^2-1)|(\nu^2-\overline{\nu})(\mu^2-\overline{\mu})|^2}{|\mu|^4|\nu|^4}\right)^2
\nonumber\\
&\times\left(\left(|\X_1|-|\X_2|\right)^2-2\Re(\X_1+\X_2)+1\right)\left(\left(|\X_1|+|\X_2|\right)^2-2\Re(\X_1+\X_2)+1\right).
\end{align}
\end{lemma}
To prove Lemma \ref{Lemma-strike}, one needs to plug the values of $\tr \bA\bB$ and $\tr \bA^{-1}\bB$ given by \eqref{traceAB-KR} 
and \eqref{traceAmB-KR} in the polynomial $S^2-4P$. I don't know a better proof than using brute force and a computer.
\begin{proof}[Proof of Theorem \ref{pairexists}.]
We already know from the discussion at the end of section \ref{sectionclasspairs} that the non-positivity of 
$Q(z,\overline{z})$ is a necessary condition. Conversely, assume that $Q(z,\overline{z})$ is negative. 
Solving \eqref{traceAB-KR} and \eqref{traceAmB-KR} with respect to $\X_1$ and $\X_2$ gives us two complex numbers 
$x_1$ and $x_2$. Proving that a pair of matrices exists such that $\Psi(\bA,\bB)=z$ is equivalent to proving that that there exists 
an ideal tetrahedron which is formed by the fixed point of $\bA$ and $\bB$ in such a way that $\X(p_B,p_A,q_A,q_B)=x_1$ and 
$\X(p_B,q_A,p_A,q_B)=x_2$. In other words, we need to check that $x_1$ and $x_2$ lie on the cross-ratio variety.
But as $Q(z,\overline{z})\leqslant 0$, Lemma \ref{Lemma-strike} implies that the numbers $x_1$ and $x_2$ satisfy 
\begin{equation}\left(\left(|x_1|-|x_2|\right)^2-2\Re(x_1+x_2)+1\right)
 \left(\left(|x_1|+|x_2|\right)^2-2\Re(x_1+x_2)+1\right)\leqslant 0
\end{equation}
The left hand side factor is smaller than the right hand side one, and this implies that $x_1$ and $x_2$ satisfy the double 
inequality \eqref{compatcross} and they can therefore be interpreted as cross-ratios.
\end{proof}
\begin{remark}
Relations \eqref{traceAB-KR} and \eqref{traceAmB-KR} show that, when $A$ and $B$ are loxodromic, fixing $\tr\bA\bB$ and 
$\tr\bA^{-1}\bB$ amounts to fixing $\X_1$ and $\X_2$. In fact, using the normalisation \eqref{normal-2loxo}, it is possible
to compute the trace of the commutator in terms of the eigenvalues $\mu$ and $\nu$ and the cross-ratios $\X_1$, $\X_2$ and $\X_3$.
The exact expression can be found in \cite{Parktraces} or \cite{Wi4}. The value of $\X_3$ is determined up to the sign of its 
imaginary part from $\X_1$ and $\X_2$, just as $\tr[\bA,\bB]$ is determined up to the same ambiguity by $\Phi(\bA,\bB)$.
\end{remark}
In a recent preprint \cite{GongoPar}, Gongopadhyay and Parsad began a similar work for two generators subgroups of SU(3,1), and 
classified pairs of loxodromic isometries using traces and cross-ratios.

\subsection{Pairs of elliptics\label{section-paupert}}
In \cite{PopR}, Paupert has adressed the question of knowing which elliptic conjugacy classes are in the image of 
the map $\pi$ defined in \eqref{mu-prod} when $\mathcal{C}_1$ and $\mathcal{C}_2$ are two elliptic conjugacy classes in PU(2,1). 
Recall that the conjugacy class of an elliptic element is described by an (unordered) pair of angles 
$\{\theta_1,\theta_2\}\in[0,2\pi]^2$ (see the discussion on elliptics in section \ref{section-isom-types}). To fix a chart, 
we make the choice that $\theta_2\leq \theta_1$.  The set of conjugacy classes of elliptic elements is then identified with 
the quotient $\mathbb{T}^2/\mathfrak{S}$, where $\mathbb{T}^2$ is the torus $\R^2/(2\pi\Z)^2$ and $\mathfrak{S}$ is the 
reflection about the diagonal. In affine chart it appears as the (closed) subdiagonal triangle of the square $[0,2\pi]^2$, 
where the horizontal and vertical sides are identified as indicated on figure \ref{elliptic-classes}.
We can thus rephrase the problem as: if $A$ has angles $\{\theta_1,\theta_2\}$ and $B$ has angles $\{\theta_3,\theta_4\}$, what are 
the possible angles $\{\theta_5,\theta_6\}$ for the product $AB$? 
\subsubsection{Reducible cases.}
Paupert begins with analysing the reducible configurations, which are as follows.
\begin{enumerate}
\item The pair $(A,B)$ is {\it totally reducible} if $A$ and $B$ commute, that is if $A$ and $B$ have a common fixed point and 
the same invariant complex lines (see section \ref{section-isom-types}).
\item The pair $(A,B)$ is {\it spherical reducible} if $A$ and $B$ have a common fixed point.
\item The pair $(A,B)$ is {\it hyperbolic reducible} if $A$ and $B$ have a common stable complex line.
\end{enumerate}
\paragraph{Totally reducible pairs.}
In general, there are two totally reducible conjugacy classes for $AB$, which correspond to the two pairs of angles 
$\{\theta_1+\theta_3,\theta_2+\theta_4\}$ and $\{\theta_1+\theta_4,\theta_2+\theta_3\}$ 
(sums are taken mod $2\pi$). These two conjugacy classes correspond in general to two 
points $D_1$ and $D_2$ in $\mathbb{T}^2/\mathfrak{S}$. In special cases, these two points can equal (this is the case 
for instance if one of the two conjugacy classes correspond to a complex reflection about a point, which always has two equal 
rotation angles).
\paragraph{Spherical reducible pairs.}
In the case where $A$ and $B$ have a common fixed point, they can be lifted to U(2,1) as follows (here we use the ball model of $\HdC$).
\begin{equation}
 {\bA}=\begin{bmatrix}
        e^{i\theta_1} & 0 & 0\\
        0 & e^{i\theta_2} & 0\\
        0 & 0 & 1
       \end{bmatrix}
\mbox{ and }
 {\bB}=\begin{bmatrix}
        \bB' & \\
           & 1
       \end{bmatrix},\mbox{where $\bB'\in$ U(2).}
\end{equation}
The determinants of $\bA$, $\bB$ and $\bA\bB$ are respectively equal to  $e^{i(\theta_1+\theta_2)}$, $e^{i(\theta_3+\theta_4)}$ and 
$e^{i(\theta_5+\theta_6)}$, and therefore we see that
\begin{equation}\label{slope-1}
 \theta_5+\theta_6 = \theta_1+\theta_2+\theta_3+\theta_4 + 2k\pi \mbox{ with }k\in\Z.
\end{equation}
Relation \eqref{slope-1} shows that this segment has slope $-1$ in chart points corresponding to spherical reducible configurations
are contained in a line of slope $-1$ or a union of such lines. In fact, the allowed pairs of angles for $AB$ are exactly the points 
of the convex segment connecting $D_1$ to $D_2$ in the torus (which can appear as the union of two disconnected segments in affine
 chart). This fact follows for instance from the more general \cite{Bis,FW1}. However, in this special case, it can be obtained by 
analysing the action of $A$ and $B$ on the $\C P^1$ of complex lines through their common fixed point and use spherical 
geometry. This point of view is exposed in \cite{FMS}. One can verify that the integer $k$ in \eqref{slope-1} can in fact 
only take the values $0$, $-1$ and $-2$. It can be interpreted as a Maslov index (see the references in \cite{PopR}).
\begin{figure}
 \begin{center}
  \scalebox{0.6}{\includegraphics{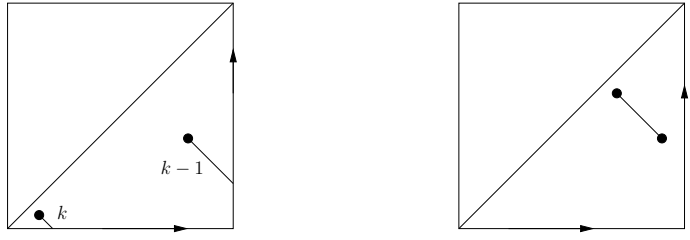}}\\
 \end{center}
\caption{Possible configurations of spherical reducible pairs.\label{elliptic-classes}
Black dots correspond to totally reducible configurations. The spherical reducible segment can appear connected (right) or 
disconnected (left) depending on the value of the integer $k$ given by equation \eqref{slope-1}  }
\end{figure}
\begin{figure}
\begin{center}
\begin{tabular}{ccc}
\scalebox{0.8}{\scalebox{0.4}{\includegraphics{5-10.eps}}} && \scalebox{0.8}{\scalebox{0.32}{\includegraphics{5-12.eps}}}\\
&&\\
\begin{minipage}{5cm}$\{\frac{\pi}{3},\frac{\pi}{3}\}$,\,$\{\frac{\pi}{4},\frac{\pi}{4}\}$ : $A$ and $B$ are reflections about points.The spherical 
reducible  segment collapses to a point. All representations are hyperbolic reducible\end{minipage}
&&\begin{minipage}{5cm}$\{\frac{2\pi}{3},0\}$,\,$\{\frac{8\pi}{5},0\}$ : $A$ and $B$ are reflections about lines. As the mirrors intersect in $\C P^1$, all 
pairs are reducible, either spherical (when the intersection is inside $\HdC$ or hyperbolic if not). \end{minipage} \\
&&\\
\scalebox{0.8}{\scalebox{0.4}{\includegraphics{5-11.eps}}} && \scalebox{0.8}{\scalebox{0.32}{\includegraphics{5-13.eps}}} \\
&&\\
\begin{minipage}{5cm}$\{\frac{\pi}{2},\frac{\pi}{3}\}$,\,$\{\frac{\pi}{4},\frac{\pi}{4}\}$ : one of the two classes is a 
reflection about a point. This makes the spherical reducible segment collapse to a point.  
\end{minipage}
 &&\begin{minipage}{5cm}$\{\frac{4\pi}{3},\frac{2\pi}{3}\}$,\,$\{\frac{2\pi}{3},0\}$ : one of the two classes is a 
reflection about a line. \end{minipage}\\
\end{tabular}
\end{center}
\caption{The image of the map $\pi$ when at least one of the two classes $\mathcal{C}_1$ and $\mathcal{C}_2$ is a complex reflection\label{polyg-image-reflection}. }
\end{figure}
\paragraph{Hyperbolic reducible pairs.}
The case where $A$ and $B$ preserve a common complex line is dealt with in a similar way. The (elliptic) product $AB$ is determined by 
two angles $(\theta_C,\theta_N)$, where $\theta_C$ is the rotation angle in the common stable complex line, and $\theta_N$ is the 
rotation angle in the normal direction. There are therefore a priori 4 families of hyperbolic reducible configurations, depending 
on the respective rotation angle of $A$ and $B$ in the common stable complex line. Let us assume that $A$ and $B$ 
respectively rotates of angles $\theta_2$ and $\theta_4$ in their common stable complex line. Then, using an adapted basis 
for $\C^3$ they can be lifted to U(2,1) (in the ball model Hermitian form):
\begin{equation}
\bA =\begin{bmatrix}
      e^{i\theta_1} & 0 & 0\\
      0 & e^{i\theta_2} & 0\\
      0 & 0 & 1
     \end{bmatrix}
\mbox{ and }
\bB=\begin{bmatrix}
     e^{i\theta_3} & \\
      & \bB'
    \end{bmatrix}
\mbox{ where }\bB'\in\mbox{U(1,1)}. 
\end{equation}
Here $\bB'$ has positive type eigenvalue $e^{i\theta_4}$ and negative type eigenvalue $1$. The product is given by
\begin{equation}
\bA\bB =\begin{bmatrix}
      e^{i(\theta_1+\theta_3)} & \\
      0 & \bC'
     \end{bmatrix},
\end{equation}
where $\bC'$ has eigenvalues $e^{i\psi_1}$ (positive type) and $e^{i\psi_2}$ (negative type). Therefore the rotation angles of $AB$ 
are given by $\theta_C=\psi_1-\psi_2$ and $\theta_N=\theta_1+\theta_3-\psi_2$. Considering determinant again, we see that 
\begin{equation}
\theta_C=2\theta_N+\theta_2+\theta_4-2\theta_1-2\theta_3\mod 2\pi. 
\end{equation}
In particular, this means that the pair $\lbrace\theta_C,\theta_N\rbrace$ lie on a (family of) lines of slope $1/2$ or $2$
(note that the slope here is only defined up to $x\longmapsto 1/x$ because of the action of the symmetry about the diagonal of 
the square).The precise range for $\theta_C$ in this reducible case is exactly the set possible rotation angles of the product 
of two elliptic elements of the Poincar\'e disc of respective angles $\theta_2$ and $\theta_4$ (see Proposition 2.3 and 
Lemma 2.1 of \cite{PopR}). In Paupert's notation \cite{PopR}, the family we have just described is denoted $C_{24}$, and the three 
others are $C_{13}$, $C_{23}$ and $C_{14}$, where each time the pair of indices gives the rotation angles in the common stable 
complex line. These four families are given by the following Proposition (Proposition 2.3 of \cite{PopR}).
\begin{proposition}
 The familly $C_{ij}$ corresponds to the points $\{\theta_C,\theta_N\}$ so that 
$\theta_C=2\theta_N+\theta_i+\theta_j-2\theta_k-2\theta_l \mod 2\pi$ (with $i,j,k,l \in\{1,2,3,4\}$ pairwise disjoint), and 
\begin{enumerate}
 \item $\theta_C\in]\theta_i+\theta_j,2\pi[$ if $\theta_i+\theta_j<2\pi$
 \item $\theta_C\in]0,\theta_i+\theta_j-2\pi[$ if $\theta_i+\theta_j>2\pi$
\end{enumerate}
\end{proposition}
\subsubsection{Allowed angle pairs.}
The segments corresponding to reducible configurations are called the \textit{reducible walls}, and their set is denoted 
$W_{\rm red}$. Then any connected component of the complement of $W_{\rm red}$ in the lower half square is called a \textit{chamber}. 
This terminology comes from the analogy with the Atiyah-Guillemin-Sternberg theorem on the image of the moment map for the Hamiltonian 
action of a Lie group on a symplectic manifold (here, the product is interpreted by Paupert as a Lie group valued moment map). 
We refer the reader to \cite{PopR} and the references therein for more information on that aspect.
Analysing the situation at an irreducible pair Paupert proves that a chamber must be either full or empty. The key facts for this are
the following.
\begin{enumerate}
 \item The map $\pi$ is a local surjection at an irreducible point.  This can be seen by checking that the rank of the 
differential of $\pi$ at an irreducible pair is maximal (i.e. equal to $2$).
\item The image of $\pi$ is closed in the space of conjugacy classes of PU(2,1) (this follows for instance from \cite{FW}).
\end{enumerate}
The question is now to decide which chambers are full or empty. Paupert does not give a general statement, but provide a series of 
criteria to answer that question. The most important one is the following.
\begin{proposition}\label{prop-local-convex}
 If either $A$ nor $B$ is a complex reflection, then the image of $\mu$ contains each chamber touching a totally reducible point 
and meeting the local convex hull of $W_{red}$ at this point.
\end{proposition}
The proof of this proposition is done by analysing the second order derivative of $\pi$ at totally reducible points 
(Lemma 2.7 and 2.8 of \cite{PopR}). Paupert then gives conditions under which the image contains the corners of the lower
triangles, which lead him to a necessary and sufficient condition for the surjectivity of the map $\pi$.
\begin{theorem}
The map $\pi$ is surjective if and only if the angle pairs of $A$ and $B$ satisfy the two inequalities
\begin{align}
     \theta_1-2\theta_2+\theta_3-2\theta_4 &\geqslant 2\pi\nonumber\\
     2\theta_1-\theta_2+2\theta_3-\theta_4&\geqslant 6\pi
\end{align}
\end{theorem}
\begin{figure}
\begin{center}
\begin{tabular}{ccc}
\scalebox{0.8}{\scalebox{0.32}{\includegraphics{5-15.eps}}} && \scalebox{0.8}{\scalebox{0.32}{\includegraphics{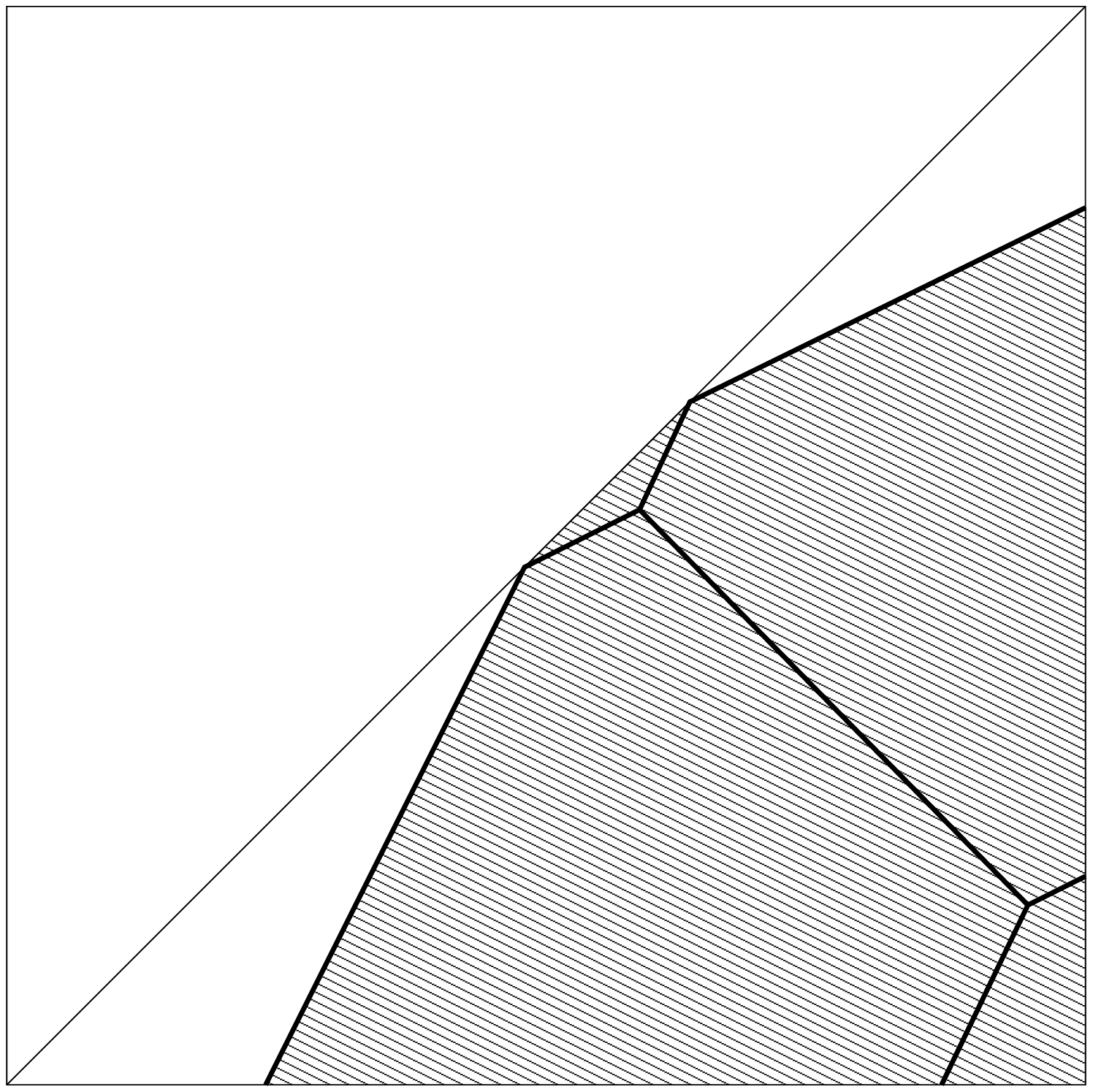}}} \\
&&\\
\begin{minipage}{5cm}$\{\frac{2\pi}{3},\frac{2\pi}{4}\}$,\,$\{\frac{2\pi}{5},\frac{2\pi}{6}\}$ : two regular elliptic pairs. Here the totally 
reducible vertices are on the boundary of the image. \end{minipage}
&&\begin{minipage}{5cm}$\{\frac{2\pi}{3},-\frac{2\pi}{4}\}$,\,$\{\frac{2\pi}{5},-\frac{2\pi}{6}\}$ : two regular elliptic pairs. Here the totally 
reducible vertices are interior points of the image.\end{minipage}\\
\end{tabular}
\end{center}
\caption{The image of  $\pi$ when the two classes $\mathcal{C}_1$ and $\mathcal{C}_2$ are regular elliptic 
\label{polyg-image-reg}. }
\end{figure}
\subsection{Pairs of parabolics, and representations of the 3-sphere in PU(2,1)}
 In \cite{ParkWi}, Parker and Will have adressed the same question as Paupert, but for parabolic 
isometries, and described all pairs $(A,B)$ such that $A$, $B$ and $AB$ are parabolic. 
Their results imply that for any  triple of parabolic conjugacy classes $(\mathcal{C}_1,\mathcal{C}_2,\mathcal{C}_3)$, there exists a 2 dimensional 
family of triples of triples $(A_1,A_2,A_3)$ such that $A_1A_2A_3=Id$ and $A_i\in\mathcal{C}_i$. In other words, if one consider 
the map $\pi$ for two parabolic conjugacy classes, then any parabolic conjugacy class is in the image. Each parabolic conjugacy 
class is determined by a unit modulus complex number $u_i$, which is the eigenvalue associated to the boundary fixed point of the 
parabolic. Denoting by $p_i$ the fixed point of $A_i$, and by $p_4=A_1^{-1}p_3$, relation \eqref{Xcross-parab} gives
\begin{equation}\X(p_1,p_2,p_3,p_4)=u_1u_2u_3.\label{X-cross-sym}\end{equation}
Parker and Will prove that to any ideal 4-tuple of points $(p_1,p_2,p_3,p_4)$, one can associate an (explicit) triple 
$(A_i)_{i=1,2,3}$ with the right conjugacy classes and satisfying \eqref{X-cross-sym}.
This gives a parametrisation of the fiber of the map $\pi$ above a parabolic class in the case where $\mathcal{C}_1$ 
and $\mathcal{C}_2$  are parabolic conjugacy classes, and shows that the set of conjugacy classes of pairs of parabolic maps 
$(A,B)$ such that $AB$ is also parabolic have dimension 5.

The next question adressed in \cite{ParkWi}, is knowing what other (non conjugate) words in the group can be 
simultaneously parabolic. The motivation for asking this comes from Schwartz's results and conjectures on the 
discreteness of triangle groups, and the generalisation of these ideas to general triangle groups (see the discussion 
in section \ref{section-Schwartz-triangle-groups}. The following result is proved.
\begin{theorem}\label{theo-parab-class}
There exists a one parameter family of conjugacy classes of pairs $(A,B)$ such that $A$, $B$, $AB$, $AB^{-1}$, 
$AB^{2}$, $A^{2}B$ and $[A,B]$ all are parabolic.
\end{theorem}
In \cite{ParkWi}, this family is obtained as the intersection of the loci where the four words $AB^{-1}$, $AB^{2}$, 
$A^{2}B$ and $[A,B]$ are parabolic, but it seems difficult to produce a direct description of this family. This makes 
difficult the study of the discreteness of this family. 
\section{The question of discreteness\label{section-discrete}}
\subsection{Sufficient conditions for discreteness}
 The following result is classical, and can be found \cite{CG}. Its consequence Theorem \ref{discret}  can be found in 
chapter 6 of \cite{Go}. 
\begin{theorem}\label{Zdense-dd}
 Let $\Gamma$ be a Zariski dense subgroup of PU($n$,1). Then $\Gamma$ is either dense or discrete. 
\end{theorem}
A subgroup $\Gamma$ of PU($n$,1) is Zariski dense if and only if it acts on $\HnC$ with to stable 
proper totally geodesic subspace: this is Zariski density for the structure of real algebraic group 
of PU($n$,1). 
\begin{proof}
 Let $G=(\overline\Gamma)_0$ be the identity component of the closure
of $\Gamma$. Because $G$ is closed and connected, it is a Lie subgroup
of PU($n$,1). Let $\mathfrak{g}$ be its Lie algebra. Any
$\gamma\in\Gamma$ normalises $G$, and therefore
$Ad(\gamma)(\mathfrak{g})=\mathfrak{g}$. The latter condition on
$\gamma$ being algebraic, the normaliser $\mathcal{N}$ of
$\mathfrak{g}$ in PU($n$,1) is an algebraic subgroup of PU($n$,1) that
contains $\Gamma$. As $\Gamma$ is Zariski dense, we have
$\mathcal{N}=$PU($n$,1). Therefore $G$ is normal in PU($n$,1), and
because PU($n$,1) is simple $G$ is either trivial or equal to
PU($n$,1). In the first case, $\Gamma$ is discrete, and in the second
one it is dense.
\end{proof}
\n This result provides a sufficient condition for discreteness  when combined with the following lemma.
\begin{lemma}\label{openelli}
The set of regular elliptic elements of PU($n$,1) is open. In
particular, the set of elliptic elements of PU($n$,1) contains an open
set.
\end{lemma}
\begin{proof}
Let $A$ be a regular elliptic element in PU($n$,1), and ${\bf A}$ of it to U($n$,1). 
There exists a neighbourhood $\mathcal{U}$ of ${\bf A}$ in U($n$,1) containing only 
matrices with pairwise distinct eigenvalues. Because ${\bf A}$ has a negative
eigenvector, there exists an open $V\subset U$ where any matrix has a
negative eigenvector. Projecting $V$ to PU($n$,1) gives the result.
\end{proof}
As a direct consequence, one obtains
\begin{theorem}\label{discret}
A Zariski dense subgroup $\Gamma$ of PU($n$,1) such that the identity
is not an accumulation point of elliptic elements of $\Gamma$ is
discrete. 
\end{theorem}
In particular this implies that a Zariski dense subgroup $\Gamma$ with no elliptic 
elements is discrete. This result is not true in PSL(2,$\C$) (see for instance \cite{L.Green}). 
This difference is well illustrated by the comparison of the trace
functions in SU(2,1) and SL(2,$\C$). The image by the trace function
of the set of elliptic matrices in SL(2,$\C$) is the interval $(-2,2)$
which has empty interior. In contrast the set of regular elliptic
matrices in SU(2,1) projects by the trace onto the inside of the
deltoid curve depicted in figure \ref{delto}.

In fact if one thinks of $\Gamma$ as a subgroup of PGL(3,$\C$), being Zariski dense means having no proper complex totally 
geodesic subspace. This condition is not sufficient to give Theorem \ref{Zdense-dd}. Indeed, PU(3,1) contains PO(3,1) which 
is a copy of PSL(2,$\C$) and thus contains totally loxodromic non discrete subgroups in the stabiliser of a real plane.

In theory, Theorem \ref{discret} should allow proofs of the discrenetess of a subgroup of PU(2,1) that it contains no elliptic
elements, but to my knowledge it has never been done. Sandler derived in \cite{Sa} a beautiful combinatorial formula that 
expresses the trace of any element in the special case of ideal triangle group. This formula was generalised by 
Prattoussevitch \cite{Pra} to the case of any triangle group. However, these formulae are more useful as sufficient 
conditions for non-discreteness, as in \cite{Pra1}.
\subsection{Necessary conditions for discreteness. }
As any Lie group, PU(2,1) has Zassenhaus neighbourhoods, that is neighbourhoods $U$ of the identity element such that for 
any discrete group $\Gamma\subset$ PU(2,1), the group generated by $\Gamma\cap U$ is elementary (see section 4.12 of \cite{Kapo} 
or chapter 8 of \cite{Rag}). In the frame of PSL(2,$\C$), the J\o rgensen inequality gives a quantitative statement that has the 
same meaning. It is usually stated as follows (see \cite{Jorg} or section 5.4 of \cite{Bear}). For any matrices $A$ and $B$ in 
SL(2,$\C$) such that  the corresponding subgroup of  PSL(2,$\C$) is  discrete and non-elementary, it holds
\begin{equation}|\tr^2(A)-4|+|\tr [A,B]-2|\geq 1.\label{jorg}\end{equation}
 In the special case where $A$ is parabolic this result can
be given a slightly simpler statement (this is case 1 in the proof of
the J\o rgensen inequality given in \cite{Bear}) and is known as the
Shimizu lemma. Though the J\o rgensen inequality does not add much in theory, it can be useful in practice for instance when one consider a family 
of examples and one want to decide what are the discrete groups in it as it gives a very simple condition to decide when a group 
is non discrete.  In the frame of PU(2,1), this situation happens quite often (see for instance in \cite{DPP2}), and there have 
been many generalisations of the J\o rgensen inequality. However such a simple statement as \eqref{jorg} has not been given. The 
strength  of \eqref{jorg} is the fact that the condition is expressed in terms of traces of elements in the group, and the trace of 
element is an easy information to get.The J\o rgensen inequality can be stated in a different way. Denote by $[z_1,z_2,z_3,z_4]$ the 
usual cross-ratio in $\C P^1$. If $A\in$ SL(2,$\C$) either elliptic or loxodromic with fixed points $\alpha$ and $\beta$ in $\C P^1$, 
then if either
\begin{align}\label{jorg-cross}
&|\tr^2 A-4|\left([B(\alpha),\beta,\alpha,B(\beta)]+1\right)<1\nonumber\\
\mbox{ or }&\nonumber\\
&|\tr^2 A-4|\left([B(\alpha),\alpha,\beta,B(\beta)]+1\right)<1\nonumber,
\end{align}
the group generated by $A$ and $B$ is non-discrete or elementary. Jiang, Kamiya and Parker \cite{JKP} have generalised this 
cross-ratio version of Jorgensen's inequality under the following form,  which is to my knowledge the most accurate result to 
this day. The result they obtained holds for pairs $(A,B)$ where $A$ or $B$ is loxodromic or a complex reflection. A similar 
but slightly less accurate result had been obtained by Basmajian and Miner in the beautiful article \cite{BM}. Here is the main 
result of \cite{JKP} ($\X$ is the Koranyi-Reimann cross-ratio see section \ref{section-quadruples}).
\begin{theorem}\label{JKPJorg}
Let $A$ and $B$ be two elements of PU(2,1) with $A$ either loxodromic
or a complex reflection about a line. In both cases, let $p$ and $q$
be two distinct fixed points of $A$ on $\partial\HdC$.  Let $\lambda$
be the dilation factor of $A$ and $M$ be the quantity
$|\lambda-1|+|\lambda^{-1}-1|$. If one of the following condititions
is satisfied, then the group generated by $A$ and $B$ is either
elementary or non discrete.
\begin{enumerate}
\item $M\left(|\X(B(p),q,p,B(q))|^{1/2}+1\right)<1$
\item $M\left(|\X(B(p),p,q,B(q))|^{1/2}+1\right)<1$
\item $M<\sqrt{2}-1$ and
$$|\X(B(p),p,q,B(q))|+|\X(B(p),q,p,B(q))|<\dfrac{1-M+\sqrt{1-2M+M^2}}{M^2}$$
\item $M + |\X(p,q,B(p),B(q))|^{1/2}<1$
\end{enumerate}
\end{theorem}
Other results in the same flavour can be found in  \cite{Kam1,Kam2,KPII,KPIV,KPT,Parkshimizu,ParkHeis,XJ,XJW}.
Applying Theorem \ref{JKPJorg} requires to know the fixed points of $A$ and $B$, whereas the classical inequality is purely in 
terms of traces. Finding the fixed points of an element is an elementary operation in itself, but it can become quite tricky when 
working with parameters. On the opposite, computing a trace is straightforward. To work more efficiently with families of examples, 
obtaining a polynomial J\o rgensen inequality expressed purely in terms of traces of words in $A$ and $B$ and without assumption on 
the conjugacy classes of $A$ and $B$ would be interesting.

One of the main applications that this generalisations of J\o rgensen's inequality have found is that of estimating the 
volume of complex hyperbolic manifolds, see \cite{HersPau} and \cite{ParkVol}. In particular, if $\Gamma$ is a discrete 
subgroup of PU($n$,1) containing a parabolic element $P$, it is possible to use Shimizu's lemma in complex hyperbolic space 
to produce subhorospherical regions that are invariant under the group generated by $P$. This leads to estimates on the volume 
of (finite volume) complex hyperbolic manifolds (see \cite{HersPau} and \cite{ParkVol}). See also \cite{KiKi} where the bounds on 
volumes of cusps obtained in \cite{HersPau,ParkVol} and \cite{ParkVol} are improved, and \cite{Hwang} where 
the author uses arguments from algebraic geometry. The specific case of $\HdC$ is also studied in \cite{KimJoon}.
For generalisations to the frame of quaternionic hyperbolic geometry, see \cite{KimDa,KimPark}.
\subsection{Building fundamental domains\label{fundadoms}}
To prove that a given subgroup of PU(2,1) is discrete the main method that has been used is to construct a fundamental 
domain or at least of domain of discontinuity. The first modern examples of discrete subgroups of PU(2,1) were given by Mostow in 
his famous \cite{Most}. There he was describing non-arithmetic lattices of PU(2,1), and constructed explicit fundamental domains 
for the action of these groups (see section \ref{section-higher-nonarithmetic}). Knowing a fundamental domain for a group $\Gamma$, 
one can obtain via Poincar\'e's polyhedron theorem (see \cite{FZ,Parbook}) a presentation for $\Gamma$. The main problem is to be 
able to construct such a domain. The famous Dirichlet procedure gives a way to construct such a domain, this is what Mostow did. 
The Dirichlet domain for a group $\Gamma$ centred at a point $x_0\in\HdC$ is defined as the region
\begin{equation}\label{diridom}
\mathcal{D}_{x_0,\Gamma}= \underset{\gamma\in\Gamma}{\bigcap}\left\{ z\in\HdC, d(x_0,z)<d(z,\gamma x_0)\right\}.
\end{equation}
The group $\Gamma$ acts properly discontinuously on $\HdC$ if and only if $\mathcal{D}_{x_0,\Gamma}$ is non-empty.
Clearly, hypersurfaces equidistant from two points play a crucial role in this construction. They are commonly called 
\textit{bisectors} in the field. In order to understand the combinatorics of the Dirichlet domain domain, it is necessary to 
describe the intersections of the various faces, and therefore one needs to understand the intersection of (at least) two given 
bisectors. As any real hypersurface in complex hyperbolic space a bisector are not totally geodesic, ant it separates $\HdC$ in two 
non-convex half spaces.  In particular the intersection of two bisectors can be quite complicated : 
it is sometimes not connected. However, Bisectors appearing in a Dirichlet construction have the additional property of 
being \textit{coequidistant} from the basepoint $x_0$ (see chapter 9 of \cite{Go})  . This simplifies the study of their intersections, 
as it implies that their pairwise intersections are connected (this is Theorem 9.2.6 in \cite{Go}). General bisector intersections 
are studied in \cite{Go}. The following proposition (which could be taken as a definition of bisectors) is often useful, as it 
provides more geometric information on bisectors. 
\begin{proposition}\label{spines}
Let $\mathcal{B}$ be a bisector. There exists a unique complex line $L$ and a unique geodesic $\sigma$ contained in $L$ 
such that 
\begin{equation}
\mathcal{B}=\Pi_L^{-1}\left(\sigma\right),
\end{equation}
where $\Pi_L$ is the orthogonal projection on $L$. The complex line $L$ and the geodesic $\sigma$ are respectivelly called 
the \textit{complex spine} and the \textit{real spine} of $\mathcal{B}$.
\end{proposition}
Because of Proposition \ref{spines}, Mostow refers to bisector as {\it spinal surfaces} in \cite{Most}. 
The following facts are consequence of Proposition \ref{spines}.
\begin{itemize}
\item A bisector admits a foliation by complex lines that are the fibers of $\Pi_L$ above $\sigma$. 
In other words, a bisector is a $\C$-sphere.
\item A bisector admits a (singular) foliation by the set of real planes that contain the real spine $\sigma$ 
(see chapter 5 of \cite{Go})
\end{itemize}
A great deal of information concerning bisectors, and their extensions to $\C P^2$ as extors is gathered in chapters 5, 8 and 
9 of \cite{Go}. The Ford domain, which is a variant of the Dirichlet domain where the center is a boundary point, has also been 
generalised to the frame of complex hyperbolic geometry. Bisectors appear there as faces just as in the Dirichlet domain. 
The interested reader will find examples of Dirichlet and Ford constructions in \cite{Der,FFP,FJP2,GP,Most,ParFord,Phillips}. 
However in most of the cases, the authors do not prove directly that the Dirichlet or Ford region is non empty. The method used is 
in fact more often the following. One starts with the conjecture that a given group is discrete, and then try to produce a candidate 
fundamental polyhedron as in \eqref{diridom}, but for  elements in a finite subset of the considered group (hopefully small).  
If this is achieved, then one try to prove that the obtained polyhedron is indeed fundamental for $\Gamma$ applying Poincar\'e's 
Polyhedron theorem. We refer the reader to \cite{Der} where a good discussion of this method can be found.
The fact that the construction of the candidate polyhedron fails can mean two things. Either the group is not discrete, or 
the choice of the basepoint is bad: it gives a Dirichlet domain with a very large number of faces (possibly even infinite 
see for instance \cite{GP2}).

Since Mostow's work, different techniques have been developed to produce fundamental domains. In particular different classes of 
hypersurfaces have been used to produce faces of polyhedron. One natural idea is to generalise bisectors by replacing them by 
hypersurfaces that are foliated by totally geodesic subspaces. This leads to the notion of {\it $\C$-surface} or {\it $\R$-surface}. 
These surfaces are typically diffeomorphic to $\R\times \HuC$ or $\R\times \HdR$. Examples of these were developed in 
\cite{FZ,FJP,FK,FK1,S,S1,S3,Schbook,Wi2,Wi7}.  The $\C$-surfaces are somewhat easier to handle than $\R$-spheres because of 
the duality between the Grasmanian of complex lines in $\HdC$ and $\C P^2\setminus \overline{\HdC}$ induced by the Hermitian form.
Any complex line $L$ in $\HdC$ is the projectivisation of a complex plane $P$ in $\C^3$, which is orthogonal for the Hermitian 
form to a linear subspace $\C v\subset \C^3$. The vector $v$ is called polar to $L$. As a consequence, any $\C$-surface corresponds 
to a curve in the outside of $\HdC$.
\begin{proposition}\label{dualcurve}
For any $\C$-surface $\Sigma$, there exists a curve $\gamma: \R\longrightarrow \C P^2\setminus \overline{\HdC}$ such that
$$\Sigma = \underset{t\in\R}{\bigcup}\gamma(t)^{\perp}$$
\end{proposition}
As seen above, bisectors are the first examples of $\C$-surfaces. In this case, the real spine $\sigma$ of the bisector can 
be extended as a circle in $\C P^2$ and the curve $\gamma$ is just the complement of $\sigma$ in this circle.
Some $\R$-surfaces analogous to bisectors have been constructed (see \cite{PP1,Wi7}, by taking the inverse image of a 
geodesic $\gamma$ under the orthogonal projection onto a real plane containing  $\gamma$. These are called flat packs 
\cite{PP1} or spinal $\R$-surfaces \cite{Wi7} (see also the survey article \cite{PP2}). More sophisticated constructions 
involving $\R$-spheres can be found in \cite{Schbook}. 

The typical situation when building a fundamental domains with $\C$-surfaces or $\R$-surfaces is the following. One want to control 
the relative position of two such surfaces $S$ and $S'$ obtained respectively as
$$S=\bigcup_{t\in\R} V_t \mbox{ and } S'=\bigcup_{t\in\R} V'_t,$$
where $(V_t))_t$ and $(V'_t)_t$ are the totally geodesic leaves of $S$ and $S'$ (either real planes or complex line).
One needs to prove that 
\begin{enumerate}
 \item For any $(t_1,t_2)$ in $\R\times\R$ the two leaves $V_{t_1}$ and $V_{t_2}$  and the two leaves 
$V'_{t_1}$ and $V'_{t_2}$ are disjoint.
 \item For any $(t_1,t_2)$ in $\R\times\R$ the two leaves $V_{t_1}$ and $V'_{t_2}$ are disjoint.
\end{enumerate}
The first condition guarantees that the $V_t$'s and $V'_t$'s indeed foliate $S$ and $S'$ respectively, 
and the second conditions ensures that $S$ and $S'$ are disjoint. A natural way of doing so is to use the 
symmetries carried by the considered leaves. Let us call $\sigma_t$ and $\sigma'_t$ the symmetries associated with 
$V_t$ and $V'_t$. They are antiholomorphic involution when $V_t$ and $V'_t$ are real planes, and complex reflections 
about lines when $V_t$ and $V'_t$ are complex lines. Showing that $V_{t_1}$ and $V'_{t_2}$ are disjoint amounts to proving 
that the composition $\sigma_{t_1}\circ \sigma'_{t_2}$ is loxodromic (see Proposition 3.1. of \cite{FZ}).
This can be done by showing that for any parameters  $t_1$ and $t_2$ the trace $\tr (\sigma_{t_1}\circ\sigma'_{t_2})$ 
remains outside the deltoid (see Proposition \ref{deltofunc} and figure \ref{delto}). This can be quite subtle, especially when 
$S$ and $S'$ must be tangent at infinity (this happens for instance when the group contains parabolic elements):  in that 
case, the two leaves that correspond to the tangency point give a trace which is on the deltoid curve. The computations involved 
are often easier when working with $\C$-surfaces than $\R$-surfaces. 
Indeed, the product of two complex reflections about complex lines always has real trace,  and therefore proving that $S$ and $S'$ 
are disjoint amout to minimizing a function $\R^2\longmapsto \R$. On the other hand, with real reflection, the trace of 
$\sigma_{t_1}\circ\sigma_{t_2}$ is a complex number, and one want to show that it is outside the deltoid curve. To transform this 
into a minimisation problem, one needs to apply first the polynomial $f$ defined in Proposition \ref{deltofunc}, which has degree 4, 
and makes the computation harder. Another very good example of how the situation can be complicated is Schwartz's construction of 
$\R$-spheres for the last ideal triangle group (see chapter 19 and 20 of \cite{Schbook}).


\section{Triangle groups\label{section-triangle-groups}}
Among the most accessible examples of groups acting on the complex hyperbolic plane are  \textit{complex hyperbolic triangle groups}. 
We will denote by $\Gamma(p,q,r)$ the group of isometries of the Poincar\'e disc generated by three symmetries 
$\sigma_1$, $\sigma_2$ and $\sigma_3$ about the sides of a triangle having angles $\pi/p$, $\pi/q$ and $\pi/r$, where $1/p+1/q+1/r<1$. 
In particular the elliptic elements $\sigma_1\sigma_2$, $\sigma_2\sigma_3$, and $\sigma_3\sigma_1$ have respective orders $2p$, $2q$ and 
$2r$. The subgroup of $\Gamma(p,q,r)$ containing holomorphic (or orientation preserving) isometries is generated by 
$\sigma_1\sigma_2$ and $\sigma_2\sigma_3$.
\subsection{Schwartz's conjectures on discreteness of triangle groups\label{section-Schwartz-triangle-groups}}
A complex hyperbolic $(p,q,r)$ triangle group is a representation of $\Gamma(p,q,r)$ to PU(2,1), 
such that $I_k=\rho(\sigma_k)$ is a complex reflection about a line, conjugate in ball coordinates to $(z_1,z_2)\longmapsto(-z_1,z_2)$. We will 
denote by $L_k$ the complex line fixed by $I_k$, and refer to it as its mirror. The angles between the mirrors are the same as the ones
for the corresponding geodesics in the Poincar\'e disc. One of the reasons for which triangle groups have 
been intensively studied is the fact that for given $(p,q,r)$ the moduli space of complex hyperbolic $(p,q,r)$ triangle groups is 
quite simple.
\begin{proposition}\label{prop-moduli-triangle-groups}
For any triple of integers $(p,q,r)$ such that $1/p+1/q+1/r<1$, there exists exactly one parameter family of $(p,q,r)$-triangle 
groups up to PU(2,1) conjugacy. 
\end{proposition}
The parameter in question is in fact the angular invariant of the triangle formed by the intersections of the mirrors of $I_1$, $I_2$ 
and $I_3$. This parameter is often denoted by $t$ in this context, so that $t=\alpha(L_1\cap L_2,L_2,\cap L_3, L_3,\cap L_1)$ in our 
notation. In fact, the set of allowed values for $t$ is an interval which we will denote by $\mathcal{I}_{p,q,r}$. Conjugating by an 
antiholomorphic isometry amounts to changing $t$ to $-t$, so that $\mathcal{I}_{p,q,r}$ is symmetric about the point $t=0$, which 
corresponds to representations preserving a real plane which are easily seen to be discrete and faithful. The two endpoints of the 
interval $\mathcal{I}_{p,q,r}$ are not difficult to compute, but are not really relevent here (see for instance \cite{Pra} for 
examples).

The first examples of complex  triangle groups that have been studied are ideal triangle groups, that is when  $p=q=r=\infty$ in \cite{GP,S,S1,S3}. 
In this case the three products $I_kI_i{k+1}$ are parabolic. The result of this series of articles is the following, that had been 
conjectured by Goldman and Parker in \cite{GP}, and proved --twice-- by  Schwartz in \cite{S,S3}.
\begin{theorem}
An ideal triangle group is discrete and isomorphic to the free product of three copies of $\Z/2\Z$ if and only if the triple product 
$I_1I_2I_3$ is not elliptic.
\end{theorem}
In this case, the parameter $t$ is the Cartan invariant of the (parabolic) fixed points of the three words $I_kI_{k+1}$ and  
$\mathcal{I}=[-\pi/2,\pi/2]$ (see section \ref{section-triple-cross}). The subset of $\mathcal{I}$ where $I_1I_2I_3$ is non-elliptic 
is a closed subinterval $\mathcal{I}^0\subset\mathcal{I}$ which is symmetric about $0$. The endpoints of the interval 
$\mathcal{I}_0$ correspond to the so-called \textit{last ideal triangle group}. This group has very interesting properties, 
which we will discuss in section \ref{SCRWLC}. The striking fact here is that discreteness 
and faithfulness are governed by the conjugacy class of one element in the group. In fact, Schwartz proved that 
representations corresponding to point in $\mathcal{I}\setminus\mathcal{I}_0$ are never discrete.

A natural question is then to know how much of this behaviour remains true for other triangle groups. In his survey article \cite{S2}, 
Schwartz stated a series of conjecture predicting when these groups are discrete. To state these conjectures, we fix a labelling of 
the lines $L_1$, $L_2$ $L_3$ such that $p\leqslant q\leqslant r$ and denote by $W_A$ and $W_B$ the two words
\begin{equation}\label{eq-WA-WB}
W_A=I_1I_2I_3I_2\mbox{ and }W_B=I_1I_2I_3.
\end{equation}
Following Schwartz, we will say that a triple $(p,q,r)$ has type $A$ when $\rho_t(W_A)$ 
becomes elliptic before $\rho_t(W_B)$ as $t$ varies from $0$ to $\max(I_{p,q,r})$. We will say that it has type $B$ otherwise.
Fix the values of $p$, $q$ and $r$. Schwartz's conjectures are as follows. 

\textbf{Conjecture 1 }: The set of discrete and faithful representations $\rho_t$ of $\Gamma(p,q,r)$ consists of those values of $t$ for 
which neither $\rho_t(W_A)$ nore $\rho_t(W_B)$ is elliptic. These values form a closed subinterval 
$\mathcal{I}^0_{p,q,r}\subset \mathcal{I}_{p,q,r}$. In other words, the isometry type of $W_A$ (resp. $W_B$) controls discreteness and faithfulness for type $A$ triples 
(resp. trype $B$ triples). \\

\textbf{Conjecture 2 }: If  $p<10$ the $(p,q,r)$ has type $A$. If $p>13$, $(p,q,r)$ has type $B$.\\

\textbf{Conjecture 3 }: If $(p,q,r)$ has type $B$, then any discrete infinite representation is an embedding and correspond to a 
point in $\mathcal{I}_{p,q,r}^0$. If it has type $A$, the there exists a countable familly of non-faithful, discrete, infinite 
representations corresponding to values $(t_n)_{n\in\mathbb{N}}$ outside $\mathcal{I}^0_{p,q,r}$.\\

\textbf{Conjecture 4 }: As $t$ increases from $0$ to the boundary of $I^0_{p,q,t}$, the translation length of $\rho_t(W)$ decreases 
monotonically, where $W$ is any word of infinite order in $\Gamma_{p,q,r}$.\\

The behaviour predicted by conjecture 3 for triples of type $A$ has been indeed described in the case of $(4,4,\infty)$ and 
$(4,4,4)$-triangle groups (see \cite{S4,WyG}): there exists discrete representations of this group for which $W_A$ is (finite order) 
elliptic and $W_B$ is loxodromic. In the case of the $(4,4,4)$-triangle group, these representations correspond to values of the 
parameter $t_k$ of the parameter $t$ for which $\rho_{t_k}(W_A)$ has order $k$. For instance the value $t_5$ corresponds to a lattice 
that has been analysed by Deraux in \cite{Der2}.

A striking fact in these conjectures is that for each fixed triple $(p,q,r)$, the discreteness and faithfulness of 
a $(p,q,r)$-triangle group is controlled by the isometry type of a single element.  Triangle groups contain 2-generator subgroups of 
index two, that are generated by a $\C$-decomposable pair (see Remark \ref{RdecCdec} in Section \ref{subsection-double-root}). It is 
thus a natural question to try to generalize this to more general 2-generator subgroups of PU(2,1). A natural place to start would be 
to begin by fixing  a compatible choice of conjugacy classes for $A$, $B$ and $AB$, and examine the classes of triangle groups in 
the corresponding moduli space. 
\begin{itemize}
 \item The case where $A$, $B$ and $AB$ are parabolic generalises ideal triangle groups. Indeed, if $\Gamma=\la I_1,I_2,I_3\ra$ is an 
ideal triangle group, the products $A=I_1I_2$, $B=I_2I_3$ and $AB=I_1I_3$ all are parabolic (even unipotent).
In \cite{ParkWi}, a system of coordinates on the set of pairs $(A,B)$ surch that $A$, $B$ and $AB$ are parabolic is produced. In these 
coordinates, it is easy to spot families of discrete groups that are commensurable to those studied in \cite{FK2,GuP,GuP2} and a 
special case of \cite{Wi7}. All these example exhibit this kind of behaviour : discreteness is controlled by a single element of the 
group.
\item If one fixes three elliptic conjugacy classes $\mathcal{C}_1$, $\mathcal{C}_2$ and $\mathcal{C}_3$, it is not always true 
that there exists two elements of PU(2,1) such that $A\in \mathcal{C}_1$, $B\in \mathcal{C}_2$ and $AB\in \mathcal{C}_3$ (see section 
\ref{section-paupert}).  However, even when one knows that the choice of conjugacy classes is compatible, it is 
not at all trivial to produce an efficient parametrisation of the set of the corresponding pairs. 
\end{itemize}
The following question seems natural in this context. Let $F_2=\la a,b\ra$ be the free group of rank 2. Does there exists a finite list $(w_1,\cdots,w_k)$ such that any representation 
$\rho : F_2\longmapsto$PU(2,1) mapping all the $w_i$'s to non-elliptic isometries is discrete and faithful?

The Schwartz conjectures as well as the above question can all be stated in terms of traces of elements of the group. In \cite{Sa}, 
Sandler has derived a beautiful combinatorial formula to compute traces of words in an ideal triangle group, that has been generalized
by Prattoussevitch in \cite{Pra} to other triangle groups. It would be a tremendous progress in the field to have a sufficiently good 
understanding of the behaviour of the traces to be able to prove discreteness from this point of view. Quoting Schwartz in \cite{S2}: 
``I think that there is some fascinating algebra hiding behind the triangle groups -- in the form of the behavior of the trace 
function -- but so far it is unreachable''. Since then, progresses have been made on the understanding of traces, but nothing 
sufficiently accurate yet to attack these questions from this point of view. In particular, one knows from Theorem 
\ref{theo-FVLawton} that for any word $w\in F_2$, there exists a polynomial $P_w\in \mathbb{Z}[x,\overline x]$, where 
$x=(x_1,x_2,x_3,x_4,x_5)$ such that for any representation $\rho : F_2\longrightarrow SU(2,1)$ it holds
\begin{equation}\label{polytrace}
 \tr(\rho(w))=P_w(T,\overline{T})\mbox{, where }T=(\tr A,\tr B,\tr AB,\tr A^{-1}B,\tr [A,B]).
\end{equation}
Recall that the polynomial $P_w$ is only unique up to the ideal generated by  Relation \eqref{trace equation} in Proposition \ref{S and P exist}. Sandler's and 
Prattoussevitch's formulae appear thus as an explicit version of this polynomial in the special case of groups generated by $\C$-decomposable pairs. 
\subsection{Higher order triangle groups and the search for non-arithmetic lattices.\label{section-higher-nonarithmetic}}
A natural generalisation is to increase the order of the complex reflections, and consider groups generated by three 
higher order complex reflections. It turns out that such groups provide example of lattices in PU($n$,1).
Lattices in PU($n$,1) are far from being as undestood as in other symmetric spaces of non-compact type.  
In all symmetric spaces of rank at least 2, all irreducible lattices are arithmetic (\cite{Marg}) as well as in the rank 1 
symmetric spaces ${\bf H}^n_\mathbb{H}$ and ${\bf H}^2_\mathbb{H}$ ( \cite{Corl} and \cite{GroSch}). On the other hand, examples of 
non arithmetic lattices have been produced in ${\bf H}^n_{\R}$ for any $n\geqslant 2$ (\cite{GroPS}). In the case of complex 
hyperbolic space $\HnC$, only a finite number of examples in dimension $n=2$ are known (see \cite{Most,DM} and the more recent 
\cite{DPP2}), and one example in dimension $n=3$ (\cite{DM}).
We refer the reader to the survey article \cite{Parklattices} and the references therein for an account of what is known on the 
question of complex hyperbolic lattices. All examples known of non-arithmetic lattices in $\HdC$ are examples of groups of the 
following type.
\begin{definition}
A (higher order) symmetric triangle group is a group generated by three complex reflections $R_1$, $R_2$ and $R_3$ such that there 
exists an order three elliptic element $J$ which conjugate cyclically $R_i$ to $R_{i+1}$ (indices taken mod. 3).
\end{definition}
In particular being symmetric implies that the three complex reflections have the same order, which we will denote by $p$. The recent 
work that has been done on these groups find its root in Mostow's famous \cite{Most}. There, Mostow constructed 
the first examples of non-arithmetic lattices in $\HdC$, which are symmetric triangle groups with $p\in\{3,4,5\}$. Mostow's 
examples have been revisited by Deligne and Mostow in \cite{DM}, and the list of known non-arithmetic lattices in complex 
hyperbolic space extended. The question of knowing if there existed other examples of such non-arithmetic lattices remained 
open until very recently: in \cite{DPP,DPP2}, Deraux, Parker and Paupert have constructed new examples of non-arithmetic 
lattices in PU(2,1).

A symmetric triangle group is determined by a (symmetric) triple of complex lines which are the mirrors of the $R_i$'s and 
the integer $p$. This implies that for given $p$, the set of symmetric triangle groups has real dimension 2: 
\begin{itemize}
 \item the relative position of two complex lines is 
determined by one real number (which is their distance if the don't intersect and their angle if they do),
\item  once the pairwise relative position is known, the triple is determined by an angular invariant similar to the triple 
ratio of three points described in section \ref{invariants} (Mostow's \textit{phase shift}). 
\end{itemize}
The groups described by Mostow have the additional features that the two words $R_1R_2$ and $R_1R_2R_3=(R_1J)^3$ are finite order 
elliptic elements. It is thus very natural to explore systematically symmetric triangle groups having the property that $R_1J$ and 
$R_1R_2$ have this property. In \cite{DPP,DPP2}, the authors call these groups \textit{doubly elliptic} (in fact they allow $R_1R_2$ 
to be parabolic). Doubly elliptic symmetric triangle groups have been classified by Parker in \cite{P3} for $p=2$, and in 
\cite{ParkPau} by Parker and Paupert for $p\geqslant 2$. 

The main result of \cite{ParkPau} asserts that for each given value of $p$, a symmetric triangle group 
is either one of Mostow's groups or a subgroup of it, or belong to a finite list of groups called \textit{sporadic}. 
Sporadic groups appear thus as a natural place to look for new non-arithmetic lattices. The term \textit{sporadic} comes from the 
following fact. It is possible to translate the condition of double ellipticity into a trigonometric equation involving 
the eigenvalues of $R_1J$ and $R_1R_2$ (see section 3 of \cite{ParkPau}). The solutions of these equations form two continuous 
families and a finite set of 18 isolated solutions. These isolated solutions give the sporadic groups. One of the two continuous 
families leads to Mostow's examples, and the other one to subgroups of Mostow's groups. It should be noted that the resolution of 
these trigonometric equations is not trivial and makes use of a result of Conway and Jones on sums of cosines of rational multiples of $\pi$ 
(see Theorem 3.1 of \cite{P3}, and \cite{CJ}). This illustrates in particular the fact that finding an efficient parametrisation of a 
given family of groups is often difficult.

In \cite{DPP}, it is proved by use of Jorgensen type inequalities that at most finitely many sporadic groups are 
discrete, and conjectured that ten of these sporadic groups are non-arithmetic lattices, and among these ten, three are cocompact. 
In \cite{DPP2}, this result is proved for for five of the ten remaining groups. The crucial part of the work is the construction of a 
fundamental domain for the action of each of these groups on $\HdC$. 
\begin{remark}
\begin{enumerate}
  \item The ten sporadic groups are neither commensurable to one another, nor to any of Mostow's of Deligne and Mostow's groups.
The  invariant used to tell apart the commensurability classes  is the trace field of the adjoint representation of the considered 
groups (see section 8 of \cite{DPP2}).
 \item In \cite{Most} used Dirichlet's method to construct fundamental domains for the groups he studied. His construction has been 
revisited in \cite{Der}, where Deraux filled in gaps in Mostow's original proof. A different  and simpler construction of fundamental 
domains for Mostow's groups has been given in \cite{DFP}, where the dimension 3 faces of the polyhedron are not necessarily bisectors, 
but also cones over totally geodesic submanifolds. In \cite{DPP2}, Deraux, Parker and Paupert have proposed a new way of constructing
fundamental domains for higher order triangle groups. Regardless of the method used to construct a fundamental domain, the main 
difficulty is to analyse and check the combinatorics of the constructed polyhedron. The advantage of the method used in \cite{DPP2} 
is that it produces a polyhedron which is bounded by a finite number of pieces of bisectors. This is in contrast with Dirichlet's method where 
the number of faces could be infinite. Moreover, the fact that only bisectors are involved as 3-faces makes the  
description of their intersection simpler than if one uses ``exotic'' faces.
\item Using Poincar\'e's Polyhedron theorem, the above authors are able to provide a presentation by generators and relations for 
each of the lattices studied. They also compute the orbifold Euler characteristic of the corresponding quotients of $\HdC$.
\end{enumerate}
\end{remark}


\section{Around an example : representations of the modular group.\label{section-example-modular}}
In this section, we are going to illustrate the ideas we have exposed on an example: representations of the modular group
in PU(2,1). We are going to describe the irreducible, discrete and faithful representations of the modular 
group $\Gamma$=PSL(2,$\mathbb{Z}$) in PU(2,1).  These representations were studied around 2000 in the series of 
articles \cite{FK2,FJP,GuP,GuP2}. The modular group being a 2 generator group, we will begin by describing these representations
in terms of traces, using the results of section \ref{sectionclasspairs}. The modular group is generated by an involution and an 
order three elliptic element in $\PSL$ with parabolic product. We will use the following presentation for $\Gamma$ 
\begin{equation}\label{modulpresent}
\Gamma =\la {\tt e},{\tt p}\,\vert\,{\tt e}^{2}=({\tt ep})^3=1\ra.
\end{equation}
We denote ${\tt c}={\tt ep}$. We are going to describe representations of $\Gamma$ to PU(2,1) such that $\rho({\tt p})$ is 
parabolic.  To do so, we need first to specify the conjugacy classes we choose for $\rho({\tt e})$ and  
$\rho({\tt c})$. Indeed, an order two elliptic in PU(2,1) can be either a complex reflection about a point or about a complex line. 
Similarly, an elliptic of order three can be either regular elliptic, a complex reflection about a point or a complex reflection 
about a line. This leaves a priori six possibilities. However, it is an easy exercise to check that when $\rho({\tt c})$ is a complex 
reflection of either type, the representation is always reducible, and rigid (see \cite{FJP}). We will thus only consider the case 
when $\rho({\tt c})$ is regular elliptic. Reducible representations appear under this assumption too, but they are flexible. 
\subsection{Traces.\label{subsection-traces-modular}}
Here is the family of representation we are interested in.
\begin{definition}
Let $\mathcal{R}_\Gamma$ be the set of conjugacy classes of
irreducible representations of the modular group in PU(2,1) that map {\tt p} to a parabolic and {\tt c} to a regular elliptic.
\begin{equation}\mathcal{R}_\Gamma=\left\{\rho :\Gamma\longrightarrow {\rm PU(2,1)},\rho({\tt c})\, \mbox{is regular elliptic},\,
\rho({\tt p}) \mbox{ is parabolic}\right\}/{\mbox{PU(2,1)}}.
\end{equation}
\end{definition}
For any such representation $\rho$, we denote by $E$, $P$ and $C$ the
images  $\rho({\tt e})$, $\rho({\tt p})$ and $\rho({\tt c})$. Any such pair $(E,P)$ of isometries can be lifted to 
SU(2,1) into a pair $({\bf E},{\bf P})$ such that ${\bf E}^2=I_3$ and $({\bf EP})^3=I_3$. The starting point is the following lemma.
\begin{lemma}\label{order23}
\begin{enumerate}
\item Any element of order two in SU(2,1) has trace $-1$.
\item Any element of SU(2,1) of order 3 with pairwise distinct
  eigenvalues has trace equal to 0.
\end{enumerate}
\end{lemma}
\begin{proof}
\begin{enumerate}
\item Let ${\bf A}$ be an element of SU(2,1). The Cayley-Hamilton relation for $\bA$ is
 \begin{equation}\label{CH}
\bA^3-\tr \bA \cdot \bA^2+\overline{\tr \bA}\cdot \bA-I_3=0
\end{equation}
When $\bA^2=I_3$, we have $\bA^{-1}=\bA$ and thus $\tr
\bA=\overline{\tr \bA^{-1}}$ is real. Plugging these facts in \eqref{CH},
we see that $\tr \bA$ satisfies $\left( \tr \bA+1\right)\left(\tr \bA
-3\right)=0$. An element of SU(2,1) with trace three is either the
identity or a parabolic element, which cannot be an
involution. Therefore $\tr \bA=-1$.
\item A regular elliptic isometry or order $3$ is lifted to SU(2,1) as a matrix with
three pairwise distinct eigenvalues, all of which are cube roots of unity, and therefore their sum is zero.
\end{enumerate}
\end{proof}

\begin{remark}\label{sym}
 Note that there are two conjugacy classes of involutions in SU(2,1),
 namely complex reflections in a line or in a point.  In ball
 coordinates, these two classes are represented by either the mapping
 $(z_1,z_2)\longrightarrow(z_1,-z_2)$ or
 $(z_1,z_2)\longrightarrow(-z_1,-z_2)$. The two classes of involutions
 can both be lifted to SU(2,1) under the form
\begin{equation}\label{liftinvol}
\bz\longmapsto -\bz+2\dfrac{\la\bz,\bv\ra}{\la\bv,\bv\ra}\bv
\end{equation}
where $\bv$ is a non-null vector in $\C^3$. When $\la \bv,\bv\ra>0$,
then $\bv$ is polar to a complex line $L_{\bv}$ which is pointwise
fixed by the transformation given by \eqref{liftinvol}, so that it is a reflection in a line. 
When $\la\bv,\bv\ra<0$,  $\bv$ is a lift to $\C^3$ of the unique fixed point in $\HdC$ of \eqref{liftinvol},
which is a complex reflection in a point. 
\end{remark}
As a consequence, when constructing a representation of $\Gamma$ into
PU(2,1), one must chose to map ${\tt e}$ to a reflection either in a
point or in a line. The first case has been studied in
\cite{FK2,GuP,GuP2} and the second in \cite{FJP}. We postpone the separation 
between these two cases to the next section.

 Let us denote by $u$ the (unit modulus) eigenvalue of ${\bf P}$
 associated with its fixed point on $\partial\HdC$. Multiplying by a
 central element of SU(2,1) if necessary, we may assume that
 $u=e^{i\alpha}$, with $\alpha\in(-\pi/3,\pi/3]$. As a consequence of Lemma \ref{order23}, we see that
$$\tr {\bf E}=-1,\,\tr {\bf P}=2u+u^{-2},\mbox{ and }\tr {\bf EP}=\tr {\bf E}^{-1}{\bf P}=0.$$
We can now plug these values in the polynomials $S$ and $P$ given by
\eqref{sumroots} and \eqref{prodroots}. Namely, we do the replacements

$$ x_1=x_5=-1,\, x_2=\overline{x_6}=2u+u^{-2},x_3=x_4=x_7=x_8=0.$$
We see thus that the coefficients of equation \eqref{trace equation} are given by
\begin{equation}
s  =  4u^{-3}(u^3+1)^2\mbox{ and }p  = s^2/4= 4u^{-6}(u^3+1)^4.
\end{equation}
As a direct consequence, \eqref{trace equation} has a real double root. After solving \eqref{trace
  equation}, this root is seen to be equal to $2\left(2+u^3+u^{-3}\right)$. We can thus write
$$\Phi({\bf E},{\bf P})=\left(-1,2u+u^{-2},0,0,2\left(
2+u^3+u^{-3}\right)\right),$$
were $\Phi$ is as in \eqref{mappu21}.
\begin{remark}
\begin{enumerate}
\item We have made the choice that $\alpha$, the argument of $u$, belong to $(-\pi/3,\pi/3]$. The two representations obtained for 
$\alpha=-\pi/3$  and $\alpha=\pi/3$ are not conjugate in SU(2,1), but correspond to the same group in PU(2,1). We can therefore 
identify the two endpoints of the interval and consider it a circle.
\item
The fact the commutator of ${\bf E}$ and ${\bf P}$ has real
trace is not a surprise. Indeed, $[{\bf E},{\bf P}]^{-1}$ is conjugate
to $[{\bf E},{\bf P}]$ by ${\bf E} $ because $\bE=\bE^{-1}$. Therefore $[{\bf E},{\bf
    P}]$ and its inverse have the same trace. But $\tr({\bf A}^{-1})=\overline{\tr {\bf A}}$ 
for any ${\bf A}\in$ SU(2,1).
\end{enumerate}
\end{remark}

\subsection{Cartan invariant and the parabolic eigenvalue.\label{subsection-Cartan-eigenvalue-modular}}
We are now going to illustrate the connection between projective invariant and eigenvalues that we have already exposed in section 
\ref{section-cross-ratio-eigenvalue}. To any representation $\rho$ in $\mathcal{R}_\Gamma$ is associated an ideal 
triangle $\Delta_\rho=(p_1,p_2,p_3)$ of which vertices are the parabolic fixed points given by
\begin{equation}
 p_1=\mbox{fix}(P),\, p_2=\mbox{fix}(CPC^{-1}) \mbox{ and } p_3=\mbox{fix}(C^{-1}PC).
\end{equation}
Let us fix a lift $\bp_1$ of $p_1$ and set $\bC \bp_1=\bp_2$ and $\bC^{-1} \bp_1=\bp_3$. The triple product of these three vectors
can be computed as follows.
\begin{align}\label{tripleproduct}
\la \bp_1,\bp_2\ra\la\bp_2,\bp_3\ra\la\bp_3,\bp_1\ra & = 
\la\bp_1,{\bf C}\bp_1\ra\la {\bf C}\bp_1,{\bf C}^2\bp_2\ra\la {\bf
  C}^2\bp_2,{\bf C}^3\bp_1\ra\nonumber\\ & = \la\bp_1, {\bf
  C}\bp_1\ra^3= \la\bp_1,{\bf EP}\bp_1\ra^3\nonumber\\ & =  \bar u
^3\la \bp_1,{\bf E}\bp_1\ra^3.
\end{align}
Now there exists a vector $\bv$ such that ${\bf E}$ is given by
\eqref{liftinvol}. Using the fact that $\bp_1$is a null vector, we obtain then
\begin{equation}\label{bracket}
\la \bp_1,{\bf E}\bp_1\ra = 2\dfrac{|\la \bp_1,\bv\ra|^2}{\la\bv,\bv\ra}.
\end{equation}
In turn, we see that the triple product is given by 
\begin{equation}\label{tripleproduct2}
\la \bp_1,\bp_2\ra\la\bp_2,\bp_3\ra\la\bp_3,\bp_1\ra  =  8\bar u^3\dfrac{|\la\bp_1,\bv\ra|^6}{\la \bv,\bv\ra^3}.
\end{equation}
As a consequence of this computation we can express the Cartan invariant of $\Delta_\rho$ in terms of the eigenvalue of 
$\rho({\tt p})$.
\begin{lemma}\label{lem-Cartan-reflection}
 Let $\rho$ be a representation of the modular group in PU(2,1) such
that the eigenvalue of $\rho({\tt p})$ associated with its fixed point
is $u$, with $|u|=1$. 
\begin{enumerate}
\item If $\rho({\tt e})$ is a complex reflection about a point, then
$\A(\Delta_{\rho})=\arg(\overline{u}^3)$.
\item If $\rho({\tt e})$ is a complex reflection about a line, then
$\A(\Delta_{\rho})=\arg(-\overline{u}^3)$.
\end{enumerate}
\end{lemma}
\begin{proof}
It is a direct consequence of the definition of the Cartan invariant and of the fact that $\la\bv,\bv\ra$ is positive if 
and only if $\rho({\tt e})$ is a complex reflection about a line, and negative if and only if it is a reflection about a point. 
\end{proof}
Lemma \ref{lem-Cartan-reflection} should be compared to Theorem 7.1.3 of \cite{Go} which states a similar connection between the 
Cartan invariant and the eigenvalue of a product of three antiholomorphic isometric. This leads to the following
\begin{proposition}\label{prop-Cartan-eigenvalue}
Let $\rho$ be a representation of the modular group in PU(2,1) such
that the eigenvalue of $\rho({\tt p})$ associated with its fixed point
is $u=e^{i\alpha}$, with $\alpha\in(-\pi/3,\pi/3]$. Then
\begin{enumerate}
 \item If $|\alpha|<\pi/6$ then $\rho({\tt e})$  is a complex reflection about a point.
 \item If $|\alpha|>\pi/6$ then $\rho({\tt e})$ is a complex reflection about a line.
 \item If $|\alpha|=\pi/6$, both cases happen.
\end{enumerate}
\end{proposition}
\begin{proof}
In view of Lemma \ref{lem-Cartan-reflection}, we see that $\A(\Delta_\rho)=-3\alpha$ when $\rho({\tt e})$ is a reflection 
about a point, and $\A(\Delta_\rho)=\pi-3\alpha$ when it is a reflection about a line. Taking in account the fact that the 
Cartan invariant belongs to $[-\pi/2,\pi/2]$ mod. $2\pi$ gives the result.
\end{proof}
\begin{remark}\label{nonconjugate-reducible-pairs-same-phi}
 It is interesting to note that the two pairs $(\bE,\bP)$ corresponding to the value $\alpha=\pi/6$ have the same image under 
the trace map $\Phi$ (see \eqref{mappu21}), though they are not conjugate. This is due to the fact that these pairs are 
reducibles. 
\end{remark}
\subsection{Geometric description of the representations\label{subsection-geometry-modular}}
\subsubsection{Construction of the representations from an ideal triangle}
\begin{figure}
\begin{center}
\scalebox{0.5}{\includegraphics{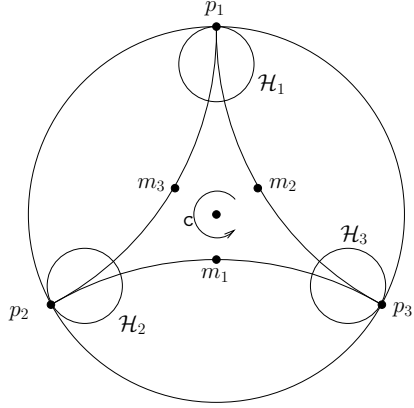}}
\end{center}
\caption{Fixed points and horosphere for the modular group in the Poincar\'e disc.\label{fix-horo}}
\end{figure}
As we have seen in the previous section,a representation of the modular group is determined by the ideal triangle formed by the 
parabolic fixed points of $\rho({\tt p})$ and its conjugate under $\rho({\tt c})$. This ideal triangle allows us to give a geometric 
description of the representations which is similar to the classical one in the Poincar\'e disc. For this, we will need the following 
elementary geometric facts.
\begin{enumerate}
\item[] \textbf{Fact 1}:  To any ideal triangle $(p_1,p_2,p_3)$ not contained in a complex line is associated a unique 
order 3 regular elliptic map mapping $p_i$ to $p_{i+1}$ (indices taken mod 3), see for instance Corollary 7.1.8. of \cite{Go}.
\item[] \textbf{Fact 2}: For any pair of distinct points $p$ and $q$ in $\partial\HdC$, and any pair of horospheres $\mathcal{H}_p$ 
and $\mathcal{H}_q$ based at $p$ and $q$ respectively, there exists a unique complex reflection about a point and a unique complex 
reflection about a line exchanging $\mathcal{H}_p$ and $\mathcal{H}_q$.
\end{enumerate}
Now, let us consider an ideal triangle $(p_1,p_2,p_3)$, not contained in a complex line. Denote by $C$ the unique order three regular 
elliptic isometry such that $C(p_1)=p_2$ and $C(p_2)=p_3$, and fix a horosphere $\mathcal{H}_1$ based at $p_1$. Denote by 
$\mathcal{H}_2$ and $\mathcal{H}_3$  and  the respective images of $\mathcal{H}_1)$ by $C$ and $C^{-1}$. With this notation, 
$\mathcal{H}_i$ is based at $p_i$. Applying Fact 2 to the three pairs of horospheres $(\mathcal{H}_i,\mathcal{H}_{i+1})$, we obtain 
a triple of complex reflections about points $(E_1,E_2,E_3)$ and a triple of complex reflections about lines $(E'_1,E'_2,E'_3)$ such 
that $E_i$ and $E'_i$ both exchange $\mathcal{H}_{i+1}$ and $\mathcal{H}_{i+2}$. These two triples of involutions do not depend on 
the choice of the first horosphere, and are determined by the fact that the configaration of horospheres has a symmetry of order 3.

It is easy to check that the two isometries $P=E_3C$ and $P'=E'_3C$ have the following properties.
\begin{enumerate}
 \item $P$ and $P'$ both fix $p_1$.
 \item $P$ and $P'$ both preserve the horosphere $\mathcal{H}_1$
\end{enumerate}
The second condition implies that $P$ and $P'$ are either parabolic or a complex reflection about a line with mirror 
containing $p_1$. This second possibility can be ruled out using the same kind of arguments as in \cite{ParkWi}.
As a consequence, we see that the two pairs $(E_3,P)$ and $(E'_3,P')$ provide two representations of the modular group to PU(2,1),
by setting $\rho_1({\tt e})=E_3$ and $\rho_1({\tt p})=P$ or $\rho_2({\tt e})=E'_3$ and $\rho_2({\tt p})=P'$.
In view of the previous two sections, the representations obtained in this way are the only ones. The choice of the isometry type 
of the involution and of the Cartan invariant of the triangle $(p_1,p_2,p_3)$ determines the conjugacy class of the parabolic element 
$\rho({\tt p})$. 
\subsubsection{$\R$-decomposability}
The trace computations done in section \ref{subsection-traces-modular} have show that the commutator of $E$ and $P$ satisfies
$$\tr[E,P]=2\left(2+u^3+\bar u^3\right).$$
When $\alpha=\pi/6$ and thus $u^3=i$, we see that $[E,P]$ is loxodromic with trace equal to 4. But the spectrum of a loxodromic element 
with real trace is $\lbrace r,1,1/r\rbrace$ or $\{-r,1,-1/r\}$ for some positive real number $r$, where the eigenvalues of non-unit 
modulus correspond to boundary fixed points. This implies that a trace 4 loxodromic element in SU(2,1) must have a real and 
positive eigenvalue associated with its boundary fixed points. This property is conserved by this specific deformation is 
deformed as $[E,P]$ cannot have 0 as an eigenvalue.

Using the main result of \cite{PaW} (see section \ref{subsection-double-root} above), this implies that for any value 
of $u$, the pair $(E,P)$ is $\R$-decomposable: there exists $(\sigma_1,\sigma_2,\sigma_3)$ a triple of real reflections
such that $E=\sigma_1\sigma_2$ and $P=\sigma_2\sigma_3$. This implies that $C=\sigma_1\sigma_3$. A simple way of identifying these 
real reflections is given by the fact that to any ideal triangle $(m_1,m_2,m_3)$ is associated a triple of real symmetries
$(s_1,s_2,s_3)$ such that $s_i$ fixes $mi_i$ and exchanges $m_{i+1}$ and $m_{i+2}$. The product $s_is_{i+1}$ is then elliptic of order three 
(see chapter 7 of \cite{Go}). These gives us $\sigma_1$ and $\sigma_3$. The two other conditions determine 
$\sigma_2$ uniquely. The existence of this decomposition is an important tool in \cite{FJP}.
\subsection{Finding explicit matrices}
We now provide explicit matrices for these representations of the modular group. We identify here isometries 
and their lifts to SU(2,1).
\begin{proposition}
Any irreducible representation $\rho:$PSL(2,$\mathbb{Z}$)$\longrightarrow$PU(2,1) is conjugate to one
given by the following two cases.
\begin{itemize}
\item $E$ is a reflection in a point.
\begin{equation}
 \rho(\tt e)=\begin{bmatrix}0 & 0 & -1\\0 & -1 & 0\\-1 & 0 & 0\end{bmatrix}
\mbox{ and }
\rho(\tt p)=\begin{bmatrix}e^{i\alpha} & \sqrt{2\cos{3\alpha}} & -e^{-2i\alpha}\\
             0 &e^{-2i\alpha} &-\sqrt{2\cos{3\alpha}}e^{-i\alpha} \\
             0 & 0 & e^{i\alpha} 
            \end{bmatrix}, 
\end{equation}
for $3\alpha\in(-\pi/2,\pi/2)$ mod. $2\pi$. 
\item $E$ is a reflection in a line.
\begin{equation}
 \rho(\tt e)=\begin{bmatrix}0 & 0 & 1\\0 & -1 & 0\\1 & 0 & 0\end{bmatrix}
\mbox{ and }
\rho(\tt p)=\begin{bmatrix}e^{i\alpha} & \sqrt{-2\cos{3\alpha}} & e^{-2i\alpha}\\
             0 &e^{-2i\alpha} &-\sqrt{-2\cos{3\alpha}}e^{-i\alpha} \\
             0 & 0 & e^{i\alpha} 
            \end{bmatrix},
\end{equation}
for $3\alpha\in(\pi/2,3\pi/2)$ mod. $2\pi$.
\end{itemize}
\end{proposition}
To obtain these matrices, we have made the choice to fix that the triple of points $(p_1,p_2,p_3)$ described in the 
previous section is given by the lifts
\begin{equation}\label{normal}
 \bp_1 =\begin{bmatrix}1 \\ 0 \\0\end{bmatrix},\quad \bp_2 =\begin{bmatrix}0 \\ 0 \\ 1\end{bmatrix},\mbox{ and }
\bp_3=\begin{bmatrix}\pm e^{-3i\alpha}\\-\sqrt{\pm 2\cos{3\alpha}}e^{-2i\alpha}\\1 \end{bmatrix}.\end{equation}
Here the choice of sign in the vector $\bp_3$ depends on the choice made for the conjugacy class of the involution $\rho({\tt e})$.
When $\alpha=0$, the three points $p_1$, $p_2$ and $p_3$ all are on $x$ axis of the Heisenberg group, which is the 
boundary of a real plane. When $3\alpha=\pi/2\mod\pi$, they are on the $t$ axis, which is the boundary of a complex line.

\subsection{Description of the moduli space and discreteness results\label{section-mod-discrete}}
It follows from the previous sections that the moduli space $\mathcal{R}_\Gamma$ has two connected components 
corresponding to the two possible conjugacy classes for $\rho({\tt e })$. Both connected components are arcs 
and we denote by $\mathcal{A}_p$ the component where $\rho({\tt e })$ is a reflection about a point and by $\mathcal{A}_{l}$
the component where it is a reflection about a line.
\begin{figure}
 \begin{center}
  \scalebox{0.5}{\includegraphics{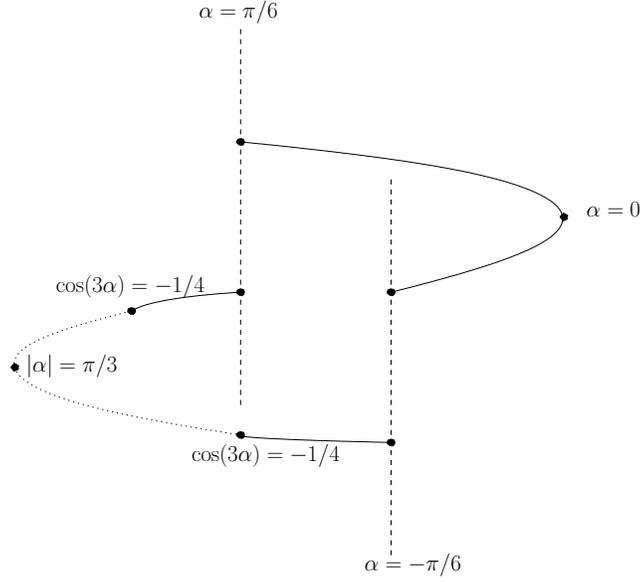}}
 \end{center}
\caption{A schematic representation of $\mathcal{R}_\Gamma$. The dotted part of the curve represents representations of the 
modular group that are either non-discrete or non-faithful.\label{picRgamma}}
\end{figure}
\begin{enumerate}
\item When $|\alpha|=\pi/6$ the representations coming from $\mathcal{A}_p$ and $\mathcal{A}_{l}$ preserve a complex line, and are 
discrete. Such groups are often $\C$-Fuchsian.
\item When $\alpha=0$ the representation coming from $\mathcal{A}_p$ is discrete and faithful, and preserve a real plane : it is 
$\R$-Fuchsian. 
\item When $\alpha=\pm\pi/3$ the representation coming from $\mathcal{A}_l$ is discrete but not faithful. It is easy to see that 
in this case the three products $E_iE_{i+1}$ are elliptic and fix the barycenter of the triangle $\Delta_\rho$.
\end{enumerate}

\begin{theorem}
\begin{enumerate}
 \item (Falbel-Koseleff \cite{FK2}, Parker-Gusevskii \cite{GuP,GuP2}) Any representation in the familly $\mathcal{A}_p$ is discrete 
 and faithful.
 \item (Falbel-Parker \cite{FJP}) A representation in $\mathcal{A}_l$ is discrete and faithful if and only if 
$\cos(3\alpha)\in[-1/4,0]$ (or equivalently $\cos(\A(\Delta_\rho))\in[0,1/4]$).
\end{enumerate}
 \end{theorem}
Going back to the trace description of the representations we can rewrite the trace of the commutator $[{\bf E},{\bf P}]$
terms of the Cartan invariant of $\Delta_\rho$.
\begin{equation}\label{tracecomEP}
 \tr \left([{\bf E},{\bf P}]\right) = \left\{\begin{array}{rl}
4 + 4\cos\A & \mbox{when $E$ is a reflection in a point}\\
4 - 4\cos\A & \mbox{when $E$ is a reflection in a line}
\end{array}\right.
\end{equation}
The discreteness and faithfulness of a representation $\rho$ in $\mathcal{R}_\Gamma$ is thus controlled by the isometry 
type of the commutator $[E,P]$. Indeed, \eqref{tracecomEP} shows that this element  is always loxodromic when $E$ is a reflection in 
a point, and that is is non elliptic if and only if $\cos\A>1/4$ when $E$ is a reflection in a line.  We can therefore state:
\begin{theorem}\label{discretmodular}
Let $\rho$ be an irreducible representation of the modular group given by $E$ and $P$ as above. Then $\rho$ is 
discrete and faithful if and only if the commutator $[E,P]$ is non-elliptic.
\end{theorem}
As mentioned above, when $\alpha=\pm\pi/3$, the representation is discrete but non-faithful. Other such examples exist. See for 
instance Proposition 5.10 in \cite{FJP2}, and the discussion before it. There, Falbel and Parker have provided 
generators and relations for the Eisenstein-Picard lattice PU(2,1,$\mathcal{O}_3$).  On their way, they prove that it can be seen as 
$\la G,T\ra$, where $G$ is a discrete but non-faithful representation of the modular group and $T$ an element of order 6. See also the 
discussion in \cite{FWa}. This representation belongs to the family $\mathcal{A}_l$, and correspond to the value $\alpha=\pm 2\pi/9$.

The most natural way of proving the discreteness of these representations is to consider the subgroup generated by the 
three involutions $E_1=E$, $E_2=CEC^{-1}$ and $E_3=C^{-1}EC$, and use Klein's combination theorem. 
This is what is done for instance in \cite{FJP}. To do so, the authors construct three hypersurfaces 
$S_1$, $S_2$ and $S_3$ satisfying the following properties.
They construct three hypersurfaces $S_1$, $S_2$ and $S_3$ satisfying the following properties.
\begin{enumerate}
\item For $i=1,2,3$, $\HdC\setminus S_i$ has two connected components.
\item The $S_i$'s are disjoint in $\HdC$, and for all $i$ the $S_{i+1}$ and $S_{i+2}$ are 
contained in the same connected component of $\HdC\setminus S_i$.
\item For $i=1,2,3$, $E_i$ preserves $S_i$ and exchanges the two connected components of $\HdC\setminus S_i$.
\item The $S_i$'s are cyclically permuted by $E$ : $S_{i+1}=E S_i$ (indices taken mod. 3).
\end{enumerate}
These conditions allow them to apply Klein's combination theorem, and conclude that the group is discrete.
These hypersurfaces are foliated by totally geodesic subspaces, $\C$-surfaces or $\R$-surfaces, as described in 
section \ref{fundadoms}.  One of the reasons that make these constructions quite tricky and technical is that the 
hypersurfaces $S_i$ and $S_{i+1}$ have to be tangent at the parabolic fixed point $P_{i+2}=C^{i+2}PC^{i+3}$ (indices taken mod. 3), 
and this leads to heavy computations (see section 4 of \cite{FJP}). In addition, for the last discrete and faithful representation, the 
(parabolic) fixed point of $[E,P]$ is another tangency point. The $\C$-surfaces used  in \cite{FJP} are piecewise bisectors, and the 
$\R$-surfaces are similar to those used by Schwartz in \cite{S3,Schbook}. 
\subsection{From the modular group to cusped surfaces group : the Gusevskii-Parker examples.}
It is a well-known fact that for any data $(g,p)$ such that $2-2g-p<0$ and $p>0$, there exists a subgroup of PSL(2,$\mathbb{Z}$) 
which is isomorphic to the fundamental group of the oriented surface of genus $g$ with $p$ punctures (see for instance \cite{Mil}).
Therefore, by passing to finite index subgroups, any representation $\rho$ of the modular group to PU(2,1) provides examples of 
representations of cusped surfaces of any topological type.  In \cite{GuP2}, Gusevskii and Parker have exploited this fact to 
produce examples of discrete and faithful representations of cusped hyperbolic surfaces. They consider the images by 
representations belonging to $\mathcal{A}_p$ of finite index subgroup of the modular group. They obtain this ways 1-parameter 
families of representations of $\pi_1(\Sigma_{g,p})$ into PU(2,1) having the following features.
\begin{enumerate}
 \item Each family contains only discrete and faithul representations mapping peripheral loops to parabolics. 
 \item Each family connects an $\R$-Fuchsian representation to a $\C$-Fuchsian one, and takes all values of the 
Toledo invariant allowed by the Milnor-Wood inequality (see \cite{GuP2})
\end{enumerate}
Of course, one can play a similar game with the representations in $\mathcal{A}_l$, described in \cite{FJP}. This time one produces
one parameters families of representations of $\pi_1(\Sigma_{g,p})$ that start from a $\C$-Fuchsian representation, and stop being 
discreteand faithful at a certain point, which correspond to the value of $\alpha$ for which $[E,P]$ becomes parabolic. Let us make 
explicit these parabolic elements for some simple surfaces.

\paragraph{The once punctured torus.} The subgroup $\Gamma_1$ of the modular group generated by the two elements 
${\tt a_1}={\tt cec}^{-1}{\tt e}$ and ${\tt b_1}={\tt cecec}$ uniformises a 1-punctured torus, and has index 6 in 
PSL(2,$\mathbb{Z}$). This can be easily checked using the matrix representatives of ${\tt e}$, ${\tt c}$ and ${\tt p}$, given by
\begin{equation}\label{matrixpsl2z}
 {\tt e}=\begin{bmatrix}0 & -1\\1 & 0\end{bmatrix},\,{\tt p}=\begin{bmatrix}1 & 1\\0 & 1\end{bmatrix}\mbox{ and }
{\tt c}=\begin{bmatrix}0 & -1\\1 & 1\end{bmatrix}.
\end{equation}
As $[{\tt e},{\tt p}]={\tt epep}^{-1}={\tt cec}^{-1}{\tt e}$, we see that the point where $[E,P]$ is parabolic corresponds to 
$A_1=\rho(\tt a_1)$ becoming parabolic. In other words, the family of representations of $\Gamma_1$ we have obtained stops being 
discrete when one pinches the simple closed curve corresponding to ${\tt a_1}$ on the 1-punctured torus.

\paragraph{The 3-punctured sphere.} Similarly, the subgroup $\Gamma_2$ of the modular group generated by the two elements 
${\tt a}={\tt cece}$ and ${\tt b}={\tt c}^{-1}{\tt ece}{\tt c}^{-1}$ uniformises a 3-punctured sphere, and has index 6 in 
PSL(2,$\mathbb{Z}$). By a direct verification, we see that
$$[\tt e,\tt p]^3=({\tt cec}^{-1}{\tt e})^3=[{\tt b}^{-1},{\tt a}^{-1}].$$
This time the corresponding family of representations of $\Gamma_2$ stops being discrete when the commutator $[A,B]$ 
(which is conjugate to $[B^{-1},A^{-1}]^{-1}$) becomes parabolic.

\subsection{Deformations transverse to the Gusevskii-Parker familly}
The Farey set $\mathcal{F}\subset S^1$ can be see as the set of fixed points of parabolic elements in PSL(2,$\Z$). In the 
upper-half model of $\HuC$, it is nothing but the one point compactification of $\Q$. The Farey tesselation of the Poincar\'e 
disc is obtained from $\mathcal{F}$ by connecting by a geodesic those rationnals $p/q$ and $p'/q'$ such that $|pq'-p'q|=1$. 
Clearly, the Farey tesselation is invariant under the action of the modular group. Moreover two parabolic maps in PSL(2,$\Z$) 
have the same fixed point if and only if they are in a common cyclic group. This means that
given a discrete and faithful representation $\rho$ of PSL(2,$\Z$) to PU(2,1) that maps parabolics to parabolics, 
one can construct a $\rho$-equivariant map 
\begin{align}
\phi_\rho & : \mathcal{F}\longrightarrow\partial\HdC\nonumber \\
 &  m={\rm fix}(g)\longmapsto {\rm fix}(\rho(g)),
\end{align}
where $g$ here is a parabolic fixing $m$. In fact if one think of $\mathcal{F}$ as acted on by the fundamental group $\pi_1$ of a cusped 
surface  $\Sigma$ (seen for instance as a subgroup of PSL(2,$\Z$)), it is essentially equivalent to construct a ($\pi_1$-equivariant) map 
$\phi:\mathcal{F}\longrightarrow\partial\HdC$ and a representation of $\pi_1$ to PU(2,1). 

In the case of \cite{FJP}, for a given choice of $\alpha$, or equivalently for a given choice of parabolic eigenvalue 
for $\rho({\tt p})$ all triangles of the Farey tesselation are mapped to ideal triangles with Cartan invariant equal to $\A=\pi-3\alpha$. 
In \cite{Wi6}, this point of view is adopted to construct representations for which all triangles are contained in a real plane 
(that is they have Cartan invariant equal to zero). The method can be summed up as follows.
\begin{enumerate}
 \item Embed isometrically the Farey tesselation into a real plane $\HdR\subset\HdC$. The representation one obtains this way is $\R$-Fuchsian : it 
is discrete, faithful and maps parabolics to parabolics. 
 \item Shear and bend along the edges in a $\pi_1$-invariant way. One obtain this way a $\pi_1$ invariant familly of real ideal 
triangles. The corresponding representation  no longer preserve $\HdR$ provided that the bending angle are not zero. It has a 
priori no reason to be discrete, even though it is expected that for small values if the shear-bend parameters it should be so.
\end{enumerate}
The $\pi_1$ invariant shear and bend deformation can be encoded via a decoration of an ideal triangulation of the surface 
$\Sigma$ considered by cross-ratio like invariants (we refer to \cite{Wi6} for details), in the spirit  Penner coordinates on the 
decorated Teichm\"uller space (see \cite{Penner}). The main result of \cite{Wi6} is concerned with those representations obtained by 
from an $\R$-Fuchsian one by bending of the  same angle $\alpha$ along each edge of the ideal triangulation. This kind of bending 
is called \textit{regular} in \cite{Wi6}. A simplified version of the main result is as follows. 
\begin{theorem}\label{theo-bend}
For any $\alpha\in[-\pi/2,\pi/2]$, and any shearing data, the representation obtained by a regular bending of angle $\alpha$ is 
discrete, faithful, and maps peripheral curves to parabolics.
\end{theorem}
The proof of Theorem \ref{theo-bend} is as follows. For any two ideal triangles $\Delta$ and $\Delta'$ sharing a common edge $\gamma$, 
it is possible to construct a canonical hypersurface  $S(\Delta,\Delta')$  having the property that $\Delta$ and $\Delta'$ are 
contained in opposite connected components of $\HdC\setminus S(\Delta,\Delta')$. This hypersurface is in fact a spinal $\R$-surface, 
or flat pack (see  the discussion following Proposition \ref{dualcurve}). The 
condition that $\alpha\in[-\pi/2,\pi/2]$ guarantees that all the surfaces $S(\Delta,\Delta')$ are disjoint when $(\Delta,\Delta')$ 
run over all pairs of neighbouring triangles of the trianguations. This provides a fundamental domain for the action of the image of 
the corresponding representation. Moreover, these representations preserve a disc which is piecewise a real plane : it is obtained as
 a union of real ideal triangles. A direct consequence of Theorem \ref{theo-bend} is the following corollary.
\begin{corollary}\label{embed}
For each $\alpha\in[-\pi/2,\pi/2]$, the regular bending of angle $\alpha$ induces an embedding of the Teichm\"uller space of 
$\Sigma$ into the PU(2,1)-representation variety of $\Sigma$, of which image contains only discrete, faithful and type-preserving 
representations.
\end{corollary}
Corollary \ref{embed} is just a reformulation of Theorem \ref{theo-bend} using shear coordinates on the Teichm\"uller space of 
$\Sigma$.
\begin{remark}
\begin{enumerate}
\item There is a slight inacuracy in the above two statements. If one starts from an arbitrary ideal triangulation of $\Sigma$, then one 
obtains a representation of $\pi$ in Isom($\HdC$) : some elements of $\pi$ can be mapped to antiholomorphic isometries. If one wants 
to obtain a representation in PU(2,1), one needs to start from a bipartite triangulation. This fact must be taken into account in the 
definition of regular bending. We refer the reader to \cite{Wi6} for details.
\item In \cite{Wi2}, the same result is proved in the special case where $\Sigma$ is a 1-punctured torus, from the point of view of 
groups generated by three real symmetries.
\item The fact that all representations obtained in Theorem \ref{theo-bend} preserve a piecewiese totally real disc implies that their
Toledo invariants are equal to zero. In contrast, the Gusevskii-Parker family of representations of $\pi$ take all possible values of 
the Toledo invariant. It can be proved that the only intersection between these two family is the $\R$-Fuchsian class in the 
Gusevskii-Parker family.
\end{enumerate}
\end{remark}
\subsection{A spherical CR structure on the Whitehead link complement\label{SCRWLC}}
In this section, we come back to representations of the modular group, and we fix once and for all a value of $\alpha$ so that 
$\cos(3\alpha)=-1/4$. This leaves two choices, but the corresponding representations are conjugate by an antiholomorphic isometry, 
so that the precise choice is of no importance for us. We are thus considering what we could call the 
\textit{last complex hyperbolic modular group}. We will from now on denote by $\Gamma$ the image of the modular group by the 
representation rather than the group PSL(2,$\mathbb{Z}$).

Denote by $\Gamma_0$ the subgroup of $\Gamma$ generated by $E_1=E$, $E_2=CEC^{-1}$ and $E_3=C^{-1}EC$. 
It has index three in $\Gamma$ and the product $E_1E_2=ECEC^{-1}$ is unipotent parabolic. 
We obtain therefore a group generated by three complex reflections of order 2, with parabolic pairwise products. Moreover, 
the triple product $E_1E_2E_3$ is equal to $P^{3}$ and is thus parabolic as well. This implies that $\la E_1,E_2,E_3\ra$ is 
in fact a copy of the \textit{last ideal triangle group} (see \cite{FJP,GP,S1,Schbook} and section 
\ref{section-Schwartz-triangle-groups}). Let $C'$ be the order three elliptic isometry cyclically permuting the fixed points 
of the three parabolic maps $E_iE_{i+1}$ ($C'$ is different from $C$!). The isometry $C'$ also cyclically conjugates the 
parabolic maps $E_iE_{i+1}$. Let $\Gamma_3$ be the group generated by $\Gamma_0$ and $C'$. In \cite{S1}, Schwartz has proved 
the following.
\begin{theorem}\label{theo-Schwartz-WLC}
The group $\Gamma_3$ is discrete, and its manifold at infinity is homeomorphic to the complement of the Whitehead link.
\end{theorem}
This is an example of what is called a \textit{spherical CR structure} on a 3-manifold. In general such a structure is 
an atlas such that the transition maps are restrictions of elements of PU(2,1). In other words, it is an $(X,G)$-structure 
where $X=S^3$ and $G=$PU(2,1). The very special feature here is that the Whitehead link complement is a hyperbolic 3 manifold. 
In particular, the quotient of $\HdC$ by $\Gamma_3$ is a complex hyperbolic orbifold of which boundary is a real hyperbolic 
manifold.

Here, $\Gamma_3$ contains elliptic elements and therefore it is not a faithful representation of the fundamental group of the 
Whitehead link complement, which is torsion free. This does not contradict the fact that the quotient of the discontinuity region of 
$\Gamma_3$ is a manifold. Indeed the only elliptic elements in $\Gamma_3$  are regular elliptic isometries (of order 3). 
This means inparticular that they act on $S^3$ without fixed point. Moreover, the discontinuity region is the complementary of a 
curve and is not simply connected.

A very natural question is to decide which 3-manifolds admit such a spherical CR structure.
In \cite{KaTu}, Kamishima and Tsuboi have studied spherical CR structures on Seifert fiber spaces. In particular their
Theorem 3 shows that in a sense the main class of closed orientable 3-manifold with an $S^1$ invariant spherical CR structure are
circle bundles over Euclidean or hyperbolic 2-orbifolds.
Explicit  spherical CR structure on circle bundles over hyperbolic surfaces are relatively easy to produce by considering 
discrete and faithful representations of surface groups in PU(2,1). Many examples can be found in the litterature 
(among these \cite{AGG2,AGG,GKL,Gaye}).  In \cite{S4}, Schwartz has given an example of a spherical CR structure on a 
closed hyperbolic 3 manifold. Recently, in \cite{DF} Deraux and Falbel have described a spherical CR structure on the complement 
of the figure eight knot. In the article to come \cite{ParkWi2}, Parker and Will produce an example of a spherical CR structure on the 
complement of the Whitehead link that is not conjugate to Schwartz's one. The question of knowing which hyperbolic 3 manifolds 
admit a spherical CR structure is still wide open.


\frenchspacing\begin{thebibliography}{100}

\bibitem{AGG2}
S.~Anan'in, C.~Grossi, and N.~Gusevskii.
\newblock Complex hyperbolic structures on disc bundles over surfaces.
  $\mbox{II}$. $\mbox{E}$xample of a trivial bundle.
\newblock {\em arXiv:math/0512406}.

\bibitem{AGG}
S.~Anan'in, C.~Grossi, and N.~Gusevskii.
\newblock Complex hyperbolic structures on disc bundles over surfaces.
\newblock {\em Int. Math. Res. Not.}, (19):4295--4375, 2011.

\bibitem{BM}
A.~Basmajian and R.~Miner.
\newblock Discrete subgroups of complex hyperbolic motions.
\newblock {\em Invent. Math.}, 131:85--136, 1998.

\bibitem{Bear}
A.~Beardon.
\newblock {\em The geometry of discrete groups}.
\newblock Springer, New York, 1983.

\bibitem{Best}
M.~Bestvina.
\newblock $\mathbb{R}$-trees in topology, geometry, and group theory.
\newblock In {\em Handbook of geometric topology}, pages 55--91. Amsterdam,
  north-holland edition, 2002.

\bibitem{Bis}
I.~Biswas.
\newblock On the existence of unitary flat connections over the punctured
  sphere with given local monodromy around the punctures.
\newblock {\em Asian J. Math}, 3(2):333--344, 1999.

\bibitem{Bow2}
B.~Bowditch.
\newblock Markoff triples and quasifuchsian groups.
\newblock {\em Proc. London. Math. Soc.}, 1998.

\bibitem{BuIoWi}
M.~Burger, A.~Iozzi, and A.~Wienhard.
\newblock Surface group representations with maximal $\mbox{T}$oledo invariant.
\newblock {\em C. R. Acad. Sci Paris}, 336:95--104, 2003.

\bibitem{CG}
S.~Chen and L.~Greenberg.
\newblock Hyperbolic spaces.
\newblock In {\em Contribution to Analysis}, pages 49--87. Academic Press, New
  York, 1974.

\bibitem{CJ}
J.H. Conway and A.J.Jones.
\newblock Trigonometric diophantine equations (on vanishing sums of roots of
  unity).
\newblock {\em Acta Arithmetica}, 30:229--240, 1976.

\bibitem{Corl}
K.~Corlette.
\newblock Archimedian superrigidity and hyperbolic geometry.
\newblock {\em Annals of Maths}, 135:165--182, 1992.

\bibitem{DM}
P.~Deligne and G.D. Mostow.
\newblock {\em Commensurabilities among Lattices in $\mbox{PU(1,$n$)}$}.
\newblock Ann. of Math. Studies. Princeton University Press, Paris, 1993.

\bibitem{Der}
M.~Deraux.
\newblock $\mbox{D}$irichlet $\mbox{D}$omains for the $\mbox{M}$ostow
  $\mbox{L}$attices.
\newblock {\em Exper. Math.}, 14:467--490, 2005.

\bibitem{Der2}
M.~Deraux.
\newblock Deforming the $\mathbb{R}$-fuchsian $(4,4,4)$-triangle group into a
  lattice.
\newblock {\em Topology}, 45:989--1020, 2006.

\bibitem{DF}
M.~Deraux and E.~Falbel.
\newblock Complex hyperbolic geometry of the fugure eight knot.
\newblock Preprint, 2014, 
\newblock {\em arXiv:1303.7096v1}.

\bibitem{DFP}
M.~Deraux, E.~Falbel, and J.~Paupert.
\newblock New constructions of fundamental polyhedra in complex hyperbolic
  space.
\newblock {\em Acta Math.}, 194:155--201, 2005.

\bibitem{DPP}
M.~Deraux, J.~Parker, and J.~Paupert.
\newblock Census of the complex hyperbolic sporadic triangle groups.
\newblock {\em Exper. Math.}, 10:467--486, 2011.

\bibitem{DPP2}
M.~Deraux, J.~Parker, and J.~Paupert.
\newblock New non-arithmetic complex hyperbolic lattices.
\newblock Preprint, 2013.
\newblock {\em arXiv:1401.0308}

\bibitem{FalJDG}
E.~Falbel.
\newblock A spherical $\mbox{CR}$ structure on the complement of the figure
  eight knot with discrete holonomy.
\newblock {\em Journal Diff. Geom.}, 79(1):69--110, 2008.

\bibitem{FFP}
E.~Falbel, G.~Francsics, and J.~Parker.
\newblock The geometry of the gauss-picard modular group.
\newblock {\em Math. Ann}, 349(2):459--508, 2011.

\bibitem{FK1}
E.~Falbel and P.V. Koseleff.
\newblock Flexibility of ideal triangle groups in complex hyperbolic geometry.
\newblock {\em Topology}, 39:1209--1223, 2000.

\bibitem{FK2}
E.~Falbel and P.V. Koseleff.
\newblock A circle of modular groups in PU(2,1).
\newblock {\em Math. Res. Let.}, 9:379--391, 2002.

\bibitem{FK}
E.~Falbel and P.V. Koseleff.
\newblock Rigidity and flexibility of triangle groups in complex hyperbolic
  geometry.
\newblock {\em Topology}, 41, 2002.

\bibitem{FMS}
E.~Falbel, J.P. Marco, and F.~Schaffhauser.
\newblock Classifying triples of Lagrangians in a Hermitian vector space.
\newblock {\em Topology and its Applications}, 144, 2004.

\bibitem{FJP}
E.~Falbel and J.~Parker.
\newblock The moduli space of the modular group.
\newblock {\em Inv. Math.}, 152, 2003.

\bibitem{FJP2}
E.~Falbel and J.~Parker.
\newblock The geometry of the Eisenstein-Picard modular
  group.
\newblock {\em Duke Math. J.}, 131(2):249--289, 2006.

\bibitem{FalPla}
E.~Falbel and I.~Platis.
\newblock The PU(2,1) configuration space of four points in
  $S^3$ and the Cross-Ratio variety.
\newblock {\em Math. Ann.}, 340(4):935--962, 2008.

\bibitem{FWa}
E.~Falbel and J.~Wang.
\newblock Branched spherical CR-structures on the complement of the
  figure eight knot.
\newblock {\em Preprint}.

\bibitem{FW1}
E.~Falbel and R.~Wentworth.
\newblock Eigenvalues of products of unitary matrices and lagrangian
  involutions.
\newblock {\em Topology}, 45:65--99, 2006.

\bibitem{FW}
E.~Falbel and R.~Wentworth.
\newblock On products of isometries of hyperbolic space.
\newblock {\em Topo. Appl.}, 156(13):2257--2263, 2009.

\bibitem{FZ}
E.~Falbel and V.~Zocca.
\newblock A Poincar\'e polyhedron theorem for complex hyperbolic
  geometry.
\newblock {\em J. reine angew. Math.}, 516:133--158, 1999.

\bibitem{Fo}
N.~Pytheas Fogg.
\newblock {\em Substitutions in Dynamics, Arithmetics and Combinatorics}.
\newblock Springer, Berlin Heidelberg, 2002.

\bibitem{FrKl1}
R.~Fricke and F.~Klein.
\newblock {\em Vorlesungen \"uber die Theorie der automorphen Funktionen. Band
  I: Die gruppentheoretischen Grundlagen}.
\newblock Teubner, Leipzig, 1897.

\bibitem{FrKl2}
R.~Fricke and F.~Klein.
\newblock {\em Vorlesungen \"uber die theorie der Automorphen Funktionen. Band
  II: Die funktionentheoretischen Ausf\"uhrungen und die Anwendungen}.
\newblock Teubner, Leipzig, 1912.

\bibitem{Gaye}
M.~Gaye.
\newblock Sous-groupes discrets de PU(2,1) engendr\'es par $n$
  r\'eflexions complexes et D\'eformations.
\newblock {\em Geom. Ded.}, 137:27--61, 2008.

\bibitem{Go3}
W.~Goldman.
\newblock The symplectic nature of the fundamental groups of surfaces.
\newblock {\em Adv. in Math.}, 54:200--225, 1984.

\bibitem{Go}
W.~Goldman.
\newblock {\em Complex Hyperbolic Geometry}.
\newblock Oxford University Press, Oxford, 1999.

\bibitem{Gotraces}
W.~Goldman.
\newblock Trace coordinates on Fricke spaces of some simple hyperbolic
  surfaces.
\newblock In {\em Handbook of Teichm\"uller theory vol. II}, volume~13 of {\em
  IRMA Lect. Math. Theor. Phys.} Eur. Math. Soc., Z\"urich, 2009.

\bibitem{GKL}
W.~Goldman, M.~Kapovich, and B.~Leeb.
\newblock Complex hyperbolic manifolds homotopy equivalent to a
  $\mbox{R}$iemann surface.
\newblock {\em Comm. Anal. Math.}, 9:61--95, 2001.

\bibitem{GM}
W.~Goldman and J.~Millson.
\newblock Local rigidity of discrete groups acting on complex hyperbolic space.
\newblock {\em Invent. Math.}, 88:495--520, 1987.

\bibitem{GP}
W.~Goldman and J.~Parker.
\newblock Complex hyperbolic ideal triangle groups.
\newblock {\em Journal f\mbox{\"u}r dir reine und angewandte Math.},
  425:71--86, 1992.

\bibitem{GP2}
W.~Goldman and J.~Parker.
\newblock Dirichlet Polyhedra for Dihedral Groups
  Acting on Complex Hyperbolic Space.
\newblock {\em Jour. Geom. Analysis}, 2(6):517--554, 1992.

\bibitem{GongoPar}
K.~Gongopadhyay and S.~Parsad.
\newblock On two-generator subgroups of SU(3,1).
\newblock {\em Preprint}, 2013.

\bibitem{L.Green}
L.~Greenberg.
\newblock Discrete subgroups of the Lorentz group.
\newblock {\em Math. Scand.}, 10:85--107, 1962.

\bibitem{GroPS}
M.~Gromov and I.~Piatetski-Shapiro.
\newblock Nonarithmetic groups in Lobatchevsky spaces.
\newblock {\em Publ. I.H.E.S.}, 66:93--103, 1988.

\bibitem{GroSch}
M.~Gromov and R.~Schoen.
\newblock Harmonic maps into singular spaces and $p$-adic superrigidity for
  lattices in groups of rank one.
\newblock {\em Publ. I.H.E.S.}, 76:165--246, 1992.

\bibitem{GuH1}
N.~Gusevskii and H.~Cunha.
\newblock On the moduli space of quadruples of points in the boundary of
  complex hyperbolic space.
\newblock {\em Transformation Groups}, 15(2):261--283, 2010.

\bibitem{GuP}
N.~Gusevskii and J.R. Parker.
\newblock Representations of free Fuchsian groups in complex
  hyperbolic space.
\newblock {\em Topology}, 39:33--60, 2000.

\bibitem{GuP2}
N.~Gusevskii and J.R. Parker.
\newblock Complex Hyperbolic Quasi-Fuchsian groups
  and Toledo's invariant.
\newblock {\em Geom. Ded.}, 97:151--185, 2003.

\bibitem{HersPau}
S.~Hersonsky and F.~Paulin.
\newblock On the volume of complex hyperbolic manifolds.
\newblock {\em Duke Math. J.}, 84(3):719--737, 1996.

\bibitem{Hwang}
J.~Hwang.
\newblock On the volumes of complex hyperbolic manifolds with cusps.
\newblock {\em Internat. J. Math.}, 15(6):567--572, 2004.

\bibitem{JKP}
Y.~Jiang, S.~Kamiya, and J.~Parker.
\newblock J\o rgensen's inequality for Complex Hyperbolic
  Space.
\newblock {\em Geom. Ded.}, 97:55--80, 2003.

\bibitem{Jorg}
T.~J\"orgensen.
\newblock On discrete groups of M\"obius transformations.
\newblock {\em Amer. J. Math.}, 98:739--749, 1976.

\bibitem{KaTu}
Y.~Kamishima and T.Tsuboi.
\newblock CR-structures on Seifert manifolds.
\newblock {\em Invent. math.}, 104:149--163, 1991.

\bibitem{Kam1}
S.~Kamiya.
\newblock Notes on non-discrete subgroups of U(1,n;F).
\newblock {\em Hiroshima Math. J.}, 13:501--506, 1983.

\bibitem{Kam2}
S.~Kamiya.
\newblock Notes on elements of U(n,1;$\C$).
\newblock {\em Hiroshima Math. J.}, 21:23--45, 1991.

\bibitem{KPIV}
S.~Kamiya and J.~Parker.
\newblock Notes on discrete subgroups of PU(1,2;$\C$) with
  Heisenberg translations IV.
\newblock {\em Surikaisekikenkyusho Kokyuroku}, (1270):138--144, 2002.

\bibitem{KPII}
S.~Kamiya and J.~Parker.
\newblock On discrete subgroups of PU(1,2;$\C$) with
 Heisenberg translations II. Travaux de la
 Conf\'erence Internationale d'Analyse
  Complexe et du 9\`eme s\'eminaire
  Roumano-Finlandais (Bra\c sov, 2001).
\newblock {\em Rev. Roumaine Math. Pures Appl.}, 47(5--6):689--695, 2002.

\bibitem{KPT}
S.~Kamiya, J.~Parker, and J.~Thompson.
\newblock Notes on complex hyperbolic triangle groups.
\newblock {\em Conform. Geom. Dyn.}, 14:202--218, 2010.

\bibitem{Kapo}
M.~Kapovich.
\newblock {\em Hyperbolic Manifolds and Discrete Groups}.
\newblock Progress in Mathematics. Birkh\"auser, Boston, 2000.

\bibitem{KH}
V.~T. Khoi.
\newblock On the SU(2,1) Representation Space of the
  Brieskorn homology spheres.
\newblock {\em J. Math. Sci. Univ. Tokyo}, 14:499--520, 2007.

\bibitem{KimDa}
Daeyong Kim.
\newblock Discreteness criterions of isometric subgroups for quaternionic
  hyperbolic space.
\newblock {\em Geom. Dedicata}, 106:51--78, 2004.

\bibitem{KiKi}
I.~Kim and J.~Kim.
\newblock On the volumes of canonical cusps of complex hyperbolic manifolds.
\newblock {\em J. Korean Math. Soc.}, 46(3):513--521, 2009.

\bibitem{KimPark}
I.~Kim and J.~Parker.
\newblock Geometry of quaternionic hyperbolic manifolds.
\newblock {\em Math. Proc. Cambridge Philos. Soc.}, 135(2):291--320, 2003.

\bibitem{KimJoon}
Joonhyung Kim.
\newblock On the canonical cusps in complex hyperbolic surfaces.
\newblock {\em J. Korean Math. Soc.}, 49(2):343--356, 2012.

\bibitem{KR}
A.~Koranyi and H.M. Reimann.
\newblock The complex cross-ratio on the Heisenberg group.
\newblock {\em L'Enseign. Math.}, 33:291--300, 1987.

\bibitem{Kost}
V.~P. Kostov.
\newblock The Deligne-Simpson problem -- a survey.
\newblock {\em J. Algebra}, 281:83--108, 2004.

\bibitem{Kourou}
A.~Kourouniotis.
\newblock Complex length coordinates for quasi-Fuchsian groups.
\newblock {\em Mathematika}, 41:173--188, 1994.

\bibitem{KoMa}
V.~Koziarz and J.~Maubon.
\newblock Harmonic maps and representations of non-uniform lattices of
  PU($m$,1).
\newblock {\em Ann. Inst. Fourier}, 58(2):507--558, 2008.

\bibitem{Law}
S.~Lawton.
\newblock Generators, relations and symmetries in pairs of 3$\times$3
  unimodular matrices.
\newblock {\em J. Algebra}, 313:782--801, 2007.

\bibitem{Marg}
G.~Margulis.
\newblock Arithmeticity of the irreducible lattices in the semisimple groups of
  rank greater than 1.
\newblock {\em Invent. Math.}, 76:93--120, 1984.

\bibitem{Mil}
M.H. Millington.
\newblock Subgroups of the classical modular group.
\newblock {\em J. London Math. Soc.}, 1:351--357, 1969.

\bibitem{Most}
G.~D. Mostow.
\newblock A remarkable class of polyhedra in complex hyperbolic space.
\newblock {\em Pac. J. Math.}, 86:171--276, 1980.

\bibitem{Parkshimizu}
J.~Parker.
\newblock Shimizu's lemma for complex hyperbolic space.
\newblock {\em Internat. J. Math.}, 3, 1992.

\bibitem{ParFord}
J.~Parker.
\newblock On Ford isometric spheres in complex hyperbolic space.
\newblock {\em Proc. Math. Camb. Phil. Soc.}, 115:501--512, 1994.

\bibitem{ParkHeis}
J.~Parker.
\newblock Uniform discreteness and Heisenberg translations.
\newblock {\em Math. Z.}, 225, 1997.

\bibitem{ParkVol}
J.~Parker.
\newblock On the volume of cusped complex hyperbolic manifolds and orbifolds.
\newblock {\em Duke Math. J.}, 94:433--464, 1998.

\bibitem{P3}
J.~Parker.
\newblock Unfaithful complex hyperbolic triangle groups I :
  Involutions.
\newblock {\em Pacific J. Math.}, 238:145--169, 2008.

\bibitem{Parklattices}
J.~Parker.
\newblock Complex hyperbolic lattices.
\newblock In {\em Discrete Groups and Geometric Structures}, volume 501 of {\em
  Contemporary Mathematics}, pages 1--42. AMS, 2009.

\bibitem{Parktraces}
J.~Parker.
\newblock Traces in complex hyperbolic geometry.
\newblock In {\em Geometry, Topology and Dynamics of Character Varieties},
  volume~23 of {\em Lecture Notes Series, National University of Singapore},
  pages 191--245. World Scientific, 2010.

\bibitem{ParkPau}
J.~Parker and J.~Paupert.
\newblock Unfaithful complex hyperbolic triangle groups II :
  Higher order reflections.
\newblock {\em Pacific J. Math.}, 239:357--389, 2009.

\bibitem{PP1}
J.~Parker and I.~Platis.
\newblock Open sets of maximal dimension in complex hyperbolic quasi-fuchsian
  space.
\newblock {\em J. Diff. Geom}, 73:319--350, 2006.

\bibitem{PP}
J.~Parker and I.~Platis.
\newblock Complex hyperbolic Fenchel-Nelsen coordinates.
\newblock {\em Topology}, 47:101--135, 2008.

\bibitem{PP3}
J.~Parker and I.~Platis.
\newblock Global, geometrical coordinates on Falbel's cross-ratio
  variety.
\newblock {\em Can. Math. Bull.}, 52:285--294, 2009.

\bibitem{PP2}
J.~Parker and I.~Platis.
\newblock Complex hyperbolic quasi-fuchsian groups.
\newblock In {\em Geometry of Riemann Surfaces}, volume 368 of {\em London
  Mathematical Society Lecture Notes}, pages 309--355. 2010.

\bibitem{ParkWi}
J.~Parker and P.~Will.
\newblock Complex hyperbolic free groups with many parabolic elements.
\newblock {\em arXiv:1312.3795}.

\bibitem{ParkWi2}
J.~Parker and P.~Will.
\newblock A complex hyperbolic Riley slice.
\newblock {\em In preparation}.

\bibitem{Parbook}
J.R. Parker.
\newblock {\em Complex Hyperbolic Kleinian Groups, to appear}.
\newblock Cambridge University Press.

\bibitem{PopR}
J.~Paupert.
\newblock Elliptic triangle groups in PU(2,1), Lagrangian
  triples and momentum maps.
\newblock {\em Topology}, 46(2):155--183, 2007.

\bibitem{PaW}
J.~Paupert and P.~Will.
\newblock Real reflections, commutators and cross-ratios in complex hyperbolic
  space.
\newblock {\em Preprint arXiv:1312.3173}, 2013.

\bibitem{Penner}
R.~Penner.
\newblock The decorated Teichm\"uller space for punctured surfaces.
\newblock {\em Comm. Math. Phys.}, 113(2):299--339, 1987.

\bibitem{Phillips}
Mark~B. Phillips.
\newblock Dirichlet polyhedra for cyclic groups in complex hyperbolic space.
\newblock {\em Trans. Amer. Math. Soc}, 115:221--228, 1992.

\bibitem{Pra}
A.~Pratoussevitch.
\newblock Traces in complex hyperbolic triangle groups.
\newblock {\em Geometriae Dedicata}, 111:159--185, 2005.

\bibitem{Pra1}
A.~Pratoussevitch.
\newblock Non-discrete complex hyperbolic triangle groups of type $\left(
  m,m,\infty\right)$.
\newblock {\em Bull. of the LMS}, 43:359--363, 2011.

\bibitem{Probook}
C.~Procesi.
\newblock {\em Lie groups. An approach through invariants and representations}.
\newblock Springer, New-York, 2007.

\bibitem{Rag}
M.S. Ragunathan.
\newblock {\em Discrete subgroups of Lie Groups}.
\newblock Springer, Berlin, 1972.

\bibitem{Sa}
H.~Sandler.
\newblock Traces on SU(2,1) and complex hyperbolic ideal triangle
  groups.
\newblock {\em Algebras groups and geometries}, 12:139--156, 1995.

\bibitem{S1}
R.~E. Schwartz.
\newblock Degenerating the complex hyperbolic ideal triangle groups.
\newblock {\em Acta Math.}, 186:105--154, 2001.

\bibitem{S}
R.~E. Schwartz.
\newblock Ideal triangle groups, dented tori, and numerical analysis.
\newblock {\em Ann. of Math. (2)}, 153:533--598, 2001.

\bibitem{S2}
R.~E. Schwartz.
\newblock Complex hyperbolic triangle groups.
\newblock {\em Proc. Int. Math. Cong.}, 1:339--350, 2002.

\bibitem{S4}
R.~E. Schwartz.
\newblock Real hyperbolic on the outside, complex hyperbolic on the inside.
\newblock {\em Inv. Math.}, 151(2):221--295, 2003.

\bibitem{S3}
R.~E. Schwartz.
\newblock A better proof of the Goldman-Parker
  conjecture.
\newblock {\em Geometry and Topology}, 9, 2005.

\bibitem{Schbook}
R.~E. Schwartz.
\newblock {\em SphericalCR Geometry and Dehn
  Surgery}, volume 165 of {\em Ann. of Math. Studies}.
\newblock Princeton University Press, Princeton, 2007.

\bibitem{Simp}
C.~Simpson.
\newblock Products of matrices.
\newblock In {\em Differential geometry, global analysis, and topology
  (Halifax, NS, 1990)}, number~12 in CMS Conf. Proc., pages 157--185.
  PROVIDENCE, RI, Amer. Math. Soc edition, 1991.

\bibitem{Tan}
S.P. Tan.
\newblock Complex Fenchel-Nielsen coordinates for
  quasi-Fuchsian structures.
\newblock {\em Internat. J. Math}, 5:239--251, 1994.

\bibitem{Tol}
D.~Toledo.
\newblock Representations of surface groups in complex hyperbolic space.
\newblock {\em J. Differ. Geom.}, 29:125--133, 1989.

\bibitem{Vo}
H.~Vogt.
\newblock Sur les invariants fondamentaux des \'equations
  diff\'erentielles lin\'eaires du second ordre.
\newblock {\em Ann. Sci. E. N. S. $\mbox{3}^{\mbox{$\mbox{\`e}$me}}$
  s\mbox{\'e}rie}, 1886.

\bibitem{Wen}
Z.X. Wen.
\newblock Relations polynomiales entre les traces de produits de matrices.
\newblock {\em C.R. Acad. Sci Paris}, 314:99--104, 1994.

\bibitem{Wi3}
P.~Will.
\newblock {\em Groupes libres, groupes triangulaires et tore \'epoint\'e dans
  PU(2,1)}.
\newblock Th\`ese de l'universit\'e Paris VI, 2006.

\bibitem{Wi2}
P.~Will.
\newblock The punctured torus and Lagrangian triangle groups in
  PU(2,1).
\newblock {\em J. reine angew. Math.}, 602:95--121, 2007.

\bibitem{Wi6}
P.~Will.
\newblock Groupes triangulaires lagrangiens en g\'eom\'etrie hyperbolique
  complexe.
\newblock {\em Actes du S\'eminaire Th\'eorie Spectrale et G\'eom\'etrie},
  25:189--209, 2008.

\bibitem{Wi4}
P.~Will.
\newblock Traces, Cross-ratios and 2-Generator
 Subgroups of PU(2,1).
\newblock {\em Canad. J. Math.}, 61(6):189--209, 2009.

\bibitem{Wi7}
P.~Will.
\newblock Bending Fuchsian representations of fundamental groups of
  cusped surfaces PU(2,1).
\newblock {\em J. Differential. Geom.}, 90(3):473--520, 2012.

\bibitem{WyG}
J.~Wyss-Galifent.
\newblock {\em Discreteness and Indiscreteness Results for Complex Hyperbolic
  Triangle Groups}.
\newblock 2000.

\bibitem{XJ}
B-H. Xie and Y-P. Jiang.
\newblock Discreteness of subgroups of PU(2,1) with regular elliptic
  elements.
\newblock {\em Linear Algebra Appl.}, 428(4), 2008.

\bibitem{XJW}
B-H. Xie, Y-P. Jiang, and H.~Wang.
\newblock Discreteness of subgroups ofPU(1,$n$,$\C$).
\newblock {\em Proc. Japan Acad. Ser. A Math. Sci.}, 84(6), 2008.

\end{thebibliography}
\end{document}